\documentclass[10pt]{article}
\usepackage[a4paper]{geometry}
\usepackage[dvipsnames]{xcolor}
\usepackage{mathtools}
\usepackage{fancyhdr}
\usepackage{amssymb}
\usepackage{amsthm}
\numberwithin{equation}{section}
\usepackage{tikz}
\usepackage{tikz-cd}
\usetikzlibrary{matrix, positioning}
\usepackage{enumerate}
\theoremstyle{definition}
\newtheorem{definition}{Definition}[section]
\newtheorem{example}[definition]{Example}
\theoremstyle{plain}
\newtheorem{lemma}[definition]{Lemma}
\newtheorem{proposition}[definition]{Proposition}
\newtheorem{corollary}[definition]{Corollary}

\newtheorem{Conjecture}[definition]{Conjecture}
\newtheorem{theorem}[definition]{Theorem}
\theoremstyle{remark}
\newtheorem{remark}[definition]{Remark}
\usepackage{epigraph}
  \usepackage{etoolbox}
    \makeatletter
    \newlength\epitextskip
    \pretocmd{\@epitext}{\em}{}{}
    \apptocmd{\@epitext}{\em}{}{}
    \patchcmd{\epigraph}{\@epitext{#1}\\}{\@epitext{#1}\\[\epitextskip]}{}{}
    \makeatother

\usepackage{titlesec}
\usepackage{appendix}
\usepackage{framed}
\usepackage{graphicx}
\usepackage{afterpage}
\usepackage{booktabs}
\usepackage{placeins}
\usepackage{array}
\usepackage{hyperref}
\hypersetup{
    colorlinks=true,
    linkcolor=black, 
    citecolor=black,
    urlcolor=blue,
    pdftitle={A Markov model for factorisation of iterated cubic polynomials},
    pdfauthor={Javier San Martín Martínez}
}

\title{A Markov model for factorisation of iterated cubic polynomials}
\author{Javier San Martín Martínez\\ Universit\`a di Pisa, Dipartimento di Matematica, Pisa, Italy \\j.sanmartinmartin@studenti.unipi.it
}
\date{November 2026}

\begin{document}

\maketitle
\begin{abstract}
	We propose a Markov model for the factorisation of the iterates of \emph{post-critically finite} (\emph{PCF}) cubic polynomials, extending work of Jones and Boston and of Goksel on quadratic polynomials. The model is based on the \emph{combined critical orbit}. For most of the PCF cubic polynomials over $\mathbb{Q}$ classified by Anderson, Manes and Tobin, whose combined critical orbits have length one, two or three, we construct groups $M_n$ whose cycle data follow the model. For length one, $M_n$ is the arboreal Galois group of the cubic Belyi map computed by Benedetto et al., which serves as a consistency check of the model; the groups for lengths two and three are new. In some cases the co-critical points, or the critical orbit itself, force relations between the types, and we refine the model and the groups accordingly. We describe all these groups by sign conditions coming from discriminants, and we use this description to prove that they contain, up to conjugation, the Galois groups of the iterates of $f$ and of its affine conjugates $f_t$. Consequently, equality with the group of the model of $f_t$ is equivalent to the factorisation of the iterates of $f_t$ following this model. For each fixed level, we conjecture that equality holds outside a thin set of parameters (possibly depending on the level), and we give computational evidence.

\end{abstract}
\section*{Introduction}
In this introduction, $K$ is a field and $f \in K[x]$ is a polynomial of degree $d\geq 2$ all of whose iterates $f^n=f\circ\dots\circ f$ ($n$ times) are separable. The roots of the iterates form a regular $d$-ary rooted tree $\mathcal{T}$: the vertices of level $n\geq 0$ are the roots of $f^n$ (where $f^0(x)=x$), and a root $\alpha$ of $f^n$ is joined to the root $f(\alpha)$ of $f^{n-1}$. Since Galois automorphisms commute with $f$, the profinite group $\mathrm{Gal}(f^{\infty}):= \varprojlim_n \mathrm{Gal}(f^n)$ acts on $\mathcal{T}$ by tree automorphisms. Identifying $\mathcal{T}$ with the standard $d$-ary tree (see Subsection~\ref{1.1}), we obtain the \emph{arboreal representation} $\mathrm{Gal}(f^{\infty})\hookrightarrow \mathrm{Aut}(\mathcal{T})$.

This way of looking at Galois groups has proved very useful, especially in the study of the prime divisors of sequences of the form $x_n = f(x_{n-1})$, where $f \in K[x]$. This was Odoni's original motivation in~\cite{https://doi.org/10.1112/plms/s3-51.3.385,Odoni1985OnTP,https://doi.org/10.1112/S002557930000632X}. 

Galois groups arising in this way are generally difficult to describe. Therefore, we restrict our study to post-critically finite polynomials. A polynomial $f$ is \emph{post-critically finite} (\emph{PCF}) if the forward orbit $\{f^n(\alpha)\}_{n\geq 1}$ of every critical point $\alpha$ is finite. For such $f$ over a global field of characteristic $0$ or $>d$, the group $\mathrm{Gal}(f^{\infty})$ has infinite index in $\mathrm{Aut}(\mathcal{T})$~\cite[Theorem 3.1]{JonesSurvey}; for the cubic polynomials studied here, this also follows from Theorem~\ref{thm:intro} below and the Hausdorff dimensions of Proposition~\ref{prop:hdim}. The Galois groups arising from these polynomials have been studied in some cases: polynomials of degree 2~\cite{AHMAD2022842,Pink2013}; degree 3~\cite{Benedetto2017,HlushchankaLukinaWardell2025}; some families of arbitrary degrees, such as Belyi maps~\cite{Bouw2018DynamicalBM}; quadratic rational maps~\cite{Ejder2024}; and maps with colliding critical orbits~\cite{benedetto2024arborealgaloisgroupscubic,Benedetto2024}.

Recently, Goksel proposed an indirect approach to these Galois groups~\cite{article}. His idea is to use the Chebotarev density theorem in the ``opposite'' direction: instead of deducing the factorisation of $f^n$ modulo primes from the Galois group, one models the factorisation and then constructs groups whose cycle data occur with the predicted frequencies. For quadratic polynomials, the factorisation of the iterates over finite fields had been modelled by a Markov process by Jones and Boston~\cite{citation-key}, and this model was refined by Goksel, Xia and Boston~\cite{goksel2013refined}. For two of the critical orbit types whose computations deviate from the Jones--Boston model, Goksel later proved that certain transitions of the model never occur, using relations between the quadratic characters of the values at the post-critical points~\cite{goksel2026note}; this explains the multi-step Markov model proposed in~\cite{goksel2013refined}. Section~\ref{sec:cocritical} establishes restrictions of the same kind for cubic polynomials. For quadratic polynomials, Goksel conjectured that the groups obtained from the Markov model contain the Galois groups $\mathrm{Gal}(f^n)$, and he observed that this could be proved by a purely group-theoretic argument in the automorphism group $\mathrm{Aut}(\mathcal{T}_n)$ of the first $n$ levels of the binary tree~\cite{article}: every element of $\mathrm{Gal}(f^n)$ can be conjugated into the Markov group by an element of the kernel of $\mathrm{Aut}(\mathcal{T}_n)\to\mathrm{Aut}(\mathcal{T}_{n-1})$, and the question is whether the whole group can be conjugated into it by a single such element. In~\cite{goksel2023localglobal} he proved the first results on this local-global question, and he conjectured a positive answer for every subgroup of $\mathrm{Aut}(\mathcal{T}_n)$ that contains an element acting transitively on level $n$. Wardell has recently given a counterexample to this conjecture~\cite{Wardell2026LocalglobalCQ}. It is not known whether his groups occur as Galois groups and Markov groups, so this does not decide the conjecture on the Markov groups, but a proof of it has to use further properties of these groups. For this reason, we do not go through a local-global statement for cubic polynomials. Instead, we describe our groups by sign conditions on the levels of the tree, and we prove the inclusion by an arithmetic argument with the discriminants of the iterates (Section~\ref{5}).

As in~\cite{article}, we consider $f$ together with its affine conjugates. For $t\in K$ we put
\begin{equation}\label{eq:conjugate}
    f_t(x):=f(x+t)-t,\qquad\text{so that}\qquad f_t^n(x)=f^n(x+t)-t\quad\text{for every }n\geq 1,
\end{equation}
and $f_0=f$. (In~\cite{article} the parameter is called $a$.) The map $x\mapsto x+t$ sends the roots of $f_t^n$ onto the solutions of $f^n(x)=t$, it commutes with the Galois action, and it respects the tree structure. Hence $f_t^n$ and $f^n-t$ have the same splitting field, and studying the conjugate $f_t$ is the same as studying the preimages of the point $t$ under $f$, as in~\cite{Benedetto2017}. The other affine conjugates give nothing new: if $\varphi(x)=cx+t$ with $c\neq 0$, then the roots of the $n$-th iterate of $\varphi^{-1}\circ f\circ\varphi$ are the images under $\varphi^{-1}$ of the solutions of $f^n(x)=t$, so this polynomial has the same arboreal representation as $f_t$. If $f$ is PCF, then so is $f_t$, and all the iterates of $f_t$ are separable if and only if $t$ does not lie in the forward orbit of a critical point of $f$. So the conjugates are needed: for instance, the iterates of $-2z^3+3z^2$ itself are not separable, because its critical point $0$ is fixed.

We introduce a Markov model for the factorisation of cubic polynomials and use it as a framework for constructing groups that follow the model. The main results of the paper can be summarised as follows:
\begin{itemize}
    \item We define a Markov model for the factorisation of iterates $f^n$ of PCF cubic polynomials (Section~\ref{2}).
    \item For combined critical orbits of lengths one and two, we construct groups $M_n\subseteq \mathrm{Aut}(\mathcal{T}_n)$ whose cycle data follow the model (Sections~\ref{3} and \ref{4}). For length one, $M_n$ coincides with the arboreal Galois group $E_n$ of the Belyi map computed in~\cite{Benedetto2017,Bouw2018DynamicalBM}, so Section~\ref{3} is a consistency check of the model; the groups of Section~\ref{4} are new.
    \item We show that some admissible types never occur when the co-critical points of $f$ lie on the combined critical orbit. For length two we refine the model (Section~\ref{sec:cocritical}) and construct the corresponding groups $\widetilde{M}_n$ (Section~\ref{sec:refined-two}).
    \item For the polynomial of length three, $-2z^3+3z^2+\frac{1}{2}$, the post-critical set is a single cycle of length three. This forces a relation between the letters of every type (Subsection~\ref{sec:three-model}), and we construct the groups of the corresponding model (Section~\ref{sec:three}).
    \item We describe all these groups by sign conditions coming from discriminants, we prove that the Galois groups of the iterates of $f$ and of its conjugates are contained in them, and we deduce that equality with the group of the model of $f_t$ is equivalent to the factorisation of $f_t^n$ following this model (Section~\ref{5}).
    \item We conjecture equality for every fixed level outside a thin subset of parameters, which may depend on the level (Conjecture~\ref{conj:main}), and give computational evidence for individual $t$ (Appendix~\ref{app:computations}).
\end{itemize}
More precisely, our main result is the following (see Theorems~\ref{thm:Markov-discriminant}, \ref{thm:upper-bound} and \ref{thm:refined-upper-bound}).
\begin{theorem}\label{thm:intro}
    Let $F$ be a number field containing $\sqrt{-3}$, and let $f\in F[x]$ be a PCF cubic polynomial belonging to one of the four classes in Definition~\ref{def:model-of-f}. Denote by $M_n(f)$ the corresponding group of Section~\ref{3}, \ref{4}, \ref{sec:three} or \ref{sec:refined-two}; for the co-critical case of length two, we take the refined group $\widetilde{M}_n$. Then for every $n\geq 1$ and $t\in F$ such that $f_t^n$ is separable, $\mathrm{Gal}(f_t^n/F)$ is conjugate in $\mathrm{Aut}(\mathcal{T}_n)$ to a subgroup of $M_n(f)$.
\end{theorem}
The inclusion is arithmetic: its proof does not rely on the Markov model or on the Chebotarev density theorem. We show that $M_n$ is cut out of $\mathrm{Aut}(\mathcal{T}_n)$ by conditions on the signs of the permutations induced on the levels of the tree, and that these conditions are forced by discriminant identities which hold because $f$ is PCF. The model has two parts. The transitions of the Markov process are the same for $f$ and for all its conjugates, while the initial data and the group $M_{n,t}$, a subgroup of the Markov group of $f$, depend on the conjugate $f_t$: through the factorisation of $f_t$ over $F$, and through which of the values $(f^j(\gamma_1)-t)(f^j(\gamma_2)-t)$, with $\gamma_1,\gamma_2$ the critical points of $f$, are squares in $F$. We call this the \emph{model of $f_t$}, and for almost all $t$ the group $M_{n,t}$ is the whole Markov group of $f$. The cycle data of these groups are the frequencies of the models, so the inclusion and the Chebotarev density theorem show that $\mathrm{Gal}(f_t^n/F)$ is conjugate to $M_{n,t}$ if and only if the factorisation of $f_t^n$ follows the model of $f_t$ (Theorem~\ref{thm:follows-model}). For the Belyi map this is the case for suitable values of $t$ (Theorem~\ref{thm:belyi}). We conjecture that, at each fixed level, it is the case for all the polynomials studied here outside a thin subset of the parameters, which may depend on the level (Conjecture~\ref{conj:main}). For most of the values of $t$ that we tested, the computations of Appendix~\ref{app:computations} prove equality for $n\leq 3$, and the factorisation agrees with the model of $f_t$ up to $n=5$.

We use the following definition to classify PCF cubic polynomials.
\begin{definition}\label{comb}
    Let $f \in K[x]$ be a PCF cubic polynomial with critical points
$\gamma_1, \gamma_2$ (we allow $\gamma_1 = \gamma_2$). For $k \geq 1$, set
\[
  P_k := \{ f^k(\gamma_1), f^k(\gamma_2) \},
\]
regarded as an unordered set (so $P_k$ has one element if
$f^k(\gamma_1) = f^k(\gamma_2)$). Since $P_{k+1} = f(P_k)$ and $f$ is PCF,
the sequence $(P_k)_{k \geq 1}$ takes only finitely many values and is
eventually periodic. Let $m \geq 1$ be the smallest integer such that
\[
  P_{m+1} = P_{l} \quad \text{for some } 1 \leq l \leq m,
\]
and denote this (unique) $l$ by $l_f$. The \emph{combined critical orbit}
of $f$ is the sequence $(P_1, \dots, P_m)$, and $m$ is its \emph{length}.
\end{definition}

In particular, we study the PCF polynomials over $\mathbb{Q}$ given by Anderson et al. in~\cite{cubic-PCF}, which are listed up to conjugation, together with their affine conjugates. Of the $15$ polynomials in their list, two are unicritical and two have colliding critical orbits; the other $11$ have combined critical orbits of length one, two or three. We construct groups for all of them, except for two polynomials of length one that we do not consider (see the beginning of Section~\ref{3}).

We recall some background in Section~\ref{1}. In Subsection~\ref{1.1}, we provide the basic definitions of automorphisms of the regular rooted tree. In Subsection~\ref{1.2}, we recall how, for a number field $K$ and $f \in K[x]$, the factorisation of $f^n$ modulo the primes $\mathfrak{p}$ of $\mathcal{O}_K$ is related to the Galois group $\mathrm{Gal}(f^n)$ by the Chebotarev density theorem.

In Section~\ref{2}, we define a Markov model for the factorisation of iterates of cubic polynomials, and in Section~\ref{sec:cocritical} we study the relations between the letters of a type and refine the model. In Sections~\ref{3}, \ref{4} and \ref{sec:three}, we construct, by explicit generators, groups which follow the Markov model for polynomials with combined critical orbits of lengths one, two and three, and in Section~\ref{sec:refined-two} we construct the refined groups for the co-critical polynomials of length two. Only in Section~\ref{5} do we turn to Galois groups: we describe all these groups by sign conditions, we compute their orders and Hausdorff dimensions (Proposition~\ref{prop:hdim}), we prove that they contain the Galois groups $\mathrm{Gal}(f_t^n)$ up to conjugation, we treat the Belyi map, where equality holds, and we state our conjecture. In Appendix~\ref{app:computations} we summarise the computations for individual values of $t$; the results for the individual pairs are given in Appendix~\ref{app:tables}.

 \section*{Acknowledgements}
This paper grew out of my bachelor's thesis at Utrecht University, supervised by Dr. Valentijn Karemaker. I would like to thank her for her guidance and support; she was essential in selecting the problem and the methods used, and I am grateful for her careful review of the text and for her corrections and stylistic suggestions. I also thank the anonymous reviewer for the numerous suggestions on a previous submission of this manuscript.
\subsection*{AI use statement}
During the last revision of this manuscript, AI was used to rephrase the text and to correct English mistakes. It also produced the code for the computations in the appendix and most of the tables. Finally, it suggested simplifications of some proofs and found references containing those results. All results in the article are the sole responsibility of the author.
 \tableofcontents

 \section{Background}\label{1}
 \subsection{Background on tree automorphisms}\label{1.1}
Let $\mathcal{T}$ be the regular ternary rooted tree: every vertex has three children, and every vertex other than the root has exactly one parent.
We recall some standard notation from the theory of groups acting on rooted trees.
Let $X = \{0,1,2\}$. The set $X^{\ast}$ of finite words over $X$ is a monoid under concatenation, which we write as juxtaposition $uw$, and we denote by $X^n$ the set of words of length $n$.
We identify the vertices of $\mathcal{T}$ with the elements of $X^{\ast}$: the root is the empty word, and the three children of a vertex $w$ are, from left to right, $w0$, $w1$ and $w2$. Thus the vertices of level $n$ are the words of length $n$. For $n\geq 0$ we denote by $\mathcal{T}_n$ the finite subtree formed by the vertices of level at most $n$.
We depict the first two levels of the tree below.

\[\begin{tikzcd}
	&& {} && \bullet \\
	& 0 &&& 1 &&& 2 \\
	00 & 01 & 02 & 10 & 11 & 12 & 20 & 21 & 22
	\arrow[from=1-5, to=2-2]
	\arrow[from=1-5, to=2-5]
	\arrow[from=1-5, to=2-8]
	\arrow[from=2-2, to=3-1]
	\arrow[from=2-2, to=3-2]
	\arrow[from=2-2, to=3-3]
	\arrow[from=2-5, to=3-4]
	\arrow[from=2-5, to=3-5]
	\arrow[from=2-5, to=3-6]
	\arrow[from=2-8, to=3-7]
	\arrow[from=2-8, to=3-8]
	\arrow[from=2-8, to=3-9]
\end{tikzcd}\]

Automorphisms of $\mathcal{T}$ and of $\mathcal{T}_n$ act on the right: we write $(w)\sigma$ for the image of a vertex $w$ under $\sigma$, and $\sigma\tau$ means ``first $\sigma$, then $\tau$''.
 \begin{definition}
     Let $\sigma \in \mathrm{Aut}(\mathcal{T})$ and $u\in X^{*}$. Since $\sigma$ maps the subtree below $u$ onto the subtree below $(u)\sigma$, there is a unique $\sigma_u\in\mathrm{Aut}(\mathcal{T})$ such that
     \begin{equation*}
         (uw)\sigma=(u)\sigma\, (w)\sigma_u\qquad\text{for all } w\in X^{*}.
     \end{equation*}
     We call $\sigma_u$ the \emph{section} of $\sigma$ at $u$. The same definition applies to $\sigma\in\mathrm{Aut}(\mathcal{T}_n)$, and then $\sigma_u\in\mathrm{Aut}(\mathcal{T}_{n-k})$ if $u$ has length $k$.
 \end{definition}
 The sections at the first level give an isomorphism
 \begin{equation*}
     \mathrm{Aut}(\mathcal{T})\cong \bigl(\mathrm{Aut}(\mathcal{T})\times \mathrm{Aut}(\mathcal{T})\times \mathrm{Aut}(\mathcal{T})\bigr)\rtimes S_3,\qquad \sigma \mapsto (\sigma_0,\sigma_1,\sigma_2)\pi,
 \end{equation*}
 where $\pi\in S_3$ is the permutation induced by $\sigma$ on the first level $X$, which we call the \emph{top permutation} of $\sigma$, and $S_3$ acts by permuting the three factors. Concretely, $(a_0,a_1,a_2)\pi$ acts by $(iw)\mapsto (i)\pi\,(w)a_i$, and therefore
 \begin{equation}\label{eq:multiplication}
     (a_0,a_1,a_2)\pi\cdot(b_0,b_1,b_2)\rho=\bigl(a_0b_{(0)\pi},\,a_1b_{(1)\pi},\,a_2b_{(2)\pi}\bigr)\pi\rho .
 \end{equation}
 We refer to \eqref{eq:multiplication} as the \emph{multiplication rule}. For instance, if $\sigma=(a_0,a_1,a_2)$ has trivial top permutation and $\tau=(b_0,b_1,b_2)\rho$, then the conjugate $\sigma^\tau:=\tau^{-1}\sigma\tau$ has trivial top permutation and its section at the vertex $(i)\rho$ is $b_i^{-1}a_ib_i$: conjugation by $\tau$ moves the sections along $\rho$. In the same way $\mathrm{Aut}(\mathcal{T}_n)\cong\mathrm{Aut}(\mathcal{T}_{n-1})^3\rtimes S_3$. Finally, for a group $G$ of automorphisms of $\mathcal{T}$ or $\mathcal{T}_n$ we write $\mathrm{St}_G(k)$ for the subgroup of elements of $G$ that fix every vertex of level $k$, and $\mathrm{St}_1$ for the subgroup of elements of $\mathrm{Aut}(\mathcal{T}_n)$ with trivial top permutation.
\subsection{Galois groups and factorisation of polynomials}\label{1.2}
Goksel's idea~\cite{article}, which we adapt here, is to predict the Galois groups of the iterates of a polynomial from the way the iterates factor modulo primes. The bridge between the two is the Chebotarev density theorem, which relates the frequencies of the degrees of the irreducible factors of a polynomial modulo $p$ to the frequencies of the cycle types of the elements of its Galois group.

Let $f \in \mathbb{Q}[x]$ be separable of degree $d$. The Galois group $\mathrm{Gal}_{\mathbb{Q}}(f)$ of its splitting field acts on the roots of $f$, and we view it as a subgroup of $S_d$. Following~\cite{article}, the \emph{cycle type} $\mathrm{c}(\sigma)$ of $\sigma \in \mathrm{Gal}_{\mathbb{Q}}(f)$ is the partition of $d$ formed by the lengths of the cycles of $\sigma$.

\begin{example}
The permutation $\sigma=(1,2,3)(4)\in S_4$ has cycle type $\mathrm{c}(\sigma)=(3,1)$.
\end{example}

Similarly, the cycle type $\mathrm{c}(g)$ of a separable polynomial $g$ over a field is the partition of $\deg g$ formed by the degrees of its irreducible factors.

\begin{example}
For $g=(x^2+1)(x-2)\in\mathbb{Q}[x]$, we have $\mathrm{c}(g)=(2,1)$.
\end{example}

For all but finitely many primes $p$, the reduction of $f$ modulo $p$ is separable of degree $d$, and its cycle type is the cycle type of a Frobenius element at $p$. Hence the Chebotarev density theorem states that, for every partition $\lambda$ of $d$, the following limit exists and
\[
\lim_{X \to \infty} \frac{\#\{p \leq X\ \text{prime} : \mathrm{c}(f \bmod p)=\lambda\}}{\#\{p \leq X\ \text{prime}\}} = \frac{\#\{\sigma \in \mathrm{Gal}_{\mathbb{Q}}(f) : \mathrm{c}(\sigma)=\lambda\}}{|\mathrm{Gal}_{\mathbb{Q}}(f)|},
\]
where in the numerator we only count the primes of good reduction.

\begin{example}\label{ex:chebotarev-cubic}
Consider the polynomial $f=x^3+x+1\in\mathbb{Q}[x]$. It is irreducible over
$\mathbb{Q}$ and has discriminant $\Delta(f)=-31$, which is not a square.
Consequently, $\mathrm{Gal}_{\mathbb{Q}}(f)\cong S_3$. As $S_3$ consists of
one identity, two $3$-cycles, and three transpositions, the Chebotarev
density theorem gives the following natural densities of primes:
\begin{enumerate}
    \item $\frac{1}{6}$ of the primes: $f\bmod p$ splits into three linear
          factors, so $\mathrm{c}(f\bmod p)=(1,1,1)$.
    \item $\frac{1}{3}$ of the primes: $f\bmod p$ is irreducible,
          so $\mathrm{c}(f\bmod p)=(3)$.
    \item $\frac{1}{2}$ of the primes: $f\bmod p$ has one irreducible
          quadratic factor and one linear factor, so
          $\mathrm{c}(f\bmod p)=(2,1)$.
\end{enumerate}
The finitely many primes of bad reduction do not affect these densities.
\end{example}

We aim to provide a probabilistic model for the factorisation of $f^n$: given a randomly chosen prime $\mathfrak{p}$, the model predicts the probability that $f^n$ factors in a given way modulo $\mathfrak{p}$. The model only uses arithmetic information attached to the critical points of $f$, namely the values of $f$ along its combined critical orbit (Definition~\ref{comb}), and it computes the probabilities recursively, as a Markov process. We then construct groups whose cycle data occur with exactly the frequencies predicted by the model.

By the Chebotarev density theorem, if the Galois group of $f_t^n$ is equal to the group we construct for $f_t$, then the factorisation of $f_t^n$ follows the model of $f_t$. However, the frequencies of the cycle data do not determine a group up to conjugation, so the relation between our groups and the Galois groups requires a separate argument. In Section~\ref{5} we prove that the Galois groups are always contained in our groups up to conjugation (Theorem~\ref{thm:upper-bound}), and that, given this inclusion, equality of the groups is equivalent to the factorisation following the model of $f_t$ (Theorem~\ref{thm:follows-model}).

 \section{A Markov model for cubic polynomials}\label{2}
In this section we introduce a Markov model for the factorisation of iterates
of cubic polynomials. The key input is a description of how the composition
$g \circ f$ factors when $g$ is irreducible: it is governed by the
factorisation of $f - \beta$ over $K(\beta)$, where $\beta$ is a root of $g$.
This will let us predict the factorisation of $f^n$ from that of $f^{n-1}$.

\begin{lemma}\label{lem:composition}
Let $K$ be a finite field, let $f, g \in K[x]$ with $\deg f \geq 1$, and
assume $g$ is irreducible. Let $\beta \in \overline{K}$ be a root of $g$,
put $L := K(\beta)$, and assume that $f - \beta$ is separable. If
$f - \beta$ is a product over $L$ of irreducible factors of degrees
$e_1, \dots, e_r$, then $g \circ f$ is a product over $K$ of irreducible
factors of degrees $e_1 \deg g, \dots, e_r \deg g$.
\end{lemma}

\begin{proof}
Let $\varphi$ be the Frobenius of $K$, that is, $y\mapsto y^{|K|}$, and let $D=\deg g$. Then $\varphi$ permutes
the roots $\beta,\varphi(\beta),\dots,\varphi^{D-1}(\beta)$ of $g$ cyclically,
and $\varphi^{D}$ is the Frobenius of $L$. The roots of $g\circ f$ are the
roots of the polynomials $f-\varphi^{i}(\beta)$, which are separable, and
$\varphi$ maps the roots of $f-\varphi^{i}(\beta)$ onto those of
$f-\varphi^{i+1}(\beta)$. Hence every orbit of $\varphi$ on the roots of
$g\circ f$ meets the roots of $f-\beta$ in one orbit of $\varphi^{D}$, and it
is $D$ times as long. The orbits of $\varphi$ correspond to the irreducible
factors of $g\circ f$ over $K$ and the orbits of $\varphi^{D}$ to those of
$f-\beta$ over $L$, with lengths equal to the degrees.
\end{proof}

When $f$ is a cubic, the square class of the discriminant of $f-\beta$ over a finite field restricts its factorisation: a square discriminant permits either an irreducible cubic or three distinct linear factors, whereas a non-square discriminant forces a linear and an irreducible quadratic factor. Combining this with Lemma~\ref{lem:composition}, we obtain the following criterion.

\begin{lemma}\label{lem:cubic-factorisation}
Let $K$ be a finite field of characteristic $p \neq 2, 3$, and let
$f = ax^3 + bx^2 + cx + d \in K[x]$ with $a \neq 0$. Let
$\gamma_1, \gamma_2 \in \overline{K}$ be the critical points of $f$, i.e.\
the roots of $f'$. Let $g \in K[x]$ and $n \geq 1$ be such that
$h := g \circ f^{n-1}$ is irreducible over $K$, and set
$D := \deg h = 3^{n-1} \deg g$. Assume moreover that
\[
  g\bigl(f^n(\gamma_1)\bigr)\, g\bigl(f^n(\gamma_2)\bigr) \neq 0.
\]
Then $g\bigl(f^n(\gamma_1)\bigr)\, g\bigl(f^n(\gamma_2)\bigr) \in K$, and:
\begin{enumerate}
  \item if $(-3)^{\deg g}\, g\bigl(f^n(\gamma_1)\bigr)\, g\bigl(f^n(\gamma_2)\bigr)$
        is a square in $K$, then $g \circ f^n$ is either irreducible or the
        product of three irreducible polynomials of degree $D$;
  \item if $(-3)^{\deg g}\, g\bigl(f^n(\gamma_1)\bigr)\, g\bigl(f^n(\gamma_2)\bigr)$
        is not a square in $K$, then $g \circ f^n$ is the product of two
        irreducible polynomials, of degrees $2D$ and $D$.
\end{enumerate}
\end{lemma}

\begin{proof}
The product $g(f^n(\gamma_1))\,g(f^n(\gamma_2))$ is symmetric in
$\gamma_1, \gamma_2$, which are the roots of $f' \in K[x]$; hence it lies
in $K$.

Let $\beta \in \overline{K}$ be a root of $h$ and put $L := K(\beta)$, so
$[L : K] = D$. Since $g \circ f^n = h \circ f$, by Lemma~\ref{lem:composition}
(applied to $h$ and $f$) it suffices to determine how the cubic $f - \beta$
factors over $L$, provided $f - \beta$ is separable.

\emph{Step~1: the discriminant of $f - \beta$.}
For a cubic $F$ with leading coefficient $a$ we have
$\Delta(F) = -a^{-1}\operatorname{Res}(F, F')$. Since
$F' = f' = 3a(x - \gamma_1)(x - \gamma_2)$, we get
$\operatorname{Res}(F, F') = \operatorname{Res}(F', F) = (3a)^3 F(\gamma_1) F(\gamma_2)$,
and so
\[
  \Delta(f - \beta) = -27a^2\,\bigl(f(\gamma_1) - \beta\bigr)\bigl(f(\gamma_2) - \beta\bigr).
\]
If $f(\gamma_i) = \beta$ for some $i$, then
$0 = h(f(\gamma_i)) = g(f^n(\gamma_i))$, contrary to our assumption. Hence
$\Delta(f - \beta) \neq 0$ and $f - \beta$ is separable.

\emph{Step~2: factorisation of a separable cubic over a finite field.}
The Galois group of $f - \beta$ over $L$ is cyclic, generated by the
Frobenius, and $\Delta(f - \beta)$ is a square in $L$ if and only if the
Frobenius acts as an even permutation of the three roots. An even
permutation of three letters is the identity or a $3$-cycle, so in this case
$f - \beta$ splits completely or is irreducible over $L$. An odd
permutation is a transposition, so otherwise $f - \beta$ is the product of
an irreducible quadratic and a linear factor over $L$.

\emph{Step~3: squares and the norm.}
The norm $N_{L/K} \colon L^\times \to K^\times$ is a surjective group
homomorphism with $N_{L/K}\bigl((L^\times)^2\bigr) \subseteq (K^\times)^2$.
Since $|K|$ is odd, $(K^\times)^2$ has index $2$ in $K^\times$, so
$N_{L/K}^{-1}\bigl((K^\times)^2\bigr)$ has index $2$ in $L^\times$. It
contains $(L^\times)^2$, which also has index $2$, so the two subgroups are
equal. Hence an element $y \in L^\times$ is a square in $L$ if and only if
$N_{L/K}(y)$ is a square in $K$.

\emph{Step~4: computing the norm.}
Write $\Phi(X) := \bigl(f(\gamma_1) - X\bigr)\bigl(f(\gamma_2) - X\bigr)$,
which lies in $K[X]$ because it is symmetric in $\gamma_1, \gamma_2$. If
$\beta_1, \dots, \beta_D$ are the roots of $h$, then
$h(x) = \operatorname{lc}(h)\prod_{j}(x - \beta_j)$, and therefore
\[
  N_{L/K}\bigl(\Phi(\beta)\bigr)
  = \prod_{j=1}^{D} \Phi(\beta_j)
  = \prod_{i=1}^{2} \prod_{j=1}^{D} \bigl(f(\gamma_i) - \beta_j\bigr)
  = \frac{h(f(\gamma_1))\, h(f(\gamma_2))}{\operatorname{lc}(h)^2}
  = \frac{g(f^n(\gamma_1))\, g(f^n(\gamma_2))}{\operatorname{lc}(h)^2}.
\]
Since $-27a^2 \in K$, we have $N_{L/K}(-27a^2) = (-27a^2)^D$, and thus
\[
  N_{L/K}\bigl(\Delta(f - \beta)\bigr)
  = (-27a^2)^D\, \frac{g(f^n(\gamma_1))\, g(f^n(\gamma_2))}{\operatorname{lc}(h)^2}.
\]
Now $-27a^2 = -3\,(3a)^2$, and $D = 3^{n-1}\deg g \equiv \deg g \pmod 2$.
Up to squares in $K^\times$, the right-hand side is therefore
$(-3)^{\deg g}\, g(f^n(\gamma_1))\, g(f^n(\gamma_2))$.

Combining Steps~2--4, $f - \beta$ is irreducible or split over $L$ if and
only if $(-3)^{\deg g}\, g(f^n(\gamma_1))\, g(f^n(\gamma_2))$ is a square
in $K$, and is a product of a quadratic and a linear factor otherwise.
By Lemma~\ref{lem:composition}, the irreducible factors of $g \circ f^n$ then
have degrees $3D$, or $D, D, D$, in the first case, and $2D, D$ in the
second.
\end{proof}
\begin{remark}\label{remark 4.5}
  Lemma~\ref{lem:cubic-factorisation} is the reason for attaching \emph{types} to polynomials, as in~\cite{article} (Definition~\ref{def types} below). When $f$ is PCF, the points $f^n(\gamma_1)$ and $f^n(\gamma_2)$ take only finitely many values as $n$ varies. Hence, for every $n$, the criterion of Lemma~\ref{lem:cubic-factorisation} only involves finitely many elements $g(f^{k}(\gamma_1))g(f^{k}(\gamma_2))$, namely those with $1\leq k\leq m$, where $m$ is the length of the combined critical orbit. Just as Goksel works over $\mathbb{Q}(i)$ for quadratic polynomials~\cite{article}, we work over number fields containing $\sqrt{-3}$. Then $-3$ is a square in every residue field of odd characteristic, so the factor $(-3)^{\deg g}$ in Lemma~\ref{lem:cubic-factorisation} plays no role, and only the elements $g(f^{k}(\gamma_1))g(f^{k}(\gamma_2))$ matter.
\end{remark}
\begin{remark}\label{remark 4.6}
  The probabilities $\frac{2}{3}$ and $\frac{1}{3}$ in the model below come from the following heuristic. In the first case of Lemma~\ref{lem:cubic-factorisation}, the Frobenius acts on the three roots of $f-\beta$ as an even permutation, that is, as the identity or as one of the two $3$-cycles, and we expect these three possibilities to be equally likely: $f-\beta$ then splits with probability $\frac{1}{3}$ and stays irreducible with probability $\frac{2}{3}$. This agrees with a count over $\mathbb{F}_p$. Among the $\frac{p^2(p-1)}{2}$ monic cubics over $\mathbb{F}_p$ with non-zero square discriminant, $\binom{p}{3}$ split into distinct linear factors and $\frac{p^3-p}{3}$ are irreducible, so the proportion of split ones is $\frac{p-2}{3p}\to\frac{1}{3}$ as $p\to\infty$. (If we also count the $p^2$ monic cubics with a repeated root, all of which split, there are $\binom{p}{3}+p(p-1)+p$ split cubics among the $\frac{p^2(p+1)}{2}$ monic cubics with square discriminant, with the same limit.)
\end{remark}

  The type of a polynomial records, for each $k$, whether $g(f^{k}(\gamma_1))g(f^{k}(\gamma_2))$ is a square. The types are the states of our probabilistic model.
\begin{definition}\label{def types}
    Let $K$ be a field of characteristic different from $2$, and let $f\in K[x]$ be a cubic polynomial with critical points $\gamma_1,\gamma_2$ such that the sets $P_k=\{f^k(\gamma_1),f^k(\gamma_2)\}$ satisfy $P_{m+1}=P_{l_f}$ as in Definition~\ref{comb}; for instance, $K$ is a residue field of a number field and $f$ is the reduction of a PCF polynomial with combined critical orbit $(P_1,\dots,P_m)$. Let $g\in K[x]$ with $g(f^k(\gamma_1))g(f^k(\gamma_2))\neq 0$ for $1\leq k\leq m$. The \emph{type} of $g$ is the word
    $$\tau(g)=\tau(g)_1\tau(g)_2\cdots\tau(g)_m$$
    of length $m$ in the alphabet $\{s,n\}$ whose $k$-th letter $\tau(g)_k$ is $s$ if $g(f^k(\gamma_1))g(f^k(\gamma_2))$ is a square in $K$, and $n$ otherwise. (This element lies in $K$, since it is symmetric in $\gamma_1$ and $\gamma_2$.)
\end{definition}

To define our model, we need to know how to pass from one state (the factorisation of $f^n$) to the next (the factorisation of $f^{n+1}$). This transition can be understood as an action of $f$ on types, as defined below.
\begin{definition}\label{def 4.8}
 Let $f \in K[x]$ be a PCF cubic polynomial with combined critical orbit
$(P_1, \dots, P_m)$ and return index $l_f$, so that $P_{m+1} = P_{l_f}$
(Definition~\ref{comb}). For a type $c = c_1 c_2 \cdots c_m$
with $c_k \in \{s, n\}$, we define the \emph{action of $f$ on $c$} by
\[
  f(c) := c_2 c_3 \cdots c_m \, c_{l_f}.
\]
That is, $f$ deletes the first letter, shifts the remaining letters one
place to the left, and appends a copy of the $l_f$-th letter.
\end{definition}
\begin{lemma}\label{lem:action-on-types}
Let $f$ be as in Definition~\ref{def 4.8} and let
$g \in K[x]$ be such that $g(f^k(\gamma_i)) \neq 0$ for $i = 1, 2$ and
$1 \leq k \leq m + 1$. Then
\[
  \tau(g \circ f) = f\bigl(\tau(g)\bigr).
\]
\end{lemma}
\begin{proof}
    By Definition~\ref{def types}, the $k$-th letter of $\tau(g \circ f)$ records
whether
\[
  (g \circ f)\bigl(f^k(\gamma_1)\bigr)\,(g \circ f)\bigl(f^k(\gamma_2)\bigr)
  = g\bigl(f^{k+1}(\gamma_1)\bigr)\,g\bigl(f^{k+1}(\gamma_2)\bigr)
\]
is a square. For $1 \leq k \leq m-1$ this is exactly the condition
defining the $(k+1)$-th letter of $\tau(g)$. For $k = m$, the equality of
sets $P_{m+1} = P_{l_f}$ means that the pair
$\bigl(f^{m+1}(\gamma_1), f^{m+1}(\gamma_2)\bigr)$ equals
$\bigl(f^{l_f}(\gamma_1), f^{l_f}(\gamma_2)\bigr)$ up to order, so
\[
  g\bigl(f^{m+1}(\gamma_1)\bigr)\,g\bigl(f^{m+1}(\gamma_2)\bigr)
  = g\bigl(f^{l_f}(\gamma_1)\bigr)\,g\bigl(f^{l_f}(\gamma_2)\bigr),
\]
and the $m$-th letter of $\tau(g \circ f)$ equals the $l_f$-th letter of
$\tau(g)$.
\end{proof}
If $g$ is irreducible and $g\circ f$ is again irreducible, Lemma~\ref{lem:action-on-types} tells us the type of the new state. However, $g\circ f$ may be reducible, and then we need the types of its irreducible factors. These are not independent of each other. For instance, if $g\circ f=h_1h_2$, then for every $j$
\begin{equation}\label{equationtype}
    (g\circ f)(f^j(\gamma_1))\,(g\circ f)(f^j(\gamma_2))=h_1(f^j(\gamma_1))\,h_1(f^j(\gamma_2))\cdot h_2(f^j(\gamma_1))\,h_2(f^j(\gamma_2)),
\end{equation}
and similarly if $g\circ f=h_1h_2h_3$. Over a finite field $\mathbb{F}_q$ of odd characteristic, whether a non-zero element is a square only depends on its class in $\mathbb{F}_q^{\times}/(\mathbb{F}_q^{\times})^2\cong\{\pm 1\}$, so \eqref{equationtype} becomes an equation between types. To write it down, we multiply types letter by letter.
\begin{definition}\label{def:type-product}
    We identify the letters $s$ and $n$ with $1$ and $-1$, so that the $k$-th letter of a type is the quadratic character of the corresponding element. For two types $a=a_1\cdots a_m$ and $b=b_1\cdots b_m$ of the same length we put $ab:=(a_1b_1)\cdots(a_mb_m)$; thus $ss=nn=s$ and $sn=ns=n$, and the types of length $m$ form a group isomorphic to $\{\pm1\}^m$.
\end{definition}
With this notation, if $g\circ f=h_1\cdots h_r$ with $r\in\{2,3\}$, then \eqref{equationtype} and Lemma~\ref{lem:action-on-types} give
$$\tau(h_1)\cdots\tau(h_r)=\tau(g\circ f)=f\bigl(\tau(g)\bigr).$$
This equation typically has several solutions $(\tau(h_1),\dots,\tau(h_r))$, which we call \emph{admissible}.
Not all admissible tuples need to occur. For instance, if the critical points collide, that is, if $f^l(\gamma_1) = f^l(\gamma_2)$ for some $l$, then the $l$-th letter of every type is $s$, since $g(f^l(\gamma_1))g(f^l(\gamma_2))=g(f^l(\gamma_1))^2$.
We do not address this situation here, as it has been studied in detail in~\cite{benedetto2024arborealgaloisgroupscubic}.

For certain polynomials the model can be refined, because some admissible types never occur. This happens when the co-critical points of $f$ form one of the sets $P_j$ of the combined critical orbit; we study this in Section~\ref{sec:cocritical}. For the polynomial of length three, $-2z^3+3z^2+\frac{1}{2}$, the critical orbit itself forces a relation between the letters of every type (Lemma~\ref{lem:three-relation}), and we refine the model in Subsection~\ref{sec:three-model}.
Similar behaviour occurs for quadratic polynomials, see~\cite[Proposition 3.4]{citation-key},~\cite{goksel2013refined} and~\cite[Theorems 1.3 and 1.7]{goksel2026note}.
For the other polynomials with combined critical orbits of lengths one and two, the computations in Appendix~\ref{app:computations} agree with the model, except for a few values of $t$ (Subsection~\ref{app:exceptions}).

We can now state our model, using Lemma~\ref{lem:cubic-factorisation} and Remark~\ref{remark 4.6}. The model assumes that, among the possible outcomes allowed by Lemma~\ref{lem:cubic-factorisation} and by the equation between types above, all admissible tuples of types are equally likely.

\begin{definition} \label{Definition 4.9}
 Let $K$ be a number field containing $\sqrt{-3}$ and let $f\in K[x]$ be a PCF cubic polynomial whose critical orbits do not collide, with combined critical orbit of length $m$. The \emph{Markov model associated with the factorisation of} $f$ is the Markov process whose states are pairs $[c,k]$, where $c\in\{s,n\}^m$ is a type and $k\geq 1$ is a degree, with the following transitions, where $f(c)$ is as in Definition~\ref{def 4.8}.
 \begin{enumerate}
     \item If $c$ starts with $s$, then with probability $\frac{2}{3}$ the state $[c,k]$ is replaced by $[f(c),3k]$, and with probability $\frac{1}{3}$ it is replaced by the three states $[d_1,k],[d_2,k],[d_3,k]$, where $(d_1,d_2,d_3)$ runs with equal probabilities over the ordered triples of types with $d_1d_2d_3=f(c)$.
     \item If $c$ starts with $n$, then $[c,k]$ is replaced by the two states $[d_1,2k],[d_2,k]$, where $(d_1,d_2)$ runs with equal probabilities over the ordered pairs of types with $d_1d_2=f(c)$.
 \end{enumerate}
 Different states evolve independently. (As in~\cite{citation-key} and~\cite{article} we call this a Markov process; more precisely, it is a multi-type branching process.) Starting from the single state $[c,1]$, where $c$ is the type of $x$, the multiset of states obtained after $n$ steps models the irreducible factors of $f^n$ modulo a random prime of $K$, together with their types and degrees.
\end{definition}
\begin{remark}\label{rem:conjugates}
    Let $f$ be as in Definition~\ref{Definition 4.9} and let $t\in K$. The conjugate $f_t$ of \eqref{eq:conjugate} has critical points $\gamma_1-t$ and $\gamma_2-t$, and $f_t^k(\gamma_i-t)=f^k(\gamma_i)-t$ for $k\geq 1$. Hence the sets $P_k$ of $f_t$ are those of $f$ translated by $-t$, and $f_t$ has the same length and the same index $l_f$ as $f$. So $f$ and $f_t$ have the same action on types and the same Markov process, and only the initial state changes: the type of $x$ with respect to $f_t$, which we call the \emph{root type} of $f_t$, has $k$-th letter $s$ if $(f^k(\gamma_1)-t)(f^k(\gamma_2)-t)$ is a square modulo the prime, and $n$ otherwise. The co-critical points of Definition~\ref{def:cocritical} are translated in the same way.
\end{remark}
 
We recall~\cite[Definition 2.8]{article}, adapted to cubic polynomials.
\begin{definition}
    A level $n$ \emph{type sequence} is a partition of $3^n$ into parts of the form $2^a3^b$, with $a,b\geq 0$, together with a type for each part. A level $n$ \emph{datum} is a level $n$ type sequence together with a probability $0 \leq p\leq 1$, and a finite collection of such pairs whose probabilities add up to $1$ is called \emph{level $n$ data}.
\end{definition}
We write a datum as $[[c_1,k_1]\cdots[c_r,k_r],p]$, and we write $[c,k]^j$ for $j$ equal states. We illustrate these definitions in the following examples.
\begin{example}
    Let $a\in K$ and take the cubic polynomial $f=-2(x+a)^3+3(x+a)^2-a\in K[x]$. Its critical points $\gamma_1=-a$ and $\gamma_2=1-a$ are both fixed by $f$, so the combined critical orbit is $(P_1)$ with $P_1=\{-a,1-a\}$, it has length one, and $f(c)=c$ for every type $c$. Starting from the level $n$ datum $[[s,3^n],1]$, Definition~\ref{Definition 4.9} gives the level $n+1$ data
    $$[[s,3^{n+1}],\tfrac{2}{3}],\ [[s,3^n]^3,\tfrac{1}{12}],\ [[s,3^n][n,3^n]^2,\tfrac{1}{4}].$$
    Indeed, the possible triples of types $(d_1,d_2,d_3)$ are $(s,s,s)$, $(s,n,n)$, $(n,s,n)$ and $(n,n,s)$, so when the state splits, the combination $[s,3^n][n,3^n]^2$ occurs $\frac{3}{4}$ of the times and $[s,3^n]^3$ only $\frac{1}{4}$ of the times.

    Similarly, starting from $[[n,k],1]$ we obtain $[[n,2k][s,k],\frac{1}{2}],\ [[s,2k][n,k],\frac{1}{2}]$.
\end{example}

We give another slightly more complicated example.
\begin{example}\label{Example 4.13}
    Consider the PCF cubic polynomial $-z^3+\frac{3}{2}z^2-1$ over $\mathbb{Q}$. Since $2^{-1}\equiv -3\pmod 7$, its reduction modulo $7$ is $-z^3-2z^2-1$. The critical points are $\gamma_1=0$ and $\gamma_2=1$, with orbits $0\mapsto -1\mapsto -2\mapsto -1$ and $1\mapsto 3\mapsto 3$ modulo $7$.

    The combined critical orbit is $(P_1,P_2)=(\{-1,3\},\{-2,3\})$, and $P_3=P_1$, so $f(c_1c_2)=c_2c_1$. Starting from $[[ns,k],1]$ we obtain
    $$[[sn,2k][ss,k],\tfrac{1}{4}],\ [[ss,2k][sn,k],\tfrac{1}{4}],\ [[ns,2k][nn,k],\tfrac{1}{4}],\ [[nn,2k][ns,k],\tfrac{1}{4}].$$
    Starting from $[[nn,k],1]$ we obtain
    $$[[nn,2k][ss,k],\tfrac{1}{4}],\ [[ss,2k][nn,k],\tfrac{1}{4}],\ [[sn,2k][ns,k],\tfrac{1}{4}],\ [[ns,2k][sn,k],\tfrac{1}{4}].$$
    Starting from $[[ss,k],1]$ we obtain
    \[
\resizebox{\textwidth}{!}{$
[[ss,3k],\frac{2}{3}],\ [[ss,k]^3,\frac{1}{48}],\ [[ss,k][nn,k]^2,\frac{1}{16}],\ [[ss,k][ns,k]^2,\frac{1}{16}],\ [[ss,k][sn,k]^2,\frac{1}{16}],\ [[nn,k][ns,k][sn,k],\frac{1}{8}].
$}
\]
    Finally, starting from $[[sn,k],1]$ we obtain
    \[
\resizebox{\textwidth}{!}{$
[[ns,3k],\frac{2}{3}],\ [[ss,k]^2[ns,k],\frac{1}{16}],\ [[ss,k][sn,k][nn,k],\frac{1}{8}],\ [[ns,k]^3,\frac{1}{48}],\ [[ns,k][sn,k]^2,\frac{1}{16}],\ [[ns,k][nn,k]^2,\frac{1}{16}].
$}
\]
\end{example}

\section{Relations between the letters of a type}\label{sec:cocritical}
For some of the polynomials of~\cite{cubic-PCF} the letters of a type are related, so that some admissible types never occur, and the Markov model of Definition~\ref{Definition 4.9} can be refined. In Subsections~\ref{sec:cocritical-relation} and \ref{sec:cocritical-types} the relation is caused by the co-critical points of $f$, for two polynomials of length two. In Subsection~\ref{sec:three-model} it is caused by the critical orbit of the polynomial of length three. In the first two subsections, $F$ is a field of characteristic different from $2$ and $3$, $f\in F[x]$ is a cubic polynomial with leading coefficient $a$ and distinct critical points $\gamma_1,\gamma_2$, and for $k\geq 1$ we put
$$P_k(x):=\bigl(f^k(\gamma_1)-x\bigr)\bigl(f^k(\gamma_2)-x\bigr)\in F[x],$$
which only depends on the set $P_k$ of Definition~\ref{comb}.

\subsection{The co-critical relation}\label{sec:cocritical-relation}
\begin{definition}\label{def:cocritical}
    For $i=1,2$ the polynomial $f(z)-f(\gamma_i)$ has a double root at $\gamma_i$, and we write
    $$f(z)-f(\gamma_i)=a(z-\gamma_i)^2(z-\gamma_i^{\ast}).$$
    Following~\cite[Section 1]{BonifantKiwiMilnor2010}, we call $\gamma_1^{\ast},\gamma_2^{\ast}$ the \emph{co-critical points} of $f$. Since the sum of the roots of $f(z)-f(\gamma_i)$ is $\frac{3}{2}(\gamma_1+\gamma_2)$, we have $\gamma_1^{\ast}=\frac{3\gamma_2-\gamma_1}{2}$ and $\gamma_2^{\ast}=\frac{3\gamma_1-\gamma_2}{2}$.
\end{definition}
The following identity explains how the co-critical points enter the factorisation of the iterates.
\begin{lemma}\label{lem:cocritical-identity}
    In $F[x]$ we have
    $$P_1\bigl(f(x)\bigr)=\Bigl(\frac{f'(x)}{3}\Bigr)^2(x-\gamma_1^{\ast})(x-\gamma_2^{\ast}).$$
    In particular, if $\{\gamma_1^{\ast},\gamma_2^{\ast}\}=P_j$ for some $j\geq 1$, then
    \begin{equation}\label{eq:cocritical}
        P_1\bigl(f(x)\bigr)=\Bigl(\frac{f'(x)}{3}\Bigr)^2P_j(x).
    \end{equation}
\end{lemma}
\begin{proof}
    Evaluating Definition~\ref{def:cocritical} at $z=x$ gives $f(\gamma_i)-f(x)=-a(x-\gamma_i)^2(x-\gamma_i^{\ast})$ for $i=1,2$. Multiplying these two identities,
    $$P_1\bigl(f(x)\bigr)=\bigl(f(\gamma_1)-f(x)\bigr)\bigl(f(\gamma_2)-f(x)\bigr)=a^2(x-\gamma_1)^2(x-\gamma_2)^2(x-\gamma_1^{\ast})(x-\gamma_2^{\ast}),$$
    and $a(x-\gamma_1)(x-\gamma_2)=\frac{f'(x)}{3}$. If $\{\gamma_1^{\ast},\gamma_2^{\ast}\}=\{f^j(\gamma_1),f^j(\gamma_2)\}$, then $(x-\gamma_1^{\ast})(x-\gamma_2^{\ast})=\bigl(f^j(\gamma_1)-x\bigr)\bigl(f^j(\gamma_2)-x\bigr)=P_j(x)$.
\end{proof}
The condition $\{\gamma_1^{\ast},\gamma_2^{\ast}\}=P_j$ is very restrictive.
\begin{proposition}\label{prop:cocritical-classification}
    \begin{enumerate}[i)]
        \item If $\{\gamma_1^{\ast},\gamma_2^{\ast}\}=P_j$, then $P_{j+1}=P_1$. In particular $f$ is PCF, with combined critical orbit of length one if $j=1$, and of length at most two with $P_3=P_1$ if $j=2$.
        \item Among the polynomials of~\cite{cubic-PCF} with non-colliding critical orbits and combined critical orbit of length two, the condition $\{\gamma_1^{\ast},\gamma_2^{\ast}\}=P_j$ holds exactly for the following two, both with $j=2$:
        \begin{center}
        \renewcommand{\arraystretch}{1.3}
        \begin{tabular}{llll}
            \toprule
            $f$ & $P_1$ & $P_2$ & $\{\gamma_1^{\ast},\gamma_2^{\ast}\}$\\
            \midrule
            $-z^3+\frac{3}{2}z^2-1$ & $\{-1,-\frac{1}{2}\}$ & $\{\frac{3}{2},-\frac{1}{2}\}$ & $\{\frac{3}{2},-\frac{1}{2}\}$\\
            $-\frac{1}{4}z^3+\frac{3}{2}z+2$ & $\{2+\sqrt{2},2-\sqrt{2}\}$ & $\{-2\sqrt{2},2\sqrt{2}\}$ & $\{-2\sqrt{2},2\sqrt{2}\}$\\
            \bottomrule
        \end{tabular}
        \end{center}
    \end{enumerate}
\end{proposition}
\begin{proof}
    i) Since $f(\gamma_i^{\ast})=f(\gamma_i)$, we get $P_{j+1}=f(P_j)=\{f(\gamma_1^{\ast}),f(\gamma_2^{\ast})\}=\{f(\gamma_1),f(\gamma_2)\}=P_1$.

    ii) For each polynomial of the list we compute the critical points, the sets $P_1$ and $P_2$ and the co-critical points. The table shows the two cases of equality, and for the other two polynomials the condition fails; for instance, $2z^3-3z^2+\frac{1}{2}$ has critical points $0$ and $1$, so $\{\gamma_1^{\ast},\gamma_2^{\ast}\}=\{\frac{3}{2},-\frac{1}{2}\}$, whereas $P_1=\{\frac{1}{2},-\frac{1}{2}\}$ and $P_2=\{0,-\frac{1}{2}\}$.
\end{proof}

\subsection{Consequences for the types}\label{sec:cocritical-types}
For a polynomial $h$ we write $\tau(h)_k$ for the $k$-th letter of its type (Definition~\ref{def types}), and we identify $s$ and $n$ with $1$ and $-1$.
\begin{proposition}\label{prop:cocritical-types}
    Let $K$ be a finite field of characteristic $p\neq 2,3$ and let $f\in K[x]$ be a cubic with $\{\gamma_1^{\ast},\gamma_2^{\ast}\}=P_j$. Let $h\in K[x]$ be irreducible with root $\beta$, assume that $h\circ f$ is separable, and let $h_1$ be an irreducible factor of $h\circ f$ with root $\alpha$, so that $f(\alpha)=\beta$. Put $e:=\deg h_1/\deg h=[K(\alpha):K(\beta)]$, and assume that $P_1(\beta)\neq 0$. Then
    $$\tau(h_1)_j=\tau(h)_1^{\,e}.$$
    Consequently, if $-3$ is a square in $K$, then: if $\tau(h)_1=s$, every irreducible factor of $h\circ f$ has $j$-th letter $s$; and if $\tau(h)_1=n$, then $h\circ f$ has a factor of degree $2\deg h$, whose $j$-th letter is $s$, and a factor of degree $\deg h$, whose $j$-th letter is $n$.
\end{proposition}
\begin{proof}
    Since $h(f^k(\gamma_1))h(f^k(\gamma_2))=\mathrm{lc}(h)^2N_{K(\beta)/K}\bigl(P_k(\beta)\bigr)$, Step~3 of the proof of Lemma~\ref{lem:cubic-factorisation} shows that $\tau(h)_k=s$ if and only if $P_k(\beta)$ is a square in $K(\beta)$. In the same way, $\tau(h_1)_j=s$ if and only if $P_j(\alpha)$ is a square in $K(\alpha)$. As $h\circ f$ is separable, $f-\beta$ has no multiple root, so $f'(\alpha)\neq 0$, and \eqref{eq:cocritical} gives
    $$P_j(\alpha)\equiv P_1(\beta)\mod \bigl(K(\alpha)^\times\bigr)^2.$$
    Finally, let $Q=|K(\beta)|$, which is odd. If $y\in K(\beta)^\times$ is not a square, then $y^{(Q-1)/2}=-1$ and
    $$y^{\frac{Q^e-1}{2}}=\bigl(y^{\frac{Q-1}{2}}\bigr)^{1+Q+\dots+Q^{e-1}}=(-1)^e,$$
    so $y$ is a square in $K(\alpha)$ if and only if $e$ is even. This proves the formula. If $-3$ is a square, then by Lemma~\ref{lem:cubic-factorisation} we have $e\in\{1,3\}$ for every factor when $\tau(h)_1=s$, while $h\circ f$ has two factors, with $e=2$ and $e=1$, when $\tau(h)_1=n$; this gives the last statement.
\end{proof}
In the language of Definition~\ref{Definition 4.9}, Proposition~\ref{prop:cocritical-types} says the following. Let $c$ be a type.
\begin{itemize}
    \item If $c_1=s$, then $[c,k]\to[f(c),3k]$ is always possible (since $P_{j+1}=P_1$ by Proposition~\ref{prop:cocritical-classification} i), the $j$-th letter of $f(c)$ is $c_1=s$), but in $[c,k]\to[d_1,k][d_2,k][d_3,k]$ only the triples in which every $d_i$ has $j$-th letter $s$ occur.
    \item If $c_1=n$, then in $[c,k]\to[d_1,2k][d_2,k]$ only the pairs in which $d_1$ has $j$-th letter $s$ and $d_2$ has $j$-th letter $n$ occur.
\end{itemize}
For $j=2$ we have $f(c_1c_2)=c_2c_1$, and the transitions of Example~\ref{Example 4.13} split as follows.
\begin{center}
\renewcommand{\arraystretch}{1.3}
\small
\begin{tabular}{lll}
    \toprule
    state & transitions that occur & transitions that never occur\\
    \midrule
    $[ss,k]$ & $[ss,3k]$,\ $[ss,k]^3$,\ $[ss,k][ns,k]^2$ & $[ss,k][nn,k]^2$,\ $[ss,k][sn,k]^2$,\ $[nn,k][ns,k][sn,k]$\\
    $[sn,k]$ & $[ns,3k]$,\ $[ns,k]^3$,\ $[ns,k][ss,k]^2$ & $[ns,k][sn,k]^2$,\ $[ns,k][nn,k]^2$,\ $[ss,k][sn,k][nn,k]$\\
    $[nn,k]$ & $[ss,2k][nn,k]$,\ $[ns,2k][sn,k]$ & $[nn,2k][ss,k]$,\ $[sn,2k][ns,k]$\\
    $[ns,k]$ & $[ss,2k][sn,k]$,\ $[ns,2k][nn,k]$ & $[sn,2k][ss,k]$,\ $[nn,2k][ns,k]$\\
    \bottomrule
\end{tabular}
\end{center}
In particular, the restricted dynamics $[nn,k]\to[nn,2k][ss,k]$, which we use to define $y_n$ in Section~\ref{4}, never occurs for these polynomials. The same restrictions apply to the first-level data of $f_t$, since the irreducible factors of $f_t=x\circ f_t$ are the children of the polynomial $x$.

Restrictions of this kind also occur for quadratic polynomials. For the critical orbit types $(2,n)$ and $(3,1)$, Goksel showed that the type of an irreducible polynomial $g$ forces relations between the letters of the types of the irreducible factors of $g(f^2(x))$, respectively $g(f^3(x))$~\cite[Theorems 1.3 and 1.7]{goksel2026note}. There the restrictions concern the descendants two or three levels down, which leads to a multi-step model. Proposition~\ref{prop:cocritical-types}, on the other hand, restricts the children of a factor, so the refined model below is still a Markov process on the types.
\begin{definition}\label{def:refined-model}
    Let $K$ be a number field containing $\sqrt{-3}$, and let $f\in K[x]$ be a cubic with $\{\gamma_1^{\ast},\gamma_2^{\ast}\}=P_2\neq P_1$, so that $f$ has combined critical orbit of length two and $P_3=P_1$. The \emph{refined Markov model} of $f$ is the Markov process of Definition~\ref{Definition 4.9} in which the tuples of types run with equal probabilities only over the tuples allowed by Proposition~\ref{prop:cocritical-types}:
    \begin{enumerate}
        \item if $c_1=s$, then with probability $\frac{2}{3}$ the state $[c,k]$ is replaced by $[f(c),3k]$, and with probability $\frac{1}{3}$ by $[d_1,k],[d_2,k],[d_3,k]$, where $(d_1,d_2,d_3)$ runs with equal probabilities over the ordered triples with $d_1d_2d_3=f(c)$ in which every $d_i$ has second letter $s$;
        \item if $c_1=n$, then $[c,k]$ is replaced by $[d_1,2k],[d_2,k]$, where $(d_1,d_2)$ runs with equal probabilities over the ordered pairs with $d_1d_2=f(c)$ such that $d_1$ has second letter $s$ and $d_2$ has second letter $n$.
    \end{enumerate}
\end{definition}
There are four allowed triples and two allowed pairs; for instance
\begin{align*}
    [ss,k]&\to[[ss,3k],\tfrac{2}{3}],[[ss,k]^3,\tfrac{1}{12}],[[ss,k][ns,k]^2,\tfrac{1}{4}],\\
    [nn,k]&\to[[ss,2k][nn,k],\tfrac{1}{2}],[[ns,2k][sn,k],\tfrac{1}{2}].
\end{align*}
\begin{remark}\label{rem:cocritical-numerics}
    We tested Proposition~\ref{prop:cocritical-types} for the co-critical polynomials of length two by factoring $f_t^k$ for $k\leq 5$ modulo $6000$ primes $p\equiv 1 \pmod 3$, for twenty-one values of $t$ (Subsection~\ref{app:relations}): among $5\,076\,629$ pairs formed by an irreducible factor and one of its children there is no violation. Moreover, except for two values of $t$, the remaining transitions occur with the frequencies predicted by Definition~\ref{def:refined-model} (Table~\ref{tab:cocrit-tests}).
\end{remark}

\subsection{The polynomial of length three}\label{sec:three-model}
By the classification of~\cite{cubic-PCF}, the only PCF cubic polynomial over $\mathbb{Q}$ with non-colliding critical orbits and combined critical orbit of length three is, up to conjugation,
$$f=-2z^3+3z^2+\frac{1}{2}.$$
Its critical points are $\gamma_1=0$ and $\gamma_2=1$, and their orbits are
$$0\mapsto\tfrac{1}{2}\mapsto 1\mapsto\tfrac{3}{2}\mapsto\tfrac{1}{2},\qquad 1\mapsto\tfrac{3}{2}\mapsto\tfrac{1}{2}\mapsto 1.$$
Hence $P_1=\{\frac{1}{2},\frac{3}{2}\}$, $P_2=\{1,\frac{1}{2}\}$, $P_3=\{\frac{3}{2},1\}$ and $P_4=P_1$, so that $l_f=1$ and the action of Definition~\ref{def 4.8} is the cyclic shift
$$f(c_1c_2c_3)=c_2c_3c_1.$$
The critical point $\gamma_2=1$ is periodic of period three, and $\gamma_1=0$ is strictly preperiodic: it is mapped into the cycle of $\gamma_2$ in one step. Hence the post-critical set $\{\frac{1}{2},1,\frac{3}{2}\}$ is a single cycle of length three, whose three subsets of two elements are $P_1$, $P_2$ and $P_3$. This has a consequence which the Markov model of Definition~\ref{Definition 4.9} does not see.
\begin{lemma}\label{lem:three-relation}
    Let $K$ be a field of characteristic different from $2$ and $3$, and let $g\in K[x]$ be a polynomial which does not vanish at $\frac{1}{2}$, $1$ and $\frac{3}{2}$. Then
    $$\prod_{k=1}^{3}g\bigl(f^k(\gamma_1)\bigr)g\bigl(f^k(\gamma_2)\bigr)=\Bigl(g\bigl(\tfrac{1}{2}\bigr)g(1)g\bigl(\tfrac{3}{2}\bigr)\Bigr)^2.$$
    In particular $P_1(x)P_2(x)P_3(x)=\bigl((\frac{1}{2}-x)(1-x)(\frac{3}{2}-x)\bigr)^2$, and if $K$ is finite, the type of $g$ (Definition~\ref{def types}) satisfies $\tau(g)_1\tau(g)_2\tau(g)_3=s$, where we identify $s$ and $n$ with $1$ and $-1$.
\end{lemma}
\begin{proof}
    The three factors on the left are $g(\frac{1}{2})g(\frac{3}{2})$, $g(1)g(\frac{1}{2})$ and $g(\frac{3}{2})g(1)$. For the second statement take $g=x-X$ over $K(X)$.
\end{proof}
Hence only the four types of
$$S:=\{sss,\ snn,\ nsn,\ nns\}$$
occur. The set $S$ is a subgroup of $\{s,n\}^3$, where types are multiplied letter by letter, and $f(S)=S$: the action of $f$ fixes $sss$ and permutes the other three types cyclically,
$$snn\mapsto nns\mapsto nsn\mapsto snn.$$
Since $c_3=c_1c_2$ for $c\in S$, a type in $S$ is determined by its first two letters. We stress that, unlike in Subsections~\ref{sec:cocritical-relation} and \ref{sec:cocritical-types}, the restriction does not come from the co-critical points: these are $\frac{3}{2}$ and $-\frac{1}{2}$, and $\{\frac{3}{2},-\frac{1}{2}\}$ is not one of the sets $P_j$, so Proposition~\ref{prop:cocritical-types} gives nothing. A relation between the letters of a type also occurs for quadratic polynomials $f$ of critical orbit type $(3,1)$, with critical point $\gamma$: Goksel proved that if $g$ is irreducible of even degree and its type begins with $nn$, then $H(f(\gamma))H(f^2(\gamma))$ is a square for every irreducible factor $H$ of $g(f^3(x))$~\cite[Theorem 1.7]{goksel2026note}. Here, in contrast, the relation holds for every polynomial $g$, because the post-critical set is a single cycle. In the computations of Appendix~\ref{app:computations} we found no further restriction, except for one special value of $t$ (Table~\ref{tab:exceptions}). We therefore restrict the Markov model of Definition~\ref{Definition 4.9} to the types in $S$.
\begin{definition}\label{def:three-model}
    Let $K$ be a number field containing $\sqrt{-3}$. The \emph{Markov model} of $f=-2z^3+3z^2+\frac{1}{2}$ is the Markov process of Definition~\ref{Definition 4.9} in which the types run over $S$:
    \begin{enumerate}
        \item if $c_1=s$, then with probability $\frac{2}{3}$ the state $[c,k]$ is replaced by $[f(c),3k]$, and with probability $\frac{1}{3}$ by $[d_1,k],[d_2,k],[d_3,k]$, where $(d_1,d_2,d_3)$ runs with equal probabilities over the $16$ ordered triples in $S^3$ with $d_1d_2d_3=f(c)$;
        \item if $c_1=n$, then $[c,k]$ is replaced by $[d_1,2k],[d_2,k]$, where $(d_1,d_2)$ runs with equal probabilities over the $4$ ordered pairs in $S^2$ with $d_1d_2=f(c)$.
    \end{enumerate}
\end{definition}
Explicitly, the transitions are the following:
\begin{align*}
    [sss,k]&\to[[sss,3k],\tfrac{2}{3}],\ [[sss,k]^3,\tfrac{1}{48}],\ [[sss,k][snn,k]^2,\tfrac{1}{16}],\\
    &\qquad[[sss,k][nsn,k]^2,\tfrac{1}{16}],\ [[sss,k][nns,k]^2,\tfrac{1}{16}],\ [[snn,k][nsn,k][nns,k],\tfrac{1}{8}],\\
    [snn,k]&\to[[nns,3k],\tfrac{2}{3}],\ [[nns,k]^3,\tfrac{1}{48}],\ [[nns,k][sss,k]^2,\tfrac{1}{16}],\\
    &\qquad[[nns,k][snn,k]^2,\tfrac{1}{16}],\ [[nns,k][nsn,k]^2,\tfrac{1}{16}],\ [[sss,k][snn,k][nsn,k],\tfrac{1}{8}],\\
    [nsn,k]&\to[[snn,2k][sss,k],\tfrac{1}{4}],\ [[sss,2k][snn,k],\tfrac{1}{4}],\\
    &\qquad[[nsn,2k][nns,k],\tfrac{1}{4}],\ [[nns,2k][nsn,k],\tfrac{1}{4}],\\
    [nns,k]&\to[[nsn,2k][sss,k],\tfrac{1}{4}],\ [[sss,2k][nsn,k],\tfrac{1}{4}],\\
    &\qquad[[snn,2k][nns,k],\tfrac{1}{4}],\ [[nns,2k][snn,k],\tfrac{1}{4}].
\end{align*}

\section{Markov groups for combined critical orbit of length one} \label{3}
All cubic PCF polynomials over $\mathbb{Q}$, up to conjugation, were classified in~\cite{cubic-PCF}. Those with non-colliding critical orbits, combined critical orbit of length one and $\{\gamma_1^{\ast},\gamma_2^{\ast}\}\neq P_1$ (Definition~\ref{def:cocritical}) are $-2z^3 + 3z^2$, $-z^3 + \frac{3}{2}z^2 + 1$, $2z^3 - 3z^2 + 1$ and $z^3 - \frac{3}{2}z^2$. Since the combined critical orbit has length one, $f(c)=c$ for every type $c$, so the model of Definition~\ref{Definition 4.9} is the same for all of them. We do not consider the two polynomials of length one in the list of~\cite{cubic-PCF} with $\{\gamma_1^{\ast},\gamma_2^{\ast}\}=P_1$, for which not all admissible types occur (Proposition~\ref{prop:cocritical-types}).

Throughout this section, $f$ is a PCF cubic polynomial with combined critical orbit of length one and critical points $\gamma_1,\gamma_2$, defined over a number field $K$ containing $\sqrt{-3}$, for instance $K=\mathbb{Q}(\sqrt{-3})$. We study the factorisation of the iterates of $f$ and of its conjugates $f_t$ with $t\in K$. By Remark~\ref{rem:conjugates}, the factorisation of $f_t^n$ is described by the Markov process of $f$ with initial state $[c,1]$, where $c$ is the root type of $f_t$, and by varying $t$ we obtain both root types. The Galois groups of the iterates are known for the Belyi map $-2z^3 + 3z^2$~\cite{Benedetto2017}, and more generally for a large class of dynamical Belyi maps~\cite{Bouw2018DynamicalBM}; we compare them with our groups in Section~\ref{5}.

Our goal is to construct groups $M_n\subseteq\mathrm{Aut}(\mathcal{T}_n)$ that follow the Markov model, that is, whose cycle data occur with the frequencies predicted by the model. The construction is guided by the following dictionary. An irreducible factor $g$ of $f_t^n$ modulo a prime corresponds to a cycle of a Frobenius element on the roots of $f_t^n$, that is, on the vertices of level $n$, and the length of the cycle is $\deg g$. By Lemma~\ref{lem:cubic-factorisation}, $g\circ f_t$ is irreducible of degree $3\deg g$, or the product of three factors of degree $\deg g$, or the product of two factors of degrees $2\deg g$ and $\deg g$. Correspondingly, a cycle of length $k$ on level $n$ can lift to level $n+1$ as one cycle of length $3k$, as three cycles of length $k$, or as one cycle of length $2k$ together with one cycle of length $k$. We first describe how to extend an automorphism of $\mathcal{T}_n$ to $\mathcal{T}_{n+1}$ so as to obtain each of these behaviours.

\begin{definition}\label{def 4.16}
    Let $\sigma\in\mathrm{Aut}(\mathcal{T}_n)$, and let $(s_u)_{u\in X^n}$ be a family of permutations $s_u\in S_3=\mathrm{Aut}(\mathcal{T}_1)$, one for each vertex of level $n$. We denote by $i_{\{s_u\}}(\sigma)\in\mathrm{Aut}(\mathcal{T}_{n+1})$ the unique automorphism which acts as $\sigma$ on $\mathcal{T}_n$ and has section $s_u$ at every vertex $u$ of level $n$, that is,
    $$(uj)\,i_{\{s_u\}}(\sigma)=(u)\sigma\,(j)s_u\qquad\text{for } u\in X^n \text{ and } j\in X.$$
\end{definition}
\begin{definition}\label{def 5.2}
    Let $\sigma\in \mathrm{Aut}(\mathcal{T}_{n})$. We write each cycle of $\sigma$ on the vertices of level $n$ as $(a_1,\dots,a_r)$, with $(a_i)\sigma=a_{i+1}$ and starting from its smallest vertex $a_1$ in the lexicographic order, and we call $a_r$ the \emph{last vertex} of the cycle. The \emph{splitting}, \emph{doubling} and \emph{tripling} of $\sigma$ are the elements $i_{\{s_u\}}(\sigma)$ of $\mathrm{Aut}(\mathcal{T}_{n+1})$ with, respectively,
    \begin{itemize}
        \item $s_u=\mathrm{id}$ for all $u$;
        \item $s_u=(0,1)$ if $u$ is the last vertex of a cycle of $\sigma$, and $s_u=\mathrm{id}$ otherwise;
        \item $s_u=(0,1,2)$ if $u$ is the last vertex of a cycle of $\sigma$, and $s_u=\mathrm{id}$ otherwise.
    \end{itemize}
\end{definition}
Let $(a_1,\dots,a_r)$ be a cycle of $\sigma$ and $\tilde\sigma=i_{\{s_u\}}(\sigma)$. Applying $\tilde\sigma$ $r$ times to a child $a_1j$ of $a_1$, we go once around the cycle and come back to the child $a_1\,(j)s_{a_1}s_{a_2}\cdots s_{a_r}$ of $a_1$. For the three elements of Definition~\ref{def 5.2} this product is $s_{a_r}$, so the cycle $(a_1,\dots,a_r)$ lifts to three cycles of length $r$ in the splitting, to a cycle of length $2r$ and a cycle of length $r$ in the doubling, and to one cycle of length $3r$ in the tripling, as desired. We illustrate these definitions with two examples.
\begin{example}\label{Example 4.18}
    The tripling of the $3$-cycle $x_1=(0,1,2)\in \mathrm{Aut}(\mathcal{T}_1)$ is a $9$-cycle on level $2$, and more generally the iterated triplings $x_n\in\mathrm{Aut}(\mathcal{T}_n)$ of $x_1$ are $3^n$-cycles on level $n$, with $x_n$ restricting to $x_k$ on $\mathcal{T}_k$ for $k\leq n$. They define an element $x=\varprojlim_n x_n\in \mathrm{Aut}(\mathcal{T})$, and one checks that $x=(\mathrm{id}, \mathrm{id},x)(0,1,2)$: this is the $3$-adic adding machine.
\end{example}
\begin{example}
    The identity $e$ of $\mathrm{Aut}(\mathcal{T}_1)$ has three cycles of length one, so its tripling is $(x_1,x_1,x_1)$, with $x_1=(0,1,2)$ at each vertex of the first level, and its $n$-fold tripling is $(x_n,x_n,x_n)\in\mathrm{Aut}(\mathcal{T}_{n+1})$: below each vertex of the first level we recover the adding machine of Example~\ref{Example 4.18}.
\end{example}

The Markov model suggests how to iterate elements whose cycles carry types.
\begin{definition}
    Let $\sigma\in\mathrm{Aut}(\mathcal{T}_n)$, and assume that each cycle of $\sigma$ on level $n$ has been given a type $s$ or $n$. The \emph{iteration} of $\sigma$ is the element of $\mathrm{Aut}(\mathcal{T}_{n+1})$ obtained by tripling the cycles of type $s$ and doubling the cycles of type $n$, that is, $i_{\{s_u\}}(\sigma)$ with $s_u=(0,1,2)$, respectively $s_u=(0,1)$, if $u$ is the last vertex of a cycle of type $s$, respectively $n$, and $s_u=\mathrm{id}$ otherwise.
\end{definition}
With these definitions in place, we are ready to introduce the Markov groups. Since $\sqrt{-3}\in K$, the root type of $f_t$ at a random prime is $s$ or $n$ with probability $\frac{1}{2}$ each (unless $P_1(t)$ is a square in $K$, see the models at the end of this section), and averaging the first step of the model from $[s,1]$ and from $[n,1]$ we obtain the following first-level data of $f_t$:
$$[[s,3],\tfrac{1}{3}],\ [[s,1]^3,\tfrac{1}{24}],\ [[n,1]^2[s,1],\tfrac{1}{8}],\ [[n,2][s,1],\tfrac{1}{4}],\ [[s,2][n,1],\tfrac{1}{4}].$$
If we forget the types, the factorisation pattern $[3]$ occurs with frequency $\frac{1}{3}$, the pattern $[1,1,1]$ with frequency $\frac{1}{24}+\frac{1}{8}=\frac{1}{6}$ and the pattern $[2,1]$ with frequency $\frac{1}{2}$. These are exactly the cycle-type frequencies of $S_3$ acting on three letters, with its two $3$-cycles, its identity and its three transpositions, so we take $M_1=S_3$.

By Definition~\ref{Definition 4.9}, the dynamics on types is
\begin{align*}
    &[s,k]\to [[s,3k],\tfrac{2}{3}],\ [[s,k]^3,\tfrac{1}{12}],\ [[s,k][n,k]^2,\tfrac{1}{4}],\\
    &[n,k]\to [[n,2k][s,k],\tfrac{1}{2}],\ [[s,2k][n,k],\tfrac{1}{2}].
\end{align*}
To define generators of our groups recursively, we only use one transition for each type, which we call the \emph{restricted dynamics}:
\begin{align*}
    &[s,k] \to [s,3k],\\
    &[n,k]\to [n,2k][s,k].
\end{align*}

We start from $M_1=S_3$ with generators $(0,1,2)$ and $(0,1)$; we attach to $(0,1,2)$ the type $[s,3]$, to $(0,1)$ the type $[n,2][s,1]$ and to the identity the type $[n,1]^2[s,1]$. Iterating these elements with respect to the restricted dynamics, we obtain elements $x_n$, $y_n$ and $z_n$ of $\mathrm{Aut}(\mathcal{T}_n)$. Rather than following Definition~\ref{def 5.2} literally, it is more convenient to define them recursively by
$$x_n=(\mathrm{id},\mathrm{id},x_{n-1})(0,1,2),\qquad y_n=(\mathrm{id},y_{n-1},x_{n-1})(0,1),\qquad z_n=(y_{n-1},y_{n-1},x_{n-1}),$$
where $x_0=y_0=z_0=\mathrm{id}\in\mathrm{Aut}(\mathcal{T}_0)$, so that $x_1=(0,1,2)$, $y_1=(0,1)$ and $z_1=\mathrm{id}$. These formulas realise the restricted dynamics: by Lemma~\ref{lem:return} below, the cycles of $x_n$ are obtained by tripling those of $x_{n-1}$, in agreement with $[s,k]\to[s,3k]$ (in fact $x_n$ is the adding machine of Example~\ref{Example 4.18}); the $2$-cycle of the top permutation of $y_n$ has return section $y_{n-1}$ and its fixed point has return section $x_{n-1}$, in agreement with $[n,k]\to[n,2k][s,k]$; and $z_n$ fixes the first level with sections $y_{n-1},y_{n-1},x_{n-1}$, which corresponds to $[n,1]^2[s,1]$.

We embed $\mathrm{Aut}(\mathcal{T}_{n-1})$ into $\mathrm{Aut}(\mathcal{T}_n)$ as the automorphisms acting only below the vertex $0$:
\begin{equation}\label{rem:embed}
       \mathrm{Aut}(\mathcal{T}_{n-1})\hookrightarrow \mathrm{Aut}(\mathcal{T}_n),\qquad \sigma \mapsto (\sigma, \mathrm{id}, \mathrm{id}).
\end{equation}
With this embedding we define recursively
$$M_n:=\langle x_n,y_n,z_n,L_{n-1}\rangle ,\qquad L_{n}:=\langle x_n,z_n,L_{n-1}\rangle ,\qquad H_n:=\langle L_{n-1}\rangle ^{L_n},$$
where $L_0$ is trivial and $\langle L_{n-1}\rangle^{L_n}$ denotes the normal closure of $L_{n-1}$ in $L_n$. The idea is that $L_n$ should account for the root type $s$ and $M_n$ for both root types, while $H_n$ collects the elements that come from three independent copies of level $n-1$.
Note that $M_n=\langle L_n,y_n\rangle $.
The rest of this section shows that $M_n$ follows the Markov model. The strategy, which we use again in Sections~\ref{4}, \ref{sec:three} and \ref{sec:refined-two}, has three steps. First we show that $M_n/L_n$ has order $2$, so that $L_n$ and the other coset of $L_n$ will account for the two root types $s$ and $n$ (Theorem~\ref{TH 4.13}). Then we show that $H_n$ consists of the elements of $\mathrm{St}_1$ all of whose sections lie in $L_{n-1}$, and we compute the index $[L_n:H_n]$ (Theorems~\ref{th 4.22} and \ref{th 4.23}). Finally, we compute the cycle data of $M_n$ coset by coset modulo $H_n$, using the cycle data of the cosets of $L_{n-1}$ at the previous level (Theorem~\ref{th 4.24}).
\begin{theorem}\label{TH 4.13}
    For every $n\geq 1$ we have $L_n\unlhd M_n$, and $Q_n:=M_n/L_n$ is a group of order $2$ generated by the image $\overline{y_n}$ of $y_n$. In particular $[M_n:L_n]=2$.
\end{theorem}
\begin{lemma}\label{lem:psi-one}
    Let $n\geq 2$ and assume that Theorem~\ref{TH 4.13} holds at level $n-1$. Every element of $M_n$ has all its first-level sections in $M_{n-1}$, and $L_n$ contains $L_{n-1}\times L_{n-1}\times L_{n-1}$, that is, every element $(a_0,a_1,a_2)$ with $a_0,a_1,a_2\in L_{n-1}$.
\end{lemma}
\begin{proof}
    The generators of $M_n$ are $x_n$, $y_n$, $z_n$ and the elements $(g,\mathrm{id},\mathrm{id})$ with $g\in L_{n-1}$, and their sections at the first level are $\mathrm{id}$, $x_{n-1}$, $y_{n-1}$ and elements of $L_{n-1}$, which all lie in $M_{n-1}$. By the multiplication rule \eqref{eq:multiplication}, the sections of a product are products of sections, so the first statement holds for every element of the finite group $M_n$.

    The sections of $x_n$ and of $x_n^2$ at the vertex $0$ are trivial, so for $g\in L_{n-1}$ conjugation by $x_n$ and $x_n^2$ moves $(g,\mathrm{id},\mathrm{id})$ to the other two vertices of the first level (see Subsection~\ref{1.1}):
    $$(g,\mathrm{id},\mathrm{id})^{x_n}=(\mathrm{id},g,\mathrm{id}),\qquad (g,\mathrm{id},\mathrm{id})^{x_n^2}=(\mathrm{id},\mathrm{id},g).$$
    These elements lie in $L_n$, and they generate $L_{n-1}\times L_{n-1}\times L_{n-1}$.
\end{proof}
We now prove the theorem.
\begin{proof}[Proof of Theorem~\ref{TH 4.13}]
    We use induction on $n$. For $n=1$ we have $L_1=\langle (0,1,2)\rangle =A_3$, which is normal of index $2$ in $M_1=S_3$, and $Q_1$ is generated by $\overline{y_1}=\overline{(0,1)}$. Let $n\geq 2$, and write $x=x_{n-1}$, $y=y_{n-1}$ and $L=L_{n-1}$. By induction, $L$ is normal in $M_{n-1}$ and $y^2\in L$. All the generators of $L_n$ have even top permutation and $y_n$ has odd top permutation, so $y_n\notin L_n$.

    We first show that $y_n^2\in L_n$. By the multiplication rule,
    $$y_n^2=(y,y,x^2)=z_n\,(\mathrm{id},\mathrm{id},x),$$
    and $(\mathrm{id},\mathrm{id},x)\in L_n$ by Lemma~\ref{lem:psi-one}.

    We now prove normality on generators. By the multiplication rule, $y_n^{-1}=(y^{-1},\mathrm{id},x^{-1})(0,1)$. For $a_0,a_1,a_2\in L$, conjugation by $y_n$ moves the sections along $(0,1)$ (see Subsection~\ref{1.1}), so
    $$(a_0,a_1,a_2)^{y_n}=(y^{-1}a_1y,\ a_0,\ x^{-1}a_2x),$$
    which lies in $L\times L\times L$, since $L$ is normal in $M_{n-1}$. In the same way $z_n^{y_n}=(y^{-1}yy,\,y,\,x^{-1}xx)=z_n$. Finally,
    $$x_n^{y_n}=(y^{-1}x,\ y,\ \mathrm{id})(0,2,1)=(y^{-1}xy^{-1},\ yx^{-1}y^{-1},\ x^{-2})\ z_n\ x_n^2,$$
    as one checks with $x_n^2=(\mathrm{id},x,x)(0,2,1)$. Here $y^{-1}xy^{-1}=(y^{-1}xy)y^{-2}\in L$, so the first factor lies in $L\times L\times L\subseteq L_n$, and therefore $x_n^{y_n}\in L_n$.

    Hence $y_n^{-1}L_ny_n\subseteq L_n$, and equality holds since $L_n$ is finite. As $M_n=\langle L_n,y_n\rangle$, the subgroup $L_n$ is normal in $M_n$, and since $y_n\notin L_n$ and $y_n^2\in L_n$, the quotient $Q_n$ has order $2$ and is generated by $\overline{y_n}$.
\end{proof}

\begin{theorem}\label{th 4.22}
    For $n\geq 2$ the following group equality holds: $H_n=L_{n-1}\times L_{n-1}^{x_n}\times L_{n-1}^{x_n^2}$. In other words, $H_n$ is the set of elements $(a_0,a_1,a_2)$ with $a_0,a_1,a_2\in L_{n-1}$.
\end{theorem}
\begin{proof}
    Write $x=x_{n-1}$, $y=y_{n-1}$, $L=L_{n-1}$ and $N:=L\times L\times L$. By Lemma~\ref{lem:psi-one}, $N\subseteq L_n$. We show that $N$ is normal in $L_n$ on generators. For $a_0,a_1,a_2\in L$ we have
    $$(a_0,a_1,a_2)^{x_n}=(x^{-1}a_2x,\ a_0,\ a_1),\qquad (a_0,a_1,a_2)^{z_n}=(y^{-1}a_0y,\ y^{-1}a_1y,\ x^{-1}a_2x),$$
    and these lie in $N$, since $L$ is normal in $M_{n-1}$ (Theorem~\ref{TH 4.13}); the other generators of $L_n$ lie in $N$. Since $N$ is normal in $L_n$ and contains the embedded copy of $L$, it contains $H_n$. Conversely, $N$ is generated by $L$ and its conjugates by $x_n$ and $x_n^2$, so it is contained in $H_n$.
\end{proof}

\begin{theorem}\label{th 4.23}
    Let $n\geq 2$ and $J_n:=\langle H_n,z_n\rangle ^{L_n}$. Then $J_n=L_n\cap\mathrm{St}_1=M_n\cap \mathrm{St}_1$, $[L_n:J_n]=3$, $J_n/H_n\cong V_4$ and $L_n/H_n\cong A_4$; in particular $[L_n:H_n]=12$. The cosets of $H_n$ in $J_n$ are $H_n$, $z_nH_n$, $z_n^{x_n}H_n$ and $z_n^{x_n^2}H_n$.
\end{theorem}
\begin{proof}
    Write $x=x_{n-1}$, $y=y_{n-1}$ and $L=L_{n-1}$. Conjugation by $x_n$ moves the sections along $(0,1,2)$, so
    $$z_n=(y,y,x),\qquad z_n^{x_n}=(x,y,y),\qquad z_n^{x_n^2}=(x^{-1}yx,\ x,\ y).$$
    Let $J'$ be the subgroup generated by $H_n$ and these three elements. Since $x_n^3=(x,x,x)\in H_n$ and $H_n$ is normal in $L_n$, conjugation by $x_n$ preserves $J'$. Hence $J'$ is normal in $L_n=\langle x_n,z_n,H_n\rangle$, and therefore $J_n=J'$. All the elements of $J_n$ fix the first level, and $L_n$ is the union of $J_n$, $J_nx_n$ and $J_nx_n^2$, whose elements have top permutations $\mathrm{id}$, $(0,1,2)$ and $(0,2,1)$. This gives $J_n=L_n\cap\mathrm{St}_1$ and $[L_n:J_n]=3$. Moreover, the elements of $M_n$ outside $L_n$ have odd top permutation (Theorem~\ref{TH 4.13}), so $M_n\cap\mathrm{St}_1=J_n$.

    Let $\theta\colon J_n\to(M_{n-1}/L)^3$ be the map sending an element to the classes modulo $L$ of its three sections (Lemma~\ref{lem:psi-one}). Elements of $J_n$ have trivial top permutation, so $\theta$ is a homomorphism by the multiplication rule, and its kernel is $H_n$ by Theorem~\ref{th 4.22}. Since $x\in L$ and $M_{n-1}/L$ has order $2$, the elements $z_n$, $z_n^{x_n}$ and $z_n^{x_n^2}$ are sent to the three vectors $(\overline{y},\overline{y},1)$, $(1,\overline{y},\overline{y})$ and $(\overline{y},1,\overline{y})$. Together with the trivial vector they form a group isomorphic to $V_4$, which is therefore the image of $\theta$. Hence $J_n/H_n\cong V_4$, with the cosets of the statement, and $[L_n:H_n]=3\cdot 4=12$. Finally, $x_n$ permutes the three non-trivial cosets cyclically, so the group $L_n/H_n$ of order $12$ has a normal subgroup $V_4$ on which an element of order $3$ acts non-trivially, that is, $L_n/H_n\cong A_4$.
\end{proof}

We are now ready to compare the cycle data of these groups with the Markov process. We first recall~\cite[Definition 3.16]{article}.
\begin{definition}
    For subsets $S_1\subseteq S_2\subseteq \mathrm{Aut}(\mathcal{T}_n)$, the \emph{cycle data of $S_1$ relative to $S_2$}, denoted $CD(S_1,S_2)$, is the collection of pairs $(c,q_c)$, where $c$ runs over the cycle types on level $n$ of the elements of $S_1$ and
    $$q_c=\frac{\#\{\sigma\in S_1 : \sigma\text{ has cycle type } c\}}{|S_2|}.$$
    In other words, $CD(S_1,S_2)$ is the distribution of the cycle type of a uniformly chosen element of $S_1$, with all weights multiplied by $|S_1|/|S_2|$. When we compare cycle data with the data of the Markov process, we forget the types of the latter.
\end{definition}
We use the following operations on cycle data, and on data of the Markov process after forgetting the types. Let $A$ and $B$ be data and $\lambda\geq 0$. Then $\lambda A$ is obtained from $A$ by multiplying all weights by $\lambda$; $A\sqcup B$ is the union of $A$ and $B$, where the weights of equal partitions are added; $A\times B$ consists of the pairs $(c\cup c',qq')$ for $(c,q)\in A$ and $(c',q')\in B$, where $c\cup c'$ is the partition formed by the parts of $c$ and of $c'$ (again adding the weights of equal partitions); and $tA$, respectively $dA$, is obtained from $A$ by multiplying every part of every partition by $3$, respectively $2$. For instance, if $A$ and $B$ describe two independent random elements, then $A\times B$ describes the pair of them, and by Lemma~\ref{lem:return} below, $tA$ and $dA$ describe the effect of tripling and doubling (Definition~\ref{def 5.2}). We also write $A^{\times 3}=A\times A\times A$.

To compute cycle data we use the following notion. Let $\sigma=(a_0,a_1,a_2)\pi\in\mathrm{Aut}(\mathcal{T}_n)$ and let $C=(i_1,\dots,i_r)$ be a cycle of $\pi$, so that $(i_j)\pi=i_{j+1}$ with indices modulo $r$. We call
$$b_C:=a_{i_1}a_{i_2}\cdots a_{i_r}\in\mathrm{Aut}(\mathcal{T}_{n-1})$$
the \emph{return section} of $\sigma$ along $C$. By the multiplication rule \eqref{eq:multiplication}, it is the section of $\sigma^r$ at the vertex $i_1$: it describes what happens below $i_1$ when we go once around the cycle $C$. Choosing another starting point of $C$ replaces $b_C$ by a conjugate.
\begin{lemma}\label{lem:return}
    Let $\sigma$, $C$ and $b_C$ be as above. The cycles of $\sigma$ on the vertices of level $n$ lying below the vertices of $C$ are in bijection with the cycles of $b_C$ on level $n-1$, and each cycle of $\sigma$ is $r$ times as long as the corresponding cycle of $b_C$. Hence the cycle type of $\sigma$ is obtained by multiplying by $|C|$ the lengths of the cycles of $b_C$, for every cycle $C$ of $\pi$, and taking the union.
\end{lemma}
\begin{proof}
    Since $(i_jw)\sigma=i_{j+1}\,(w)a_{i_j}$, the vertex $(i_1w)\sigma^k$ lies below $i_1$ if and only if $r$ divides $k$, and $(i_1w)\sigma^{r}=i_1\,(w)b_C$. Hence every cycle of $\sigma$ below $C$ meets the subtree below $i_1$ in a cycle of $b_C$ (identifying $i_1w$ with $w$), and its length is $r$ times the length of that cycle.
\end{proof}
The induction begins with the three embedded copies $H_n=L_{n-1}^3$: the induction hypothesis gives the distribution in each copy. The quotient $M_n/H_n$, generated by the images of $x_n$, $y_n$ and $z_n$, then tells us which triples of section types and which root permutations occur. The next lemma computes the contribution of each such coset from its return sections.
\begin{lemma}\label{lem:uniform}
    Let $n\geq 2$ and $\sigma=(g_0,g_1,g_2)\pi\in M_n$.
    \begin{enumerate}[i)]
        \item If $\pi=\mathrm{id}$, then $CD(H_n\sigma,M_n)=\frac{1}{3}\,CD(g_0L_{n-1},M_{n-1})\times CD(g_1L_{n-1},M_{n-1})\times CD(g_2L_{n-1},M_{n-1})$.
        \item If $\pi$ is a $3$-cycle, then $CD(H_n\sigma,M_n)=\frac{1}{12}\,t\,CD(g_0g_1g_2L_{n-1},M_{n-1})$.
        \item If $\pi$ is a transposition exchanging $i$ and $j$ and fixing $k$, then $$CD(H_n\sigma,M_n)=\tfrac{1}{6}\,d\,CD(g_ig_jL_{n-1},M_{n-1})\times CD(g_kL_{n-1},M_{n-1}).$$
    \end{enumerate}
\end{lemma}
\begin{proof}
By Theorem~\ref{th 4.22}, $H_n=L_{n-1}\times L_{n-1}\times L_{n-1}$, so
$$H_n\sigma=\{(h_0g_0,h_1g_1,h_2g_2)\pi\mid h_0,h_1,h_2\in L_{n-1}\},$$
and $(h_0,h_1,h_2)\mapsto(h_0g_0,h_1g_1,h_2g_2)\pi$ is a bijection from $L_{n-1}^3$ onto $H_n\sigma$. Choosing an element of $H_n\sigma$ uniformly at random therefore amounts to choosing $h_0,h_1,h_2$ uniformly and independently in $L_{n-1}$. Let $C=(i_1,\dots,i_r)$ be a cycle of $\pi$. The return section along $C$ of the corresponding element is
    $$b_C=h_{i_1}g_{i_1}\,h_{i_2}g_{i_2}\cdots h_{i_r}g_{i_r},$$
    which lies in the coset $g_{i_1}\cdots g_{i_r}L_{n-1}$, because $L_{n-1}$ is normal in $M_{n-1}$ (Theorem~\ref{TH 4.13}). In fact $b_C$ is uniformly distributed in this coset: once $h_{i_2},\dots,h_{i_r}$ are fixed, $h_{i_1}\mapsto b_C$ is a bijection from $L_{n-1}$ onto the coset. Moreover, the return sections along different cycles of $\pi$ involve different coordinates $h_i$, so they are independent.

    By Lemma~\ref{lem:return}, the cycle type of an element of $H_n\sigma$ is obtained from the cycle types of its return sections $b_C$ by multiplying their parts by $|C|$ and taking the union over the cycles $C$ of $\pi$. Hence the distribution of cycle types in $H_n\sigma$ is the following: in case i), the product of the distributions in $g_0L_{n-1}$, $g_1L_{n-1}$ and $g_2L_{n-1}$; in case ii), the distribution in $g_{i_1}g_{i_2}g_{i_3}L_{n-1}$ with all parts multiplied by $3$, where $g_{i_1}g_{i_2}g_{i_3}L_{n-1}=g_0g_1g_2L_{n-1}$ because $M_{n-1}/L_{n-1}$ is abelian; and in case iii), the product of the distribution in $g_ig_jL_{n-1}$ with all parts multiplied by $2$ and the distribution in $g_kL_{n-1}$.

    It remains to compare the normalisations. On the one hand, $|H_n\sigma|/|M_n|=\frac{1}{24}$ by Theorems~\ref{TH 4.13} and \ref{th 4.23}. On the other hand, $CD(gL_{n-1},M_{n-1})$ is the distribution in $gL_{n-1}$ multiplied by $|gL_{n-1}|/|M_{n-1}|=\frac{1}{2}$. This gives the constants $\frac{1}{24}\cdot2^3=\frac{1}{3}$, $\frac{1}{24}\cdot2=\frac{1}{12}$ and $\frac{1}{24}\cdot2^2=\frac{1}{6}$.
\end{proof}
We now attach types to the elements of $M_n$. For $\sigma\in M_n$ we put $\tau(\sigma)=s$ if $\sigma\in L_n$ and $\tau(\sigma)=n$ if $\sigma\in L_ny_n$, so that $\tau\colon M_n\to\{s,n\}$ is the quotient map $M_n\to Q_n$ followed by the isomorphism $Q_n\cong\{s,n\}$ sending $\overline{y_n}$ to $n$; we also identify the signs $+1,-1$ with $s,n$. Let $n\geq 2$ and let $\sigma=(a_0,a_1,a_2)\pi\in M_n$, so that $a_0,a_1,a_2\in M_{n-1}$ by Lemma~\ref{lem:psi-one}. Then
\begin{equation}\label{eq:types-one}
    \tau(a_0)\tau(a_1)\tau(a_2)=\mathrm{sgn}(\pi)=\tau(\sigma).
\end{equation}
Indeed, the generators of $L_n$ have even top permutation and $y_n$ has odd top permutation, which gives the second equality. For the first one, the multiplication rule shows that both sides are multiplicative in $\sigma$, because $\tau$ is a homomorphism on $M_{n-1}$ and $\pi$ only permutes the sections, so we only have to check the generators. The types of the sections of $x_n$, $z_n$, $(g,\mathrm{id},\mathrm{id})$ with $g\in L_{n-1}$, and $y_n$ are $(s,s,s)$, $(n,n,s)$, $(s,s,s)$ and $(s,n,s)$, and the top permutations of the first three are even, while that of $y_n$ is odd.

For combined critical orbit of length one the action of $f$ on types of Definition~\ref{def 4.8} is $f(c)=c$, so \eqref{eq:types-one} says that the types of the sections multiply to $f(\tau(\sigma))$, as in the Markov model. We introduce some notation as in~\cite{article}.
\begin{align*}
 A_1^{(n)}&:=\textit{the cycle data obtained by applying the Markov process $n-1$ times to } ([s,1]^3,\tfrac{1}{12});\\
 A_2^{(n)}&:=\textit{the cycle data obtained by applying the Markov process $n-1$ times to } ([s,3],\tfrac{2}{3});\\
 A_3^{(n)}&:=\textit{the cycle data obtained by applying the Markov process $n-1$ times to } ([n,1]^2[s,1],\tfrac{1}{4});\\
 A_{4}^{(n)}&:=\textit{the cycle data obtained by applying the Markov process $n-1$ times to } ([n,1],\tfrac{1}{2});\\
 A_5^{(n)}&:=\textit{the cycle data obtained by applying the Markov process $n-1$ times to } ([s,1],\tfrac{1}{2}).
\end{align*}
We also write $S^{(n)}:=2A_5^{(n)}$ and $N^{(n)}:=2A_4^{(n)}$ for the cycle data obtained by applying the Markov process $n-1$ times to $([s,1],1)$, respectively $([n,1],1)$.
\begin{theorem}\label{th 4.24}
    Let $n\geq 1$ and $c\in\{s,n\}$. If $\sigma$ is chosen uniformly in $\tau^{-1}(c)$, that is, in $L_n$ if $c=s$ and in $L_ny_n$ if $c=n$, then the cycle type of $\sigma$ is distributed as the data obtained by applying the Markov process of Definition~\ref{Definition 4.9} $n$ times to $[c,1]$. In other words,
    $$CD(L_n,M_n)=A_5^{(n+1)},\qquad CD(L_ny_n,M_n)=A_4^{(n+1)}.$$
    Moreover, for $n\geq 2$:
    \begin{enumerate}[i)]
        \item $A_1^{(n)}=CD(H_n,L_n);$
        \item $A_2^{(n)}=CD(J_nx_{n}\sqcup J_{n}x_{n}^2,L_n);$
        \item $A_3^{(n)}=CD(z_{n}H_n\sqcup z_{n}^{x_{n}}H_n\sqcup z_{n}^{x_n^2}H_n,L_n)$.
    \end{enumerate}
\end{theorem}
In words, a uniform element of $L_n$ behaves like a factor of type $s$, and a uniform element of the other coset $L_ny_n$ like a factor of type $n$. Items i)--iii) say more precisely which elements of $L_n$ correspond to the three possible outcomes of the first step from $[s,1]$. The proof is by induction on $n$: we split $L_n$ and $L_ny_n$ into the $24$ cosets of $H_n$, compute the cycle data of each coset with Lemma~\ref{lem:uniform} and the induction hypothesis, and compare the result with the first step of the Markov process.
\begin{proof}
    \emph{Step~1: the Markov process.} We first show that for every $n\geq 1$
    \begin{equation}\label{eq:markov-side}
        A_1^{(n)}=\tfrac{1}{12}\,S^{(n)}\times S^{(n)}\times S^{(n)},\qquad A_2^{(n)}=\tfrac{2}{3}\,tS^{(n)},\qquad A_3^{(n)}=\tfrac{1}{4}\,N^{(n)}\times N^{(n)}\times S^{(n)},
    \end{equation}
    \begin{equation}\label{eq:markov-step}
        A_5^{(n+1)}=\tfrac{1}{2}\bigl(A_1^{(n)}\sqcup A_2^{(n)}\sqcup A_3^{(n)}\bigr),\qquad A_4^{(n+1)}=\tfrac{1}{4}\,dN^{(n)}\times S^{(n)}\ \sqcup\ \tfrac{1}{4}\,dS^{(n)}\times N^{(n)}.
    \end{equation}
    In the Markov process of Definition~\ref{Definition 4.9} different states evolve independently, so applying the process to a datum with several states gives the product of the data obtained from each of them. Moreover, the transitions only multiply the degrees, so applying the process to $[c,3]$, respectively $[c,2]$, gives the data obtained from $[c,1]$ with every part multiplied by $3$, respectively $2$. This proves \eqref{eq:markov-side}. The first step from $[s,1]$ gives $[s,3]$, $[s,1]^3$ and $[s,1][n,1]^2$ with probabilities $\frac{2}{3}$, $\frac{1}{12}$ and $\frac{1}{4}$, and the first step from $[n,1]$ gives $[n,2][s,1]$ and $[s,2][n,1]$ with probability $\frac{1}{2}$ each; this proves \eqref{eq:markov-step}.
 
    \emph{Step~2: the case $n=1$.} We have $S^{(1)}=N^{(1)}=\{([1],1)\}$, so by \eqref{eq:markov-side} and \eqref{eq:markov-step}, $$A_5^{(2)}=\{([1,1,1],\tfrac{1}{6}),([3],\tfrac{1}{3})\}\qquad\text{and}\qquad A_4^{(2)}=\{([2,1],\tfrac{1}{2})\}.$$ On the other hand, $L_1=A_3$ consists of the identity and two $3$-cycles, and $L_1y_1$ consists of the three transpositions of $M_1=S_3$, so $CD(L_1,M_1)=A_5^{(2)}$ and $CD(L_1y_1,M_1)=A_4^{(2)}$.
 
    \emph{Step~3: the cosets of $H_n$.}
    Let $n\geq 2$, and write $x=x_{n-1}$, $y=y_{n-1}$ and $L=L_{n-1}$. We assume that the theorem holds at level $n-1$, that is, $CD(L,M_{n-1})=A_5^{(n)}$ and $CD(yL,M_{n-1})=A_4^{(n)}$.
    We decompose $L_n$ and $L_ny_n$ into cosets of $H_n$, whose cycle data are controlled by Lemma~\ref{lem:uniform}. Since $M_{n-1}=L\sqcup yL$, the induction hypothesis says that
    \begin{equation}\label{eq:coset-L}
        CD(gL,M_{n-1})=\begin{cases}\tfrac{1}{2}\,S^{(n)} & \text{if } g\in L,\\[2pt] \tfrac{1}{2}\,N^{(n)} & \text{if } g\in yL.\end{cases}
    \end{equation}
    By Theorem~\ref{th 4.22} and the multiplication rule, two elements of $M_n$ lie in the same coset of $H_n$ if and only if they have the same top permutation and their sections at each vertex of the first level lie in the same coset of $L$. Hence every coset of $H_n$ in $M_n$ has the form $H_n(u_0,u_1,u_2)\pi$ with $u_i\in\{\mathrm{id},y\}$, and by \eqref{eq:types-one} we have $\tau(u_0)\tau(u_1)\tau(u_2)=\mathrm{sgn}(\pi)$. There are $6\cdot 4=24$ such expressions, and $[M_n:H_n]=24$ by Theorems~\ref{TH 4.13} and \ref{th 4.23}, so all of them occur. Since $\tau(\sigma)=\mathrm{sgn}(\pi)$ for $\sigma\in M_n$, we obtain:
    \begin{itemize}
        \item $L_n$ is the union of the $12$ cosets $H_n(u_0,u_1,u_2)\pi$ with $\pi\in A_3$ and
        $$(u_0,u_1,u_2)\in\{(\mathrm{id},\mathrm{id},\mathrm{id}),(y,y,\mathrm{id}),(y,\mathrm{id},y),(\mathrm{id},y,y)\};$$
        \item $L_ny_n$ is the union of the $12$ cosets $H_n(u_0,u_1,u_2)\pi$ with $\pi$ a transposition and
        $$(u_0,u_1,u_2)\in\{(y,\mathrm{id},\mathrm{id}),(\mathrm{id},y,\mathrm{id}),(\mathrm{id},\mathrm{id},y),(y,y,y)\}.$$
    \end{itemize}
 
    \emph{Step~4: the cosets in $L_n$.} We use that $CD(X,L_n)=2\,CD(X,M_n)$ for $X\subseteq L_n$, since $|M_n|=2|L_n|$.
    \begin{enumerate}[i)]
        \item By Lemma~\ref{lem:uniform} i) and \eqref{eq:coset-L}, $CD(H_n,M_n)=\frac{1}{3}\bigl(\frac{1}{2}S^{(n)}\bigr)^{\times 3}=\frac{1}{24}\,S^{(n)}\times S^{(n)}\times S^{(n)}$, so $CD(H_n,L_n)=\frac{1}{12}\,S^{(n)}\times S^{(n)}\times S^{(n)}=A_1^{(n)}$ by \eqref{eq:markov-side}.
        \item By Theorem~\ref{th 4.23}, $J_n=L_n\cap\mathrm{St}_1$ has index $3$ in $L_n$, so $J_nx_n\sqcup J_nx_n^2$ is the set of elements of $L_n$ whose top permutation is a $3$-cycle, that is, the union of the $8$ cosets of Step~3 with $\pi\in\{(0,1,2),(0,2,1)\}$. For each of them $u_0u_1u_2\in\{\mathrm{id},y^2\}\subseteq L$; for instance, for $\sigma=(y,y,\mathrm{id})(0,1,2)$ we have $\sigma^3=(y^2,y^2,y^2)$. By Lemma~\ref{lem:uniform} ii) and \eqref{eq:coset-L}, each of these cosets has $CD(H_n\sigma,M_n)=\frac{1}{12}\,t\bigl(\frac{1}{2}S^{(n)}\bigr)=\frac{1}{24}\,tS^{(n)}$. Hence $CD(J_nx_n\sqcup J_nx_n^2,L_n)=2\cdot 8\cdot\frac{1}{24}\,tS^{(n)}=\frac{2}{3}\,tS^{(n)}=A_2^{(n)}$.
        \item By Theorem~\ref{th 4.23}, $z_nH_n$, $z_n^{x_n}H_n$ and $z_n^{x_n^2}H_n$ are the cosets of $H_n$ in $J_n$ different from $H_n$, that is, the cosets $H_n(u_0,u_1,u_2)$ with $(u_0,u_1,u_2)\in\{(y,y,\mathrm{id}),(y,\mathrm{id},y),(\mathrm{id},y,y)\}$; for instance $z_n=(y,y,x)\in H_n(y,y,\mathrm{id})$. By Lemma~\ref{lem:uniform} i) and \eqref{eq:coset-L}, each of them has cycle data $\frac{1}{3}\cdot\frac{1}{8}\,N^{(n)}\times N^{(n)}\times S^{(n)}=\frac{1}{24}\,N^{(n)}\times N^{(n)}\times S^{(n)}$ relative to $M_n$. Hence $CD(z_nH_n\sqcup z_n^{x_n}H_n\sqcup z_n^{x_n^2}H_n,L_n)=\frac{1}{4}\,N^{(n)}\times N^{(n)}\times S^{(n)}=A_3^{(n)}$.
    \end{enumerate}
    Since $L_n$ is the disjoint union of the sets in i), ii) and iii), we get $CD(L_n,M_n)=\frac{1}{2}\bigl(A_1^{(n)}\sqcup A_2^{(n)}\sqcup A_3^{(n)}\bigr)=A_5^{(n+1)}$ by \eqref{eq:markov-step}.
 
    \emph{Step~5: the cosets in $L_ny_n$.} For $\pi=(0,1)$, the return sections of $\sigma=(u_0,u_1,u_2)(0,1)$ are $u_0u_1$ along the $2$-cycle and $u_2$ at the fixed point, as $\sigma^2=(u_0u_1,u_1u_0,u_2^2)$. By Lemma~\ref{lem:uniform} iii) and \eqref{eq:coset-L} we obtain:
    The four possible triples of section signs for this transposition are $(n,s,s)$, $(s,n,s)$, $(s,s,n)$ and $(n,n,n)$. In the first two cases the doubled return section has type $n$ and the fixed one type $s$; in the last two cases they have types $s$ and $n$, respectively (since $n^2=s$). Each coset contributes $\frac{1}{24}$ times $dN^{(n)}\times S^{(n)}$ or $dS^{(n)}\times N^{(n)}$ to $CD(\cdot,M_n)$, by Lemma~\ref{lem:uniform} iii) and \eqref{eq:coset-L}.
    For $\pi=(0,2)$ and $\pi=(1,2)$ we obtain the same after relabelling the vertices. Hence
    $$CD(L_ny_n,M_n)=6\cdot\tfrac{1}{24}\,dN^{(n)}\times S^{(n)}\ \sqcup\ 6\cdot\tfrac{1}{24}\,dS^{(n)}\times N^{(n)}=\tfrac{1}{4}\,dN^{(n)}\times S^{(n)}\ \sqcup\ \tfrac{1}{4}\,dS^{(n)}\times N^{(n)}=A_4^{(n+1)}$$
    by \eqref{eq:markov-step}. This completes the induction.
\end{proof}
 
This allows us to propose the following models for $f_t$, depending on whether $f_t$ is irreducible over $K$ and on whether $P_1(t):=(f(\gamma_1)-t)(f(\gamma_2)-t)$ is a square in $K$. The initial data are the first-level data predicted by the model. If $P_1(t)$ is a square in $K$, the root type of $f_t$ is $s$ at every prime; otherwise, by the Chebotarev density theorem applied to $K(\sqrt{P_1(t)})$, it is $s$ or $n$ with density $\frac{1}{2}$ each. We write $M_{n,t}$ for the group of the model of $f_t$ at level $n$, and we put $J_1=\{\mathrm{id}\}$.
\begin{enumerate}
    \item Model~1: If $f_t$ is reducible and $P_1(t)$ is a square, then $f_t$ splits completely and we have initial data
    $$[[s,1]^3,\frac{1}{4}],[[s,1][n,1]^2,\frac{3}{4}],$$
    and $M_{n,t}=J_n$.
    \item Model~2: If $f_t$ is irreducible and $P_1(t)$ is a square, then we have initial data
    $$[[s,3],\frac{2}{3}],[[s,1]^3,\frac{1}{12}],[[s,1][n,1]^2,\frac{1}{4}],$$
    and $M_{n,t}=L_n$.
    \item Model~3: If $f_t$ is reducible and $P_1(t)$ is not a square, then $f_t$ is the product of a linear and an irreducible quadratic factor, and we have initial data
    $$[[s,1]^3,\frac{1}{8}],[[s,1][n,1]^2,\frac{3}{8}],[[n,2][s,1],\frac{1}{4}],[[s,2][n,1],\frac{1}{4}],$$
    and $M_{n,t}=\langle J_n,y_n\rangle $.
    \item  Model~4: If $f_t$ is irreducible and $P_1(t)$ is not a square, then we have initial data
    $$[[s,3],\frac{1}{3}],[[s,1]^3,\frac{1}{24}],[[s,1][n,1]^2,\frac{1}{8}],[[n,2][s,1],\frac{1}{4}],[[s,2][n,1],\frac{1}{4}],$$
    and $M_{n,t}=M_n$.
\end{enumerate}
In Models~1 and 3 we use that a reducible cubic has a root, and that, if $f_t=(x-r)q(x)$ with $q$ quadratic, the discriminant of $q$ equals the discriminant $-27a^2P_1(t)$ of $f_t$ (Lemma~\ref{lem:cubic-factorisation}) up to a square, so that $q$ splits if and only if $P_1(t)$ is a square. The next theorem compares these models with the groups.
\begin{theorem}\label{Th 5.19}
    For each of the Models~1--4 and every $n\geq 1$, the cycle data of the group $M_{n,t}$ are the data obtained by applying the Markov process $n-1$ times to the initial data of the model.
\end{theorem}
\begin{proof}
    For $n=1$ this is immediate, so let $n\geq 2$. For Models~2 and 4 the statement is Theorem~\ref{th 4.24}: $L_n=\tau^{-1}(s)$, and $M_n=\tau^{-1}(s)\sqcup\tau^{-1}(n)$ is the union of two cosets of the same size, while the initial data of Model~4 are the average of the first steps from $[s,1]$ and from $[n,1]$.

    For Model~1, recall from Theorem~\ref{th 4.23} that $J_n$ is the set of elements of $L_n$ with trivial top permutation. By Step~3 of the proof of Theorem~\ref{th 4.24}, $J_n$ is the union of the four cosets $H_n(u_0,u_1,u_2)$ with $(u_0,u_1,u_2)\in\{(\mathrm{id},\mathrm{id},\mathrm{id}),(y,y,\mathrm{id}),(y,\mathrm{id},y),(\mathrm{id},y,y)\}$, where $y=y_{n-1}$. Thus the types of the three sections of a uniform element of $J_n$ are uniformly distributed among the four triples of types with product $s$, and by Lemma~\ref{lem:uniform} i) and \eqref{eq:coset-L},
    $$CD(J_n,J_n)=\tfrac{1}{4}\,S^{(n)}\times S^{(n)}\times S^{(n)}\ \sqcup\ \tfrac{3}{4}\,N^{(n)}\times N^{(n)}\times S^{(n)},$$
    which is the data obtained by applying the process $n-1$ times to the initial data of Model~1.

    For Model~3, note that the top permutations of the elements of $M_n$ form all of $S_3$, and $J_n=M_n\cap\mathrm{St}_1$ is the kernel of this surjection. Hence $\langle J_n,y_n\rangle =J_n\sqcup J_ny_n$ is the set of elements of $M_n$ whose top permutation is $\mathrm{id}$ or $(0,1)$, and the two parts have the same size. The elements of $J_n$ behave as in Model~1. The elements of $J_ny_n$ form the four cosets of $H_n$ with $\pi=(0,1)$ in Step~5 of the proof of Theorem~\ref{th 4.24}, and the four cosets there show that a uniform element of $J_ny_n$ has cycle data $\frac{1}{2}\,dN^{(n)}\times S^{(n)}\sqcup\frac{1}{2}\,dS^{(n)}\times N^{(n)}$, as after the first step from $[n,1]$. Averaging the two halves gives the initial data of Model~3.
\end{proof}
\section{Markov groups for combined critical orbit of length two}\label{4}
We now extend our methods to combined critical orbits of length two. The polynomials in the list of~\cite{cubic-PCF} with non-colliding critical orbits and combined critical orbit of length two are
$$2z^3-3z^2+\frac{1}{2},\qquad -z^3+\frac{3}{2}z^2-1,\qquad -\frac{1}{4}z^3+\frac{3}{2}z+2,\qquad -\frac{1}{28}z^3-\frac{3}{4}z+\frac{7}{2},$$
and all of them satisfy $P_3=P_1$ in Definition~\ref{comb}, so that $f(c_1c_2)=c_2c_1$ in Definition~\ref{def 4.8}. As in Section~\ref{3}, $f$ is defined over a number field $K$ containing $\sqrt{-3}$, and we study $f_t^n$ for $t\in K$. As noted in Section~\ref{2}, for some polynomials not all admissible types occur: for $-z^3+\frac{3}{2}z^2-1$ and $-\frac{1}{4}z^3+\frac{3}{2}z+2$ the co-critical points lie on the combined critical orbit (Section~\ref{sec:cocritical}), and for them we construct refined groups in Section~\ref{sec:refined-two}. Here we construct the groups for the case in which all admissible types occur.

     We first compute the first-level data of $f_t$ when all four root types are possible, by averaging the first steps of the model from $[ss,1]$, $[sn,1]$, $[nn,1]$ and $[ns,1]$ (Example~\ref{Example 4.13}); the other models use the appropriate subset of root types:
    \emph{The complete first-level data for the generic case are displayed in Model~6 below.}

    If we forget the types, the patterns $[3]$, $[2,1]$ and $[1,1,1]$ occur with frequencies $\frac{1}{3}$, $\frac{1}{2}$ and $\frac{1}{6}$, as in Section~\ref{3}, so again $M_1=S_3=\langle (0,1,2),(0,1)\rangle$. The dynamics on types is described in Example~\ref{Example 4.13}, and we restrict it to one transition for each type:
    \begin{align*}
        &[ss,k]\to [ss,3k],\\
        &[nn,k]\to [nn,2k][ss,k],\\
        &[sn,k]\to [ns,3k],\\
        & [ns,k]\to [ns,2k][nn,k].
    \end{align*}
    Unlike in Section~\ref{3}, we attach two types to some elements of $S_3$: to $(0,1,2)$ we attach the type $[ss,3]$, to $(0,1)$ the types $[nn,2][ss,1]$ and $[ns,2][nn,1]$, and to $\mathrm{id}$ the types $[nn,1]^2[ss,1]$ and $[ns,1]^2[ss,1]$. Iterating with respect to the restricted dynamics, we denote by $x_n$, $y_n$, $z_n$, $l_n$ and $k_n$ the elements of $\mathrm{Aut}(\mathcal{T}_n)$ obtained from $[ss,3]$, $[nn,2][ss,1]$, $[nn,1]^2[ss,1]$, $[ns,2][nn,1]$ and $[ns,1]^2[ss,1]$, respectively. Explicitly, $x_n,y_n,z_n$ are as in Section~\ref{3}, and
    $$l_n=(\mathrm{id},l_{n-1},y_{n-1})(0,1),\qquad k_n=(l_{n-1},l_{n-1},x_{n-1}),$$
    with $l_0=k_0=\mathrm{id}$, so that $l_1=(0,1)$ and $k_1=\mathrm{id}$; by Lemma~\ref{lem:return}, $l_n$ follows $[ns,k]\to[ns,2k][nn,k]$ and $k_n$ corresponds to $[ns,1]^2[ss,1]$. 
    \begin{remark}
           The elements $x_n,y_n,z_n$ are the same as in Section~\ref{3}, so every statement of Section~\ref{3} that only concerns these elements, and not the groups generated by them, remains valid here.
    \end{remark}
We propose the following groups, where $L_0$ is trivial and $L_{n-1}$ is embedded in $\mathrm{Aut}(\mathcal{T}_n)$ as in \eqref{rem:embed}:
\begin{align*}
    M_n&=\langle L_n,y_n,l_n\rangle ,\\
    L_n&=\langle x_n,z_n,k_n,L_{n-1}\rangle .
\end{align*}

We will show that $M_n$ has the cycle data corresponding to the full first-level data of $f_t$ for a polynomial $f$ with combined critical orbit of length $2$, and that $L_n$ has the cycle data obtained by applying the process $n$ times to $[ss,1]$ (Theorems~\ref{4.44} and \ref{Th 6.9}).
The strategy is the same as in Section~\ref{3}. We expect
$[M_n : L_n] = 4$ for $n \geq 2$, since the data coming from $[ss,1]$ is
$\tfrac14$ of the initial data. For $n = 1$ this fails: $y_1 = l_1 = (0,1)$,
so $M_1 = S_3$ and $[M_1 : L_1] = 2$. Moreover, we expect $L_n$ to contain $L_{n-1} \times L_{n-1} \times L_{n-1}$, reflecting the transition $[ss,k]\to[ss,k]^3$, after which the three states evolve independently. The model also suggests
$[L_n : H_n] = 48$, where, as in Section~\ref{3}, $H_n$ is the normal closure of $L_{n-1}$ in $L_n$, because starting from $[ss,1]$ the data $[ss,1]^3$ has probability $\tfrac{1}{3}\cdot\tfrac{1}{16}=\tfrac{1}{48}$. This value is attained for $n \geq 3$, while for
$n = 2$ the index is only $12$ (Corollary~\ref{4.42}).

We first confirm that $[M_n:L_n]=4$ for $n\geq 2$.
\begin{theorem}\label{th 4.39}
For every $n \geq 1$ we have $L_n \trianglelefteq M_n$, and the quotient
$Q_n := M_n / L_n$ is an elementary abelian $2$-group generated by the
images $\overline{y_n}$ and $\overline{l_n}$. More precisely,
$Q_1 \cong C_2$ with $\overline{y_1} = \overline{l_1}$, and for $n \geq 2$
\[
  M_n / L_n \cong V_4 = \{1,\ \overline{y_n},\ \overline{l_n},\ \overline{y_n l_n}\}.
\]
\end{theorem}
\begin{lemma}\label{lem:structureL}
Let $n\geq 2$ and assume Theorem~\ref{th 4.39} holds at level $n-1$. For an element $\sigma=(a_0,a_1,a_2)\pi$ of $\mathrm{Aut}(\mathcal{T}_n)$ with $a_0,a_1,a_2\in M_{n-1}$, we put
$$\mathrm{cl}(\sigma):=\overline{a_0}\,\overline{a_1}\,\overline{a_2}\in Q_{n-1}.$$
\begin{enumerate}[i)]
    \item Every element of $M_n$ has all its first-level sections in $M_{n-1}$, and $\mathrm{cl}(\sigma\sigma')=\mathrm{cl}(\sigma)\,\mathrm{cl}(\sigma')$ for all $\sigma,\sigma'$ as above.
    \item Every element of $L_n$ has even top permutation and $\mathrm{cl}(\sigma)=1$.
    \item Conversely, every $\sigma=(a_0,a_1,a_2)\pi$ with $a_0,a_1,a_2\in M_{n-1}$, even top permutation and $\mathrm{cl}(\sigma)=1$ lies in $L_n$. In particular, $L_n$ contains $L_{n-1}\times L_{n-1}\times L_{n-1}$.
\end{enumerate}
\end{lemma}

\begin{proof}
The proof follows that of Lemma~\ref{lem:psi-one}; the new ingredient is that $Q_{n-1}$ may have order $4$.

\emph{i)} The first-level sections of the generators are $\mathrm{id}$, $x_{n-1}$, $y_{n-1}$, $l_{n-1}$ and elements of $L_{n-1}$, which all lie in $M_{n-1}$, and the sections of a product are products of sections. By the multiplication rule \eqref{eq:multiplication}, the sections of $(a_0,a_1,a_2)\pi\cdot(b_0,b_1,b_2)\rho$ are $a_ib_{(i)\pi}$. Since $Q_{n-1}$ is abelian and $\pi$ only permutes the factors $\overline{b_j}$, this gives the formula for $\mathrm{cl}$.

\emph{ii)} By i), we only have to check the generators. We have $\mathrm{cl}(x_n)=\overline{x_{n-1}}=1$, since $x_{n-1}\in L_{n-1}$, and $\mathrm{cl}(z_n)=\overline{y_{n-1}}^{\,2}\,\overline{x_{n-1}}=1$ and $\mathrm{cl}(k_n)=\overline{l_{n-1}}^{\,2}\,\overline{x_{n-1}}=1$, because $Q_{n-1}$ has exponent $2$. Clearly $\mathrm{cl}((g,\mathrm{id},\mathrm{id}))=1$ for $g\in L_{n-1}$, and all these generators have even top permutation.

\emph{iii)} As in Lemma~\ref{lem:psi-one}, conjugation by $x_n$ shows that $L_n$ contains $L_{n-1}\times L_{n-1}\times L_{n-1}$. Let $V:=\{v\in Q_{n-1}^3 : v_0 v_1 v_2 = 1\}$. The classes modulo $L_{n-1}$ of the sections of $z_n$ and $k_n$ are $(\overline{y_{n-1}},\overline{y_{n-1}},1)$ and $(\overline{l_{n-1}},\overline{l_{n-1}},1)$. Conjugation by $x_n$ permutes the three classes cyclically, since conjugation by the section $x_{n-1}$ is trivial in the abelian group $Q_{n-1}$. Since $Q_{n-1}$ is an elementary abelian $2$-group generated by $\overline{y_{n-1}}$ and $\overline{l_{n-1}}$ (Theorem~\ref{th 4.39} at level $n-1$), the cyclic shifts of these two vectors generate $V$. So, for every $v\in V$, there is an element of $L_n$ with trivial top permutation whose sections have classes $v_0,v_1,v_2$.

Now let $\sigma$ be as in iii). Multiplying by a power of $x_n$, which does not change $\mathrm{cl}$, we may assume that its top permutation is trivial. Then the classes of its sections form an element of $V$. Let $\omega\in L_n$ be as above, for this element of $V$. Then $\sigma\omega^{-1}$ has trivial top permutation and all its sections lie in $L_{n-1}$, so it lies in $L_n$, and hence $\sigma\in L_n$.
\end{proof}

We proceed to prove Theorem~\ref{th 4.39}.
\begin{proof}[Proof of Theorem~\ref{th 4.39}]
We argue by induction on $n$. For $n = 1$ we have $L_1 = \langle (0,1,2) \rangle
\cong A_3$ and $M_1 = S_3$, so $L_1 \trianglelefteq M_1$ and
$M_1 / L_1 \cong C_2$, generated by $\overline{y_1} = \overline{l_1} =
\overline{(0,1)}$.

Let $n \geq 2$ and assume the statement at level $n-1$. We write $x=x_{n-1}$, $y=y_{n-1}$ and $l=l_{n-1}$. We prove normality on generators, as in Theorem~\ref{TH 4.13}. Let $g$ be a generator of $L_n$ and let $w\in\{y_n,l_n\}$. By Lemma~\ref{lem:structureL} i) and ii), $g^w=w^{-1}gw$ has even top permutation and $\mathrm{cl}(g^w)=\mathrm{cl}(g)=1$, so $g^w\in L_n$ by Lemma~\ref{lem:structureL} iii). Explicitly, $z_n^{y_n}=z_n$ and
\begin{align*}
    x_n^{y_n}&=(y^{-1}x,\ y,\ \mathrm{id})(0,2,1), & x_n^{l_n}&=(l^{-1}y,\ l,\ y^{-1}x)(0,2,1),\\
    k_n^{y_n}&=(y^{-1}ly,\ l,\ x), & z_n^{l_n}&=(l^{-1}yl,\ y,\ y^{-1}xy),\\
    k_n^{l_n}&=(l,\ l,\ y^{-1}xy), & (a_0,a_1,a_2)^{y_n}&=(y^{-1}a_1y,\ a_0,\ x^{-1}a_2x).
\end{align*}
Hence $L_n$ is normal in $M_n=\langle L_n,y_n,l_n\rangle$.

Next, $y_n^2$, $l_n^2$ and the commutator of $y_n$ and $l_n$ have even top permutation and, by Lemma~\ref{lem:structureL} i), $\mathrm{cl}=1$, since $Q_{n-1}$ is an elementary abelian $2$-group. So they lie in $L_n$, and $M_n/L_n$ is an elementary abelian $2$-group generated by $\overline{y_n}$ and $\overline{l_n}$. From $y_n = (\mathrm{id}, y, x)(0,1)$ and $l_n = (\mathrm{id}, l, y)(0,1)$ we get
\[
  \mathrm{cl}(y_n) = \overline{y},
  \qquad
  \mathrm{cl}(l_n) = \overline{l}\,\overline{y},
  \qquad
  \mathrm{cl}(y_nl_n)=\overline{l}.
\]
The elements $y_n$ and $l_n$ do not lie in $L_n$, because their top permutation is odd, and $y_nl_n$ does not lie in $L_n$ by Lemma~\ref{lem:structureL} ii), because $\overline{l}\neq 1$: for $n=2$ the element $\overline{l_1}$ generates $Q_1$, and for $n\geq 3$ it is non-trivial in $Q_{n-1}\cong V_4$. Hence $1,\overline{y_n},\overline{l_n},\overline{y_nl_n}$ are distinct, and $M_n/L_n\cong V_4$. This completes the induction.
\end{proof}

As in Section~\ref{3}, we denote by $H_n:=\langle L_{n-1}\rangle^{L_n}$ the normal closure of $L_{n-1}$ in $L_n$.
\begin{theorem}\label{thm:Hn-two}
    Let $n\geq 2$. Then $H_n=L_{n-1}\times L_{n-1}^{x_{n}}\times L_{n-1}^{x_n^2}$, that is, $H_n$ is the set of elements $(a_0,a_1,a_2)$ with $a_0,a_1,a_2\in L_{n-1}$.
\end{theorem}
\begin{proof}
    By Lemma~\ref{lem:structureL} iii), $N:=L_{n-1}\times L_{n-1}\times L_{n-1}$ is contained in $L_n$. It is normal in $L_n$ on generators: conjugation by $z_n$ and by $k_n$ acts on each coordinate by conjugation by an element of $M_{n-1}$, conjugation by $x_n$ permutes the coordinates and conjugates one of them by $x_{n-1}\in L_{n-1}$, and the other generators lie in $N$; we use that $L_{n-1}$ is normal in $M_{n-1}$ (Theorem~\ref{th 4.39}). So $N$ contains $H_n$, and conversely, it is generated by $L_{n-1}$ and its conjugates by $x_n$ and $x_n^2$.
\end{proof}
We construct an intermediate subgroup that will help us compute the index of $H_n$ in $L_n$. We define $J_n=\langle z_n,k_n, L_{n-1}\rangle ^{L_n}$.
\begin{corollary}\label{4.42}
    
For $n \geq 2$, $H_n = L_{n-1} \times L_{n-1} \times L_{n-1}$, $J_n=L_n\cap\mathrm{St}_1$ and
\[
  [L_n : H_n] = 3\,|Q_{n-1}|^2 =
  \begin{cases} 12, & n = 2,\\ 48, & n \geq 3. \end{cases}
\]
Moreover $J_n / H_n \cong \{v \in Q_{n-1}^3 : v_0 v_1 v_2 = 1\}$, which is
$V_4$ for $n = 2$ and $(\mathbb{Z}/2\mathbb{Z})^4$ for $n \geq 3$. In particular $|M_2|=4\cdot 12\cdot 3^3=6^4$, so $M_2=\mathrm{Aut}(\mathcal{T}_2)$.
    
\end{corollary}
\begin{proof}
Let $V:=\{v\in Q_{n-1}^3 : v_0v_1v_2=1\}$, and let $\theta\colon L_n\cap\mathrm{St}_1\to Q_{n-1}^3$ be the map sending an element to the classes of its three sections. By the multiplication rule $\theta$ is a homomorphism. Its kernel is $L_{n-1}^3=H_n$ (Theorem~\ref{thm:Hn-two}), and its image is $V$ by Lemma~\ref{lem:structureL} ii) and iii), so $[L_n\cap\mathrm{St}_1:H_n]=|V|=|Q_{n-1}|^2$. The top permutations of the elements of $L_n$ form the group $A_3$, by Lemma~\ref{lem:structureL} ii) and because of $x_n$, so $[L_n:L_n\cap\mathrm{St}_1]=3$. Since $|Q_1| = 2$ and $|Q_{n-1}| = 4$ for $n \geq 3$, the formula for $[L_n:H_n]$ follows. The subgroup $L_n\cap\mathrm{St}_1$ is normal in $L_n$ and contains $z_n$, $k_n$ and $L_{n-1}$, hence it contains $J_n$. Conversely, $J_n$ contains $H_n$, and the images under $\theta$ of $z_n$, $k_n$ and of their conjugates by $x_n$ generate $V$, as shown in the proof of Lemma~\ref{lem:structureL}, so $J_n=L_n\cap\mathrm{St}_1$. Finally, $|H_2|=|A_3|^3=3^3$, which gives the order of $M_2$ by Theorem~\ref{th 4.39}.
\end{proof}
\begin{lemma}\label{lem:uniform-two}
    Let $n\geq 2$, let $\sigma=(g_0,g_1,g_2)\pi\in M_n$, and write $\kappa:=[M_{n-1}:L_{n-1}]$, so that $\kappa=2$ if $n=2$ and $\kappa=4$ if $n\geq 3$ (Theorem~\ref{th 4.39}).
    \begin{enumerate}[i)]
        \item If $\pi=\mathrm{id}$, then $CD(H_n\sigma,M_n)=\frac{\kappa}{12}\,CD(g_0L_{n-1},M_{n-1})\times CD(g_1L_{n-1},M_{n-1})\times CD(g_2L_{n-1},M_{n-1})$.
        \item If $\pi$ is a $3$-cycle, then $CD(H_n\sigma,M_n)=\frac{1}{12\kappa}\,t\,CD(g_0g_1g_2L_{n-1},M_{n-1})$.
        \item If $\pi$ is a transposition exchanging $i$ and $j$ and fixing $k$, then $$CD(H_n\sigma,M_n)=\tfrac{1}{12}\,d\,CD(g_ig_jL_{n-1},M_{n-1})\times CD(g_kL_{n-1},M_{n-1}).$$
    \end{enumerate}
    In particular, the constants in i), ii) and iii) are $\frac{1}{3}$, $\frac{1}{48}$ and $\frac{1}{12}$ for $n\geq 3$, and $\frac{1}{6}$, $\frac{1}{24}$ and $\frac{1}{12}$ for $n=2$.
\end{lemma}
\begin{proof}
    By Lemma~\ref{lem:structureL}, $g_0,g_1,g_2\in M_{n-1}$; by Theorem~\ref{thm:Hn-two}, $H_n=L_{n-1}\times L_{n-1}\times L_{n-1}$; and by Theorem~\ref{th 4.39}, $L_{n-1}$ is normal in $M_{n-1}$ and $M_{n-1}/L_{n-1}$ is abelian. These are the only facts used in the proof of Lemma~\ref{lem:uniform}, which therefore applies word for word: for every cycle type, the proportion of the elements of $H_n\sigma$ with this cycle type equals the proportion of the tuples of return sections $(b_C)_C\in\prod_C g_{i_1}\cdots g_{i_r}L_{n-1}$ which give this cycle type by Lemma~\ref{lem:return}. In ii) the return section lies in $g_{i_1}g_{i_2}g_{i_3}L_{n-1}=g_0g_1g_2L_{n-1}$, since $M_{n-1}/L_{n-1}$ is abelian. Only the constants change: by Theorem~\ref{th 4.39} and Corollary~\ref{4.42} we have $[M_n:H_n]=4\cdot3\kappa^2=12\kappa^2$, and $|gL_{n-1}|/|M_{n-1}|=\frac{1}{\kappa}$ for $g\in M_{n-1}$. This gives the constants $\frac{\kappa^3}{12\kappa^2}=\frac{\kappa}{12}$, $\frac{\kappa}{12\kappa^2}=\frac{1}{12\kappa}$ and $\frac{\kappa^2}{12\kappa^2}=\frac{1}{12}$.
\end{proof}
We introduce the following notation for the cycle data.

\[
\resizebox{\textwidth}{!}{$
\begin{aligned}
A_1^{(n)} &:= \textit{the cycle data obtained by applying the Markov process $n-1$ times to } ([ss,1]^3,\tfrac{1}{48}); \\
A_2^{(n)} &:= \textit{the cycle data obtained by applying the Markov process $n-1$ times to } ([ss,3],\tfrac{2}{3}); \\
A_3^{(n)} &:= \textit{the cycle data obtained by applying the Markov process $n-1$ times to } ([nn,1]^2[ss,1],\tfrac{1}{16}); \\
A_4^{(n)} &:= \textit{the cycle data obtained by applying the Markov process $n-1$ times to } ([ns,1]^2[ss,1],\tfrac{1}{16}); \\
A_5^{(n)} &:= \textit{the cycle data obtained by applying the Markov process $n-1$ times to } ([sn,1]^2[ss,1],\tfrac{1}{16}); \\
A_6^{(n)} &:= \textit{the cycle data obtained by applying the Markov process $n-1$ times to } ([ns,1][sn,1][nn,1],\tfrac{1}{8}); \\
A_7^{(n)} &:= \textit{the cycle data obtained by applying the Markov process $n-1$ times to } ([nn,1],\tfrac{1}{4}); \\
A_8^{(n)} &:= \textit{the cycle data obtained by applying the Markov process $n-1$ times to } ([ss,1],\tfrac{1}{4}); \\
A_9^{(n)} &:= \textit{the cycle data obtained by applying the Markov process $n-1$ times to } ([sn,1],\tfrac{1}{4}); \\
A_{10}^{(n)} &:= \textit{the cycle data obtained by applying the Markov process $n-1$ times to } ([ns,1],\tfrac{1}{4}).
\end{aligned}
$}
\]
We also write $B_c^{(n)}$, for a type $c\in\{s,n\}^2$, for the cycle data obtained by applying the Markov process $n-1$ times to $([c,1],1)$. Thus $B_{ss}^{(n)}=4A_8^{(n)}$, $B_{nn}^{(n)}=4A_7^{(n)}$, $B_{sn}^{(n)}=4A_9^{(n)}$, $B_{ns}^{(n)}=4A_{10}^{(n)}$, and $B_c^{(1)}=\{([1],1)\}$ for every $c$.

As in Section~\ref{3} we attach types to the elements of $M_n$. The type character of an automorphism is a group-theoretic map; its compatibility with the quadratic-character types of irreducible polynomial factors is established by the matching recurrence below and by the discriminant identities in Section~\ref{5}. For $n\geq 2$ we let $\tau\colon M_n\to\{s,n\}^2$ be the homomorphism with kernel $L_n$ sending $y_n,l_n,y_nl_n$ to $nn,ns,sn$; by Theorem~\ref{th 4.39} it induces an isomorphism $Q_n\cong\{s,n\}^2$, where types are multiplied letter by letter (Definition~\ref{def:type-product}). These values are dictated by the restricted dynamics: $y_n$ and $l_n$ realise the transitions from $[nn,k]$ and from $[ns,k]$, so they should have the types $nn$ and $ns$ of the corresponding parents. For $n=1$ we let $\tau\colon M_1\to\{s,n\}$ be the sign, which only records the first letter. Let $n\geq 3$, and let $\sigma=(a_0,a_1,a_2)\pi\in M_n$ have type $c$. Then
\begin{equation}\label{eq:types-two}
    \mathrm{sgn}(\pi)=c_1\ \text{ and }\ \tau(a_0)\tau(a_1)\tau(a_2)=f(c),
\end{equation}
where $f(c_1c_2)=c_2c_1$ (Definition~\ref{def 4.8}, since $P_3=P_1$). Indeed, by the multiplication rule both sides of the two equalities are multiplicative in $\sigma$, since $\tau$ and $f$ are homomorphisms, so we only have to check the generators. For the generators of $L_n$, of type $ss$, this is Lemma~\ref{lem:structureL} ii). The element $y_n$ has odd top permutation and sections of types $ss,nn,ss$, with product $nn=f(nn)$, and $l_n$ has odd top permutation and sections of types $ss,ns,nn$, with product $sn=f(ns)$.

For $n=2$ the same holds if we only compare the first letters of $\tau(a_0)\tau(a_1)\tau(a_2)$ and of $f(c)$, i.e.\ $\tau(a_0)\tau(a_1)\tau(a_2)=c_2$, as the signs of the sections of $y_2$ and $l_2$ show.
    \begin{theorem}\label{4.44}
        Let $n\geq 2$ and $c\in\{s,n\}^2$. If $\sigma$ is chosen uniformly in $\tau^{-1}(c)$, that is, in $L_n$, $L_ny_n$, $L_nl_n$ or $L_ny_nl_n$ for $c=ss,nn,ns,sn$ respectively, then the cycle type of $\sigma$ is distributed as the data obtained by applying the Markov process of Definition~\ref{Definition 4.9} $n$ times to $[c,1]$. In other words,
        \begin{align*}
            &CD(L_n,M_n)=A_8^{(n+1)}, && CD(L_ny_n,M_n)=A_7^{(n+1)},\\
            &CD(L_nl_n,M_n)=A_{10}^{(n+1)}, && CD(L_ny_nl_n,M_n)=A_9^{(n+1)}.
        \end{align*}
        Moreover, for $n\geq 3$:
        \begin{enumerate}[i)]
            \item $A_1^{(n)}=CD(H_n,L_n);$
            \item $A_2^{(n)}=CD(J_nx_n\sqcup J_nx_n^2,L_n);$
            \item $A_3^{(n)}=CD(z_nH_n\sqcup z_n^{x_n}H_n\sqcup z_n^{x_n^2}H_n,L_n);$
            \item $A_4^{(n)}=CD(k_nH_n\sqcup k_n^{x_n}H_n\sqcup k_n^{x_n^2}H_n,L_n);$
            \item $A_5^{(n)}=CD(k_nz_nH_n\sqcup (k_nz_n)^{x_n}H_n\sqcup (k_nz_n)^{x_n^2}H_n,L_n);$
            \item $A_6^{(n)}=CD(z_nk_n^{x_n}H_n\sqcup z_nk_n^{x_n^2}H_n\sqcup z_n^{x_n}k_nH_n\sqcup z_n^{x_n^2}k_nH_n\sqcup z_n^{x_n}k_n^{x_n^2}H_n\sqcup z_n^{x_n^2}k_n^{x_n}H_n,L_n)$.
        \end{enumerate}
    \end{theorem}
Here the induction is carried out for the length-two groups themselves. The subgroup $H_n=L_{n-1}^3$ supplies three copies of the previous distribution; the quotient, with representatives involving $x_n$, $y_n$, $l_n$, $z_n$ and $k_n$, selects the admissible triples and their root permutations. The exceptional level $n=2$ is treated separately because $M_1/L_1$ has only two elements, whereas $M_{n-1}/L_{n-1}$ has four elements for $n\geq3$.
\begin{proof}
We first record the Markov recursions. Independence of the descendants and the fact that transitions multiply their degrees give, for $n\geq1$,
\begin{equation}\label{eq:markov-side-two}
\begin{aligned}
 A_1^{(n)}&=\tfrac1{48}(B_{ss}^{(n)})^{\times3}, & A_2^{(n)}&=\tfrac23\,tB_{ss}^{(n)},\\
 A_3^{(n)}&=\tfrac1{16}(B_{nn}^{(n)})^{\times2}\times B_{ss}^{(n)}, & A_4^{(n)}&=\tfrac1{16}(B_{ns}^{(n)})^{\times2}\times B_{ss}^{(n)},\\
 A_5^{(n)}&=\tfrac1{16}(B_{sn}^{(n)})^{\times2}\times B_{ss}^{(n)}, & A_6^{(n)}&=\tfrac18 B_{ns}^{(n)}\times B_{sn}^{(n)}\times B_{nn}^{(n)}.
\end{aligned}
\end{equation}
For $c\in\{s,n\}^2$ the one-step recursion is
\begin{equation}\label{eq:markov-step-two}
 B_c^{(n+1)}=\begin{cases}
 \tfrac23\,tB_{f(c)}^{(n)}\ \sqcup\ \tfrac1{48}\displaystyle\bigsqcup_{d_0d_1d_2=f(c)}B_{d_0}^{(n)}\times B_{d_1}^{(n)}\times B_{d_2}^{(n)}, &c_1=s,\\[10pt]
 \tfrac14\displaystyle\bigsqcup_{e_1e_2=f(c)}dB_{e_1}^{(n)}\times B_{e_2}^{(n)}, &c_1=n.
 \end{cases}
\end{equation}
The unions are over the $16$ ordered triples, respectively the four ordered pairs, of types with the indicated product. When $c=ss$, the $16$ triples are $(ss,ss,ss)$, the three permutations of each of $(nn,nn,ss)$, $(ns,ns,ss)$ and $(sn,sn,ss)$, and the six permutations of $(ns,sn,nn)$. Thus
\begin{equation}\label{eq:markov-ss}
 A_8^{(n+1)}=\tfrac14 B_{ss}^{(n+1)}=\tfrac14\bigl(A_1^{(n)}\sqcup A_2^{(n)}\sqcup\cdots\sqcup A_6^{(n)}\bigr).
\end{equation}

\emph{Base case $n=2$.} We have $M_2=\mathrm{Aut}(\mathcal{T}_2)$ and $H_2=L_1^3=A_3^3$; the $48$ cosets of $H_2$ are parametrised by $(\varepsilon_0,\varepsilon_1,\varepsilon_2;\pi)\in\{s,n\}^3\times S_3$. By \eqref{eq:types-two}, the $12$ cosets in $\tau^{-1}(c)$ satisfy $\mathrm{sgn}(\pi)=c_1$ and $\varepsilon_0\varepsilon_1\varepsilon_2=c_2$. Write $B_s^{(2)}=\{([1,1,1],\tfrac13),([3],\tfrac23)\}$ and $B_n^{(2)}=\{([2,1],1)\}$. Since one Markov step depends only on the first letter, the second step is
\begin{equation}\label{eq:markov-two-n2}
 B_c^{(3)}=\begin{cases}
 \tfrac23\,tB_{c_2}^{(2)}\ \sqcup\ \tfrac1{12}\displaystyle\bigsqcup_{\varepsilon_0\varepsilon_1\varepsilon_2=c_2} B_{\varepsilon_0}^{(2)}\times B_{\varepsilon_1}^{(2)}\times B_{\varepsilon_2}^{(2)},&c_1=s,\\[10pt]
 \tfrac12\displaystyle\bigsqcup_{\varepsilon\varepsilon'=c_2}dB_{\varepsilon}^{(2)}\times B_{\varepsilon'}^{(2)},&c_1=n.
 \end{cases}
\end{equation}
Here the products in the indices are products of signs. By Lemma~\ref{lem:uniform-two} with $\kappa=2$ and $CD(gL_1,M_1)=\frac12B_{\mathrm{sgn}(g)}^{(2)}$, each of the $48$ cosets contributes $\frac1{48}$ times its return-section distribution to $CD(\cdot,M_2)$. If $c_1=s$, four of the $12$ cosets have trivial top permutation and eight have a $3$-cycle; if $c_1=n$, all $12$ have a transposition, and each of the two allowed return-section pairs occurs six times. These counts give $CD(\tau^{-1}(c),M_2)=\frac14B_c^{(3)}$ by \eqref{eq:markov-two-n2}.

\emph{Inductive step.} Suppose $n\geq3$ and the result is known for $n-1$. Put $L=L_{n-1}$ and $U=\{\mathrm{id},y_{n-1},l_{n-1},y_{n-1}l_{n-1}\}$. Since $M_{n-1}/L\cong V_4$, the induction hypothesis reads
\begin{equation}\label{eq:coset-L-two}
 CD(gL,M_{n-1})=\tfrac14B_{\tau(g)}^{(n)}\qquad (g\in M_{n-1}).
\end{equation}
By Theorem~\ref{thm:Hn-two}, $H_n=L^3$, and the $192=[M_n:H_n]$ distinct cosets are precisely
\[
 H_n(u_0,u_1,u_2)\pi,\qquad u_i\in U,\quad \pi\in S_3,
\]
subject to the two relations in \eqref{eq:types-two}. For any fixed parent type $c$, they amount to
$\mathrm{sgn}(\pi)=c_1$ and $\tau(u_0)\tau(u_1)\tau(u_2)=f(c)$; they determine $u_2$ once $u_0,u_1$ and $\pi$ are fixed. Hence $\tau^{-1}(c)$ consists of $3\cdot4^2=48$ cosets. By Lemma~\ref{lem:uniform-two} and \eqref{eq:coset-L-two}, each coset contributes $1/192$ times its normalised return-section distribution: the product $B_{d_0}^{(n)}\times B_{d_1}^{(n)}\times B_{d_2}^{(n)}$ for $\pi=\mathrm{id}$, $tB_{f(c)}^{(n)}$ for a $3$-cycle, and $dB_{e_1}^{(n)}\times B_{e_2}^{(n)}$ for a transposition. In the last case $e_1$ is the product of the types at the exchanged vertices and $e_2$ is the type at the fixed vertex. The return sections are uniformly and independently distributed in their relevant cosets, as in Lemma~\ref{lem:uniform-two}.

For $c_1=s$, there are $16$ identity-top cosets and $32$ $3$-cycle-top cosets, so
\[
 CD(\tau^{-1}(c),M_n)=\tfrac{32}{192}\,tB_{f(c)}^{(n)}\ \sqcup\ \tfrac1{192}\bigsqcup_{d_0d_1d_2=f(c)} B_{d_0}^{(n)}\times B_{d_1}^{(n)}\times B_{d_2}^{(n)}=\tfrac14B_c^{(n+1)}.
\]
For $c_1=n$, all $48$ cosets have a transposition at the top. Each of the four ordered pairs $(e_1,e_2)$ with $e_1e_2=f(c)$ occurs four times for any fixed transposition (choose freely the first exchanged section type), hence $12$ times altogether. Therefore
\[
 CD(\tau^{-1}(c),M_n)=\tfrac{12}{192}\bigsqcup_{e_1e_2=f(c)} dB_{e_1}^{(n)}\times B_{e_2}^{(n)}=\tfrac14B_c^{(n+1)}.
\]
Both identities are exactly \eqref{eq:markov-step-two} after multiplication by $\frac14$, proving the four typed-coset assertions.

It remains to identify the six sets in i)--vi). In $L_n=\tau^{-1}(ss)$, the $32$ cosets with a $3$-cycle at the top form $J_nx_n\sqcup J_nx_n^2$. Of the $16$ with trivial top permutation, one is $H_n$; the other $15$ split, according to their section types, into the $3+3+3+6$ cosets represented by $z_n$, $k_n$, $k_nz_n$ and the mixed products of their $x_n$-conjugates. Indeed,
\[
 z_n=(y_{n-1},y_{n-1},x_{n-1}),\qquad k_n=(l_{n-1},l_{n-1},x_{n-1}),
\]
and conjugation by $x_n$ cyclically permutes the section types. The resulting families have, up to permutation, types
\[
 (ss,ss,ss),\quad (nn,nn,ss),\quad(ns,ns,ss),\quad(sn,sn,ss),\quad(ns,sn,nn).
\]
By Lemma~\ref{lem:uniform-two}, each coset contributes $1/192$ times the corresponding return-section product to $CD(\cdot,M_n)$. Multiplication by $[M_n:L_n]=4$ gives weights $1/48$, $1/16$, $1/16$, $1/16$ and $1/8$ for the five identity-top families. The $32$ $3$-cycle cosets contribute $4\cdot32/192=2/3$ times $tB_{ss}^{(n)}$. By \eqref{eq:markov-side-two} these are exactly $A_1^{(n)},\ldots,A_6^{(n)}$, as required.
\end{proof}

This allows us to propose the following models for $f_t$, depending on whether $f_t$ is irreducible over $K$ and on which of $P_1(t):=(f(\gamma_1)-t)(f(\gamma_2)-t)$, $P_2(t):=(f^2(\gamma_1)-t)(f^2(\gamma_2)-t)$ and $P_1(t)P_2(t)$ are squares in $K$. By the Chebotarev density theorem applied to $K(\sqrt{P_1(t)},\sqrt{P_2(t)})$, in each case the root types are equidistributed among the types $c$ with $c_1=s$ if $P_1(t)$ is a square, $c_2=s$ if $P_2(t)$ is a square, and $c_1=c_2$ if $P_1(t)P_2(t)$ is a square. As in Section~\ref{3}, we write $M_{n,t}$ for the group of the model of $f_t$.
\begin{enumerate}
    \item \textbf{Model~1.} If $f_t$ is reducible and $P_1(t)$ and $P_2(t)$ are squares, then the polynomial splits completely and the root type is $ss$. The initial data are
\[\begin{aligned}
    &[[ss,1]^{3},\frac{1}{16}],\ [[sn,1]^{2}[ss,1],\frac{3}{16}],\ [[ns,1]^{2}[ss,1],\frac{3}{16}] ,\\    &[[nn,1]^{2}[ss,1],\frac{3}{16}],\ [[nn,1][ns,1][sn,1],\frac{3}{8}].
\end{aligned}\]
    The corresponding group is $M_{n,t}=J_n$.
    \item \textbf{Model~2.} If $f_t$ is irreducible and $P_1(t)$ and $P_2(t)$ are squares, then the root type is $ss$. The initial data are
\[\begin{aligned}
    &[[ss,3],\frac{2}{3}],\ [[nn,1]^{2}[ss,1],\frac{1}{16}],\ [[nn,1][ns,1][sn,1],\frac{1}{8}] ,\\    &[[ns,1]^{2}[ss,1],\frac{1}{16}],\ [[sn,1]^{2}[ss,1],\frac{1}{16}],\ [[ss,1]^{3},\frac{1}{48}].
\end{aligned}\]
    The corresponding group is $M_{n,t}=L_n$.
    \item \textbf{Model~3.} If $f_t$ is irreducible, $P_1(t)$ is a square and $P_2(t)$ is not, then the root types are $ss$, $sn$, each with the same frequency. The initial data are
\[\begin{aligned}
    &[[ns,3],\frac{1}{3}],\ [[ss,3],\frac{1}{3}],\ [[nn,1]^{2}[ns,1],\frac{1}{32}] ,\\    &[[nn,1]^{2}[ss,1],\frac{1}{32}],\ [[nn,1][ns,1][sn,1],\frac{1}{16}],\ [[nn,1][sn,1][ss,1],\frac{1}{16}] ,\\    &[[ns,1]^{3},\frac{1}{96}],\ [[ns,1]^{2}[ss,1],\frac{1}{32}],\ [[ns,1][sn,1]^{2},\frac{1}{32}] ,\\    &[[ns,1][ss,1]^{2},\frac{1}{32}],\ [[sn,1]^{2}[ss,1],\frac{1}{32}],\ [[ss,1]^{3},\frac{1}{96}].
\end{aligned}\]
    The corresponding group is $M_{n,t}=\langle L_n,y_nl_n\rangle$.
    \item \textbf{Model~4.} If $f_t$ is irreducible, $P_2(t)$ is a square and $P_1(t)$ is not, then the root types are $ss$, $ns$, each with the same frequency. The initial data are
\[\begin{aligned}
    &[[ss,3],\frac{1}{3}],\ [[ns,2][nn,1],\frac{1}{8}],\ [[nn,2][ns,1],\frac{1}{8}] ,\\    &[[ss,2][sn,1],\frac{1}{8}],\ [[sn,2][ss,1],\frac{1}{8}],\ [[nn,1]^{2}[ss,1],\frac{1}{32}] ,\\    &[[nn,1][ns,1][sn,1],\frac{1}{16}],\ [[ns,1]^{2}[ss,1],\frac{1}{32}],\ [[sn,1]^{2}[ss,1],\frac{1}{32}] ,\\    &[[ss,1]^{3},\frac{1}{96}].
\end{aligned}\]
    The corresponding group is $M_{n,t}=\langle L_n,l_n\rangle$.
    \item \textbf{Model~5.} If $f_t$ is irreducible, $P_1(t)P_2(t)$ is a square but neither factor is, then the root types are $ss$, $nn$, each with the same frequency. The initial data are
\[\begin{aligned}
    &[[ss,3],\frac{1}{3}],\ [[ss,2][nn,1],\frac{1}{8}],\ [[nn,2][ss,1],\frac{1}{8}] ,\\    &[[sn,2][ns,1],\frac{1}{8}],\ [[ns,2][sn,1],\frac{1}{8}],\ [[nn,1]^{2}[ss,1],\frac{1}{32}] ,\\    &[[nn,1][ns,1][sn,1],\frac{1}{16}],\ [[ns,1]^{2}[ss,1],\frac{1}{32}],\ [[sn,1]^{2}[ss,1],\frac{1}{32}] ,\\    &[[ss,1]^{3},\frac{1}{96}].
\end{aligned}\]
    The corresponding group is $M_{n,t}=\langle L_n,y_n\rangle$.
    \item \textbf{Model~6.} If $f_t$ is irreducible and none of $P_1(t)$, $P_2(t)$ and $P_1(t)P_2(t)$ is a square, then the root types are $ss$, $sn$, $ns$, $nn$, each with the same frequency. The initial data are
\[\begin{aligned}
    &[[ns,3],\frac{1}{6}],\ [[ss,3],\frac{1}{6}],\ [[ns,2][nn,1],\frac{1}{16}] ,\\    &[[ss,2][nn,1],\frac{1}{16}],\ [[nn,2][ns,1],\frac{1}{16}],\ [[nn,2][ss,1],\frac{1}{16}] ,\\    &[[sn,2][ns,1],\frac{1}{16}],\ [[ns,2][sn,1],\frac{1}{16}],\ [[ss,2][sn,1],\frac{1}{16}] ,\\    &[[sn,2][ss,1],\frac{1}{16}],\ [[nn,1]^{2}[ns,1],\frac{1}{64}],\ [[nn,1]^{2}[ss,1],\frac{1}{64}] ,\\    &[[nn,1][ns,1][sn,1],\frac{1}{32}],\ [[nn,1][sn,1][ss,1],\frac{1}{32}],\ [[ns,1]^{3},\frac{1}{192}] ,\\    &[[ns,1]^{2}[ss,1],\frac{1}{64}],\ [[ns,1][sn,1]^{2},\frac{1}{64}],\ [[ns,1][ss,1]^{2},\frac{1}{64}] ,\\    &[[sn,1]^{2}[ss,1],\frac{1}{64}],\ [[ss,1]^{3},\frac{1}{192}].
\end{aligned}\]
    The corresponding group is $M_{n,t}=M_n$.
\end{enumerate}
For Models~2--6 the initial data are the equally weighted averages of the first transitions from their listed root types; each entry has been expanded above, combining states differing only by the order of their factors. In Model~1, the root is fixed at every prime, so the three degree-one sections correspond to the identity-top cosets. Other reducible cases may be treated using subgroups fixing the roots of $f_t$ (Corollary~\ref{cor:upper-bound-models}); we do not enumerate them here. The following theorem verifies the six listed models.
\begin{theorem}\label{Th 6.9}
    For each of the Models~1--6 and every $n\geq 2$, the cycle data of the group $M_{n,t}$ are the data obtained by applying the Markov process $n-1$ times to the initial data of the model.
\end{theorem}
\begin{proof}
    In Models~2--6 the group is the union of the cosets $\tau^{-1}(c)$ for the root types $c$ of the model, which all have the same size, and the initial data are the average of the first steps from the corresponding $[c,1]$. Hence the result follows from Theorem~\ref{4.44}. For Model~1, $J_n=L_n\cap\mathrm{St}_1$ by Corollary~\ref{4.42}, so $J_n$ is the union of the cosets of $H_n$ with $\pi=\mathrm{id}$ in $L_n$, classified in the proof of Theorem~\ref{4.44}. Hence, for $\sigma$ uniform in $J_n$, the types of the three sections are uniformly distributed among the triples with product $ss$ (for $n=2$, their signs are uniformly distributed among the triples with product $s$, which is all that matters after one step). By Lemma~\ref{lem:uniform-two} i), this gives the data obtained by applying the process $n-1$ times to the initial data of Model~1.
\end{proof}

\section{Markov groups for combined critical orbit of length three}\label{sec:three}
In this section $f=-2z^3+3z^2+\frac{1}{2}$ is the polynomial of Subsection~\ref{sec:three-model}, and we construct groups that follow the model of Definition~\ref{def:three-model}. As in Sections~\ref{3} and \ref{4}, $f$ is defined over a number field $K$ containing $\sqrt{-3}$, and we study $f_t^n$ for $t\in K$. We first compute the first-level data of $f_t$ when all four root types in $S$ are possible, by averaging the first steps of the model from $[sss,1]$, $[snn,1]$, $[nsn,1]$ and $[nns,1]$; the other models use the appropriate subset of root types:
\emph{The complete first-level data for the generic case are displayed in Model~6 below.}
We restrict the dynamics to the following ones, one for each type in $S$:
\begin{align*}
    &[sss,k]\to [sss,3k]\\
    &[nsn,k]\to [nsn,2k][nns,k]\\
    &[nns,k]\to [nns,2k][snn,k]\\
    &[snn,k]\to [nns,3k]
\end{align*}
If we forget the types, the patterns $[3]$, $[2,1]$ and $[1,1,1]$ occur with frequencies $\frac{1}{3}$, $\frac{1}{2}$ and $\frac{1}{6}$, as in Section~\ref{3}, so again $M_1=S_3=\langle (0,1,2),(0,1)\rangle$. Unlike in Sections~\ref{3} and \ref{4}, we need all four restricted transitions. We attach to $(0,1,2)$ the types $[sss,3]$ and $[nns,3]$, to $(0,1)$ the types $[nsn,2][nns,1]$ and $[nns,2][snn,1]$, and to $\mathrm{id}$ the types $[nsn,1]^2[sss,1]$ and $[nns,1]^2[sss,1]$. Iterating with respect to the restricted dynamics, we denote by $x_n$, $p_n$, $q_n$, $r_n$, $\hat{p}_n$ and $\hat{q}_n$ the elements of $\mathrm{Aut}(\mathcal{T}_n)$ obtained from $[sss,3]$, $[nsn,2][nns,1]$, $[nns,2][snn,1]$, $[nns,3]$, $[nsn,1]^2[sss,1]$ and $[nns,1]^2[sss,1]$, respectively. Explicitly, $x_n$ is as in Section~\ref{3}, and
$$p_n=(\mathrm{id},p_{n-1},q_{n-1})(0,1),\qquad q_n=(\mathrm{id},q_{n-1},r_{n-1})(0,1),\qquad r_n=(\mathrm{id},\mathrm{id},q_{n-1})(0,1,2),$$
$$\hat{p}_n=(p_{n-1},p_{n-1},x_{n-1}),\qquad \hat{q}_n=(q_{n-1},q_{n-1},x_{n-1}),$$
with $p_0=q_0=r_0=\hat{p}_0=\hat{q}_0=\mathrm{id}$, so that $p_1=q_1=(0,1)$, $r_1=(0,1,2)$ and $\hat{p}_1=\hat{q}_1=\mathrm{id}$. By Lemma~\ref{lem:return}, $p_n$ follows $[nsn,k]\to[nsn,2k][nns,k]$, $q_n$ follows $[nns,k]\to[nns,2k][snn,k]$, and $r_n$, whose return section along the $3$-cycle is $q_{n-1}$, follows $[snn,k]\to[nns,3k]$; the elements $\hat{p}_n$ and $\hat{q}_n$ correspond to $[nsn,1]^2[sss,1]$ and $[nns,1]^2[sss,1]$.
\begin{remark}\label{rem:three-generators}
    The elements $y_n$ and $z_n$ of Section~\ref{3} are not used: $y_n$ realises $[nnn,k]\to[nnn,2k][sss,k]$, and $nnn\notin S$. In fact $y_n\notin M_n$ for $n\geq 3$, since $\mathrm{sgn}_3(y_n)=-1$ by the recursion \eqref{eq:sign-recursion}, while $\mathrm{sgn}_3$ is trivial on $M_n$ by Theorem~\ref{thm:Markov-discriminant}.
\end{remark}
We propose the following groups,
\begin{align*}
    M_n&=\langle L_n,p_n,q_n,r_n\rangle,\\
    L_n&=\langle x_n,\hat{p}_n,\hat{q}_n,L_{n-1}\rangle,
\end{align*}
where $L_0$ is trivial and $L_{n-1}$ is embedded in $\mathrm{Aut}(\mathcal{T}_n)$ as in \eqref{rem:embed}.

We will show that $M_n$ has the cycle data corresponding to the full initial data, and that $L_n$ has the cycle data obtained by applying the process $n$ times to $[sss,1]$ (Theorems~\ref{thm:three-data} and \ref{thm:three-models}). The strategy is the same as in Section~\ref{4}. We expect $[M_n:L_n]=4$ for $n\geq 2$, since the data coming from $[sss,1]$ is $\frac{1}{4}$ of the initial data. For $n=1$ this fails: $p_1=q_1=(0,1)$ and $r_1=x_1$, so $M_1=S_3$ and $[M_1:L_1]=2$. Moreover, we expect $L_n$ to contain $L_{n-1}\times L_{n-1}\times L_{n-1}$, reflecting the transition $[sss,k]\to[sss,k]^3$, and $[L_n:H_n]=48$, where $H_n$ is the normal closure of $L_{n-1}$ in $L_n$, because starting from $[sss,1]$ the data $[sss,1]^3$ has probability $\frac{1}{3}\cdot\frac{1}{16}=\frac{1}{48}$. As in Section~\ref{4}, this value is attained for $n\geq 3$, and for $n=2$ the index is smaller.
\begin{theorem}\label{thm:three-normal}
For every $n\geq 1$ we have $L_n\trianglelefteq M_n$, and the quotient $Q_n:=M_n/L_n$ is an elementary abelian $2$-group generated by $\overline{p_n}$ and $\overline{q_n}$, with $\overline{r_n}=\overline{p_n}\,\overline{q_n}$. More precisely, $Q_1\cong C_2$ with $\overline{p_1}=\overline{q_1}$, and for $n\geq 2$
\[
  M_n/L_n\cong V_4=\{1,\ \overline{p_n},\ \overline{q_n},\ \overline{r_n}\}.
\]
\end{theorem}
\begin{lemma}\label{lem:three-psi}
Let $n\geq 2$ and assume Theorem~\ref{thm:three-normal} holds at level $n-1$. For an element $\sigma=(a_0,a_1,a_2)\pi$ of $\mathrm{Aut}(\mathcal{T}_n)$ with $a_0,a_1,a_2\in M_{n-1}$ we put $\mathrm{cl}(\sigma):=\overline{a_0}\,\overline{a_1}\,\overline{a_2}\in Q_{n-1}$, as in Lemma~\ref{lem:structureL}. Then statements i), ii) and iii) of Lemma~\ref{lem:structureL} hold for the groups of this section.
\end{lemma}
\begin{proof}
\emph{i)} The sections at the first level of the generators are $\mathrm{id}$, $x_{n-1}$, $p_{n-1}$, $q_{n-1}$, $r_{n-1}$ and the elements of $L_{n-1}$, which lie in $M_{n-1}$. The formula for $\mathrm{cl}$ holds as in Lemma~\ref{lem:structureL}, since $Q_{n-1}$ is abelian.

\emph{ii)} We check the generators. We have $\mathrm{cl}(x_n)=\overline{x_{n-1}}=1$, and $\mathrm{cl}(\hat{p}_n)=\overline{p_{n-1}}^{\,2}\,\overline{x_{n-1}}=1$ and $\mathrm{cl}(\hat{q}_n)=\overline{q_{n-1}}^{\,2}\,\overline{x_{n-1}}=1$, because $Q_{n-1}$ has exponent $2$. The generators $(g,\mathrm{id},\mathrm{id})$ with $g\in L_{n-1}$ are clear, and all these generators have even top permutation.

\emph{iii)} The proof is that of Lemma~\ref{lem:structureL} iii), with $\hat{p}_n$ and $\hat{q}_n$ in place of $z_n$ and $k_n$: the classes of their sections are $(\overline{p_{n-1}},\overline{p_{n-1}},1)$ and $(\overline{q_{n-1}},\overline{q_{n-1}},1)$, and $Q_{n-1}$ is an elementary abelian $2$-group generated by $\overline{p_{n-1}}$ and $\overline{q_{n-1}}$ (Theorem~\ref{thm:three-normal} at level $n-1$).
\end{proof}
We proceed to prove Theorem~\ref{thm:three-normal}.
\begin{proof}[Proof of Theorem~\ref{thm:three-normal}]
We argue by induction on $n$. For $n=1$ we have $\hat{p}_1=\hat{q}_1=\mathrm{id}$, so $L_1=\langle(0,1,2)\rangle\cong A_3$, and $M_1=\langle A_3,(0,1)\rangle=S_3$, since $p_1=q_1=(0,1)$ and $r_1=(0,1,2)$. Hence $M_1/L_1\cong C_2$ is generated by $\overline{p_1}=\overline{q_1}$, and $\overline{r_1}=1=\overline{p_1}\,\overline{q_1}$.

Let $n\geq 2$ and assume the statement at level $n-1$. We write $x=x_{n-1}$, $p=p_{n-1}$, $q=q_{n-1}$ and $r=r_{n-1}$. From the definitions, using $\overline{r}=\overline{p}\,\overline{q}$ and that $Q_{n-1}$ is elementary abelian, we get
\[
  \mathrm{cl}(p_n)=\overline{p}\,\overline{q},\qquad
  \mathrm{cl}(q_n)=\overline{q}\,\overline{r}=\overline{p},\qquad
  \mathrm{cl}(r_n)=\overline{q}=\mathrm{cl}(p_n)\,\mathrm{cl}(q_n).
\]
Since $p_nq_n$ and $r_n$ have even top permutation, $r_n(p_nq_n)^{-1}$ lies in $L_n$ by Lemma~\ref{lem:three-psi} i) and iii). Hence $M_n=\langle L_n,p_n,q_n\rangle$. We prove normality on generators, as in Theorem~\ref{th 4.39}: if $g$ is a generator of $L_n$ and $w\in\{p_n,q_n\}$, then $g^w$ has even top permutation and $\mathrm{cl}(g^w)=\mathrm{cl}(g)=1$, so $g^w\in L_n$. For instance,
$$x_n^{p_n}=(p^{-1}q,\ p,\ q^{-1}x)(0,2,1),\qquad x_n^{q_n}=(q^{-1}r,\ q,\ r^{-1}x)(0,2,1),\qquad \hat{q}_n^{\,p_n}=(p^{-1}qp,\ q,\ q^{-1}xq).$$
Hence $L_n$ is normal in $M_n$. In the same way $p_n^2$, $q_n^2$ and the commutator of $p_n$ and $q_n$ have even top permutation and $\mathrm{cl}=1$, so they lie in $L_n$, and $M_n/L_n$ is an elementary abelian $2$-group generated by $\overline{p_n}$ and $\overline{q_n}$. The elements $p_n$ and $q_n$ have odd top permutation, so they do not lie in $L_n$, and $p_nq_n$ does not lie in $L_n$ by Lemma~\ref{lem:three-psi} ii), because $\mathrm{cl}(p_nq_n)=\overline{q}\neq 1$: for $n=2$ the element $\overline{q_1}$ generates $Q_1$, and for $n\geq 3$ it is a non-trivial element of $Q_{n-1}\cong V_4$. Hence $M_n/L_n\cong V_4$, the images of $p_n$, $q_n$ and $r_n$ are its three non-trivial elements, and $r_n\in L_np_nq_n$. This completes the induction.
\end{proof}
We now describe the normal closure $H_n:=\langle L_{n-1}\rangle^{L_n}$ of $L_{n-1}$ in $L_n$, and the normal closure
$$J_n:=\langle\hat{p}_n,\hat{q}_n,L_{n-1}\rangle^{L_n}$$
in $L_n$ of the subgroup generated by $\hat{p}_n$, $\hat{q}_n$ and $L_{n-1}$.
\begin{corollary}\label{cor:three-K}
For $n\geq 2$, $H_n=L_{n-1}\times L_{n-1}^{x_n}\times L_{n-1}^{x_n^2}$, that is, $H_n$ is the set of elements $(a_0,a_1,a_2)$ with $a_0,a_1,a_2\in L_{n-1}$. Moreover $J_n=L_n\cap\mathrm{St}_1$, $[L_n:J_n]=3$ and
\[
  [L_n:H_n]=3\,|Q_{n-1}|^2=\begin{cases} 12, & n=2,\\ 48, & n\geq 3,\end{cases}
\]
and $J_n/H_n\cong\{v\in Q_{n-1}^3\,:\,v_0v_1v_2=1\}$, which is $V_4$ for $n=2$ and $(\mathbb{Z}/2\mathbb{Z})^4$ for $n\geq 3$. In particular $|M_2|=4\cdot 12\cdot 3^3=6^4$, so $M_2=\mathrm{Aut}(\mathcal{T}_2)$.
\end{corollary}
\begin{proof}
    By Lemma~\ref{lem:three-psi} iii), $L_{n-1}^3$ is contained in $L_n$. As in the proof of Theorem~\ref{thm:Hn-two}, it is normal in $L_n$ on generators, so it contains $H_n$; conversely, it is generated by $L_{n-1}$ and its conjugates by $x_n$ and $x_n^2$. The rest is proved as Corollary~\ref{4.42}: the map $\theta\colon L_n\cap\mathrm{St}_1\to Q_{n-1}^3$ sending an element to the classes of its three sections has kernel $H_n$ and image $\{v\in Q_{n-1}^3\,:\,v_0v_1v_2=1\}$, the top permutations of the elements of $L_n$ form the group $A_3$, and the images of $\hat{p}_n$, $\hat{q}_n$ and of their conjugates by $x_n$ generate the image of $\theta$.
\end{proof}
\begin{lemma}\label{lem:three-uniform}
    Let $n\geq 2$ and $\sigma=(g_0,g_1,g_2)\pi\in M_n$, and write $\kappa:=[M_{n-1}:L_{n-1}]$, so that $\kappa=2$ if $n=2$ and $\kappa=4$ if $n\geq 3$. Then the statements i), ii) and iii) of Lemma~\ref{lem:uniform-two} hold, with the same constants $\frac{\kappa}{12}$, $\frac{1}{12\kappa}$ and $\frac{1}{12}$.
\end{lemma}
\begin{proof}
    By Lemma~\ref{lem:three-psi}, $g_0,g_1,g_2\in M_{n-1}$; by Corollary~\ref{cor:three-K}, $H_n=L_{n-1}\times L_{n-1}\times L_{n-1}$; and by Theorem~\ref{thm:three-normal}, $L_{n-1}$ is normal in $M_{n-1}$ and $M_{n-1}/L_{n-1}$ is abelian. These are the only facts used in the proof of Lemma~\ref{lem:uniform-two}. The constants only depend on $[M_n:H_n]=4\cdot 3\kappa^2=12\kappa^2$ and on $\kappa$, which are the same as in Section~\ref{4}.
\end{proof}
As in Section~\ref{4} we attach types to the elements of $M_n$. For $n\geq 2$ we let $\tau\colon M_n\to S$ be the homomorphism with kernel $L_n$ sending $p_n,q_n,r_n$ to $nsn,nns,snn$; by Theorem~\ref{thm:three-normal} it induces an isomorphism $Q_n\cong S$. For $n=1$ we let $\tau\colon M_1\to\{s,n\}$ be the sign, which only records the first letter. Let $n\geq 3$, and let $\sigma=(a_0,a_1,a_2)\pi\in M_n$ have type $c\in S$. Then
\begin{equation}\label{eq:types-three}
    \mathrm{sgn}(\pi)=c_1\ \text{ and }\ \tau(a_0)\tau(a_1)\tau(a_2)=f(c).
\end{equation}
Indeed, as in Section~\ref{4}, both sides of the two equalities are multiplicative in $\sigma$, so we only have to check the generators. For the generators of $L_n$, of type $sss$, this is Lemma~\ref{lem:three-psi} ii). The sections of $p_n$ have product of types $nsn\cdot nns=snn=f(nsn)$, those of $q_n$ have $nns\cdot snn=nsn=f(nns)$, and those of $r_n$ have $nns=f(snn)$; the top permutations of $p_n$ and $q_n$ are odd, and that of $r_n$ is even.

For $n=2$ the same holds if we only compare the first letters of $\tau(a_0)\tau(a_1)\tau(a_2)$ and of $f(c)$, i.e.\ $\tau(a_0)\tau(a_1)\tau(a_2)=c_2$, as the signs of the sections of $p_2$, $q_2$ and $r_2$ show. Recall that $M_2=\mathrm{Aut}(\mathcal{T}_2)$ by Corollary~\ref{cor:three-K}.

We introduce notation for the cycle data:
\[
\resizebox{\textwidth}{!}{$
\begin{aligned}
A_1^{(n)} &:= \textit{the cycle data obtained by applying the Markov process $n-1$ times to } ([sss,1]^3,\tfrac{1}{48}); \\
A_2^{(n)} &:= \textit{the cycle data obtained by applying the Markov process $n-1$ times to } ([sss,3],\tfrac{2}{3}); \\
A_3^{(n)} &:= \textit{the cycle data obtained by applying the Markov process $n-1$ times to } ([nsn,1]^2[sss,1],\tfrac{1}{16}); \\
A_4^{(n)} &:= \textit{the cycle data obtained by applying the Markov process $n-1$ times to } ([nns,1]^2[sss,1],\tfrac{1}{16}); \\
A_5^{(n)} &:= \textit{the cycle data obtained by applying the Markov process $n-1$ times to } ([snn,1]^2[sss,1],\tfrac{1}{16}); \\
A_6^{(n)} &:= \textit{the cycle data obtained by applying the Markov process $n-1$ times to } ([snn,1][nsn,1][nns,1],\tfrac{1}{8}).
\end{aligned}
$}
\]
Here the Markov process is that of Definition~\ref{def:three-model}. We also write $B_c^{(n)}$, for $c\in S$, for the cycle data obtained by applying the Markov process $n-1$ times to $([c,1],1)$, so that $B_c^{(1)}=\{([1],1)\}$ for every $c$.
\begin{theorem}\label{thm:three-data}
    Let $n\geq 2$ and $c\in S$. If $\sigma$ is chosen uniformly in $\tau^{-1}(c)$, that is, in $L_n$, $L_nr_n$, $L_np_n$ or $L_nq_n$ for $c=sss,snn,nsn,nns$ respectively, then the cycle type of $\sigma$ is distributed as the data obtained by applying the Markov process of Definition~\ref{def:three-model} $n$ times to $[c,1]$. In other words,
    $$CD(\tau^{-1}(c),M_n)=\tfrac{1}{4}\,B_c^{(n+1)}.$$
    Moreover, for $n\geq 3$:
    \begin{enumerate}[i)]
        \item $A_1^{(n)}=CD(H_n,L_n);$
        \item $A_2^{(n)}=CD(J_nx_n\sqcup J_nx_n^2,L_n);$
        \item $A_3^{(n)}=CD(\hat{p}_nH_n\sqcup \hat{p}_n^{x_n}H_n\sqcup \hat{p}_n^{x_n^2}H_n,L_n);$
        \item $A_4^{(n)}=CD(\hat{q}_nH_n\sqcup \hat{q}_n^{x_n}H_n\sqcup \hat{q}_n^{x_n^2}H_n,L_n);$
        \item $A_5^{(n)}=CD(\hat{p}_n\hat{q}_nH_n\sqcup (\hat{p}_n\hat{q}_n)^{x_n}H_n\sqcup (\hat{p}_n\hat{q}_n)^{x_n^2}H_n,L_n);$
        \item $A_6^{(n)}=CD(\hat{p}_n\hat{q}_n^{x_n}H_n\sqcup \hat{p}_n\hat{q}_n^{x_n^2}H_n\sqcup \hat{p}_n^{x_n}\hat{q}_nH_n\sqcup \hat{p}_n^{x_n^2}\hat{q}_nH_n\sqcup \hat{p}_n^{x_n}\hat{q}_n^{x_n^2}H_n\sqcup \hat{p}_n^{x_n^2}\hat{q}_n^{x_n}H_n,L_n)$.
    \end{enumerate}
\end{theorem}
\begin{proof}
We prove the claim directly for the length-three construction. The induction hypothesis applies independently in the three factors of $H_n=L_{n-1}^3$; the quotient representatives $p_n$, $q_n$ and their products determine the allowed section types, while the root permutation determines whether their cycles split, double or triple. The process of Definition~\ref{def:three-model} gives
\begin{equation}\label{eq:markov-side-three}
\begin{aligned}
 A_1^{(n)}&=\tfrac1{48}(B_{sss}^{(n)})^{\times3}, & A_2^{(n)}&=\tfrac23\,tB_{sss}^{(n)},\\
 A_3^{(n)}&=\tfrac1{16}(B_{nsn}^{(n)})^{\times2}\times B_{sss}^{(n)}, & A_4^{(n)}&=\tfrac1{16}(B_{nns}^{(n)})^{\times2}\times B_{sss}^{(n)},\\
 A_5^{(n)}&=\tfrac1{16}(B_{snn}^{(n)})^{\times2}\times B_{sss}^{(n)}, & A_6^{(n)}&=\tfrac18B_{snn}^{(n)}\times B_{nsn}^{(n)}\times B_{nns}^{(n)}.
\end{aligned}
\end{equation}
Here, for $c\in S$, the Markov recursion is
\begin{equation}\label{eq:markov-step-three}
 B_c^{(n+1)}=\begin{cases}
 \tfrac23\,tB_{f(c)}^{(n)}\ \sqcup\ \tfrac1{48}\displaystyle\bigsqcup_{d_0d_1d_2=f(c)}B_{d_0}^{(n)}\times B_{d_1}^{(n)}\times B_{d_2}^{(n)},&c_1=s,\\[10pt]
 \tfrac14\displaystyle\bigsqcup_{e_1e_2=f(c)}dB_{e_1}^{(n)}\times B_{e_2}^{(n)},&c_1=n.
 \end{cases}
\end{equation}
The unions are over ordered tuples of types in $S$. Since $S$ is an elementary abelian group of order four, for $c=sss$ the $16$ triples split into the identity triple, three permutations of each of the three equal-pair triples, and six permutations of $(snn,nsn,nns)$. Thus
\begin{equation}\label{eq:markov-sss}
 \tfrac14B_{sss}^{(n+1)}=\tfrac14\bigl(A_1^{(n)}\sqcup\cdots\sqcup A_6^{(n)}\bigr).
\end{equation}

For $n=2$, the cycle types after one step depend only on the first letter of the type. The projection $S\to\{s,n\}$, $c\mapsto c_1$, is a surjective homomorphism with kernel $\{sss,snn\}$. In the notation of the proof of Theorem~\ref{4.44}, $B_c^{(3)}$ is therefore given by \eqref{eq:markov-two-n2} with $c_2$ the second letter of $c$. At this level $\hat p_2=\hat q_2=z_2=k_2$, so $L_2$ and $M_2$ are the groups of Section~\ref{4}; also their type characters have the same two-letter data $(c_1,c_2)$. Theorem~\ref{4.44} for $n=2$ gives $CD(\tau^{-1}(c),M_2)=\frac14B_c^{(3)}$.

Suppose now $n\geq3$ and assume the result at level $n-1$. By Corollary~\ref{cor:three-K} we have $H_n=L_{n-1}^3$, while $M_{n-1}/L_{n-1}\cong S$ and $[M_n:H_n]=192$. Choose representatives $U=\{\mathrm{id},p_{n-1},q_{n-1},p_{n-1}q_{n-1}\}$. The $192$ cosets of $H_n$ are
\[
 H_n(u_0,u_1,u_2)\pi,\qquad u_i\in U,\quad\pi\in S_3,
\]
and \eqref{eq:types-three} identifies those of parent type $c$ by $\mathrm{sgn}(\pi)=c_1$ and $\tau(u_0)\tau(u_1)\tau(u_2)=f(c)$. Consequently each $\tau^{-1}(c)$ contains $48$ cosets. By the induction hypothesis,
\begin{equation}\label{eq:coset-L-three}
 CD(gL_{n-1},M_{n-1})=\tfrac14B_{\tau(g)}^{(n)}\qquad(g\in M_{n-1}).
\end{equation}
Lemma~\ref{lem:three-uniform} shows that each coset of $H_n$ contributes $1/192$ times its return-section distribution. If $c_1=s$, the $16$ identity-top and $32$ $3$-cycle-top cosets yield
\[
 CD(\tau^{-1}(c),M_n)=\tfrac1{192}\bigsqcup_{d_0d_1d_2=f(c)}B_{d_0}^{(n)}\times B_{d_1}^{(n)}\times B_{d_2}^{(n)}\ \sqcup\ \tfrac{32}{192}\,tB_{f(c)}^{(n)}=\tfrac14B_c^{(n+1)}.
\]
If $c_1=n$, all $48$ cosets have a transposition at the top. Each prescribed ordered pair of return-section types occurs four times for each transposition, hence $12$ times in total. Therefore
\[
 CD(\tau^{-1}(c),M_n)=\tfrac{12}{192}\bigsqcup_{e_1e_2=f(c)}dB_{e_1}^{(n)}\times B_{e_2}^{(n)}=\tfrac14B_c^{(n+1)}.
\]
These are precisely the two cases of \eqref{eq:markov-step-three}.

Finally, the sets in i)--vi) refine the $48$ cosets of $L_n=\tau^{-1}(sss)$. The $32$ $3$-cycle-top cosets are $J_nx_n\sqcup J_nx_n^2$, and the other $16$ split into $H_n$ and families of sizes $3,3,3,6$. These are represented by $\hat p_n=(p_{n-1},p_{n-1},x_{n-1})$, $\hat q_n=(q_{n-1},q_{n-1},x_{n-1})$, their product and the mixed products of their $x_n$-conjugates; conjugation by $x_n$ cyclically permutes the section types. Up to permutation, the section types are $(sss,sss,sss)$, $(nsn,nsn,sss)$, $(nns,nns,sss)$, $(snn,snn,sss)$, and $(snn,nsn,nns)$. The contributions of these families to the cycle data relative to $L_n$ have weights $1/48$, $1/16$, $1/16$, $1/16$, $1/8$; the $3$-cycle family contributes $2/3$. By \eqref{eq:markov-side-three}, these are $A_1^{(n)},\ldots,A_6^{(n)}$.
\end{proof}

This allows us to propose the following models for $f_t$, depending on whether $f_t$ is irreducible over $K$ and on which of
$$P_1(t)=\bigl(\tfrac{1}{2}-t\bigr)\bigl(\tfrac{3}{2}-t\bigr),\qquad P_2(t)=(1-t)\bigl(\tfrac{1}{2}-t\bigr),\qquad P_3(t)=\bigl(\tfrac{3}{2}-t\bigr)(1-t)$$
are squares in $K$. By Lemma~\ref{lem:three-relation}, $P_1(t)P_2(t)P_3(t)$ is always a square, so either none of $P_1(t),P_2(t),P_3(t)$ is a square, or exactly one of them is, or all of them are. By the Chebotarev density theorem applied to $K(\sqrt{P_1(t)},\sqrt{P_2(t)})$, in each case the root types are equidistributed among the types $c\in S$ with $c_j=s$ for every $j$ such that $P_j(t)$ is a square. As in Section~\ref{3}, we write $M_{n,t}$ for the group of the model of $f_t$.
\begin{enumerate}
    \item \textbf{Model~1.} If $f_t$ is reducible and $P_2(t)$ is a square, then the polynomial splits completely and the root type is $sss$. The initial data are
\[\begin{aligned}
    &[[sss,1]^{3},\frac{1}{16}],\ [[snn,1]^{2}[sss,1],\frac{3}{16}] ,\\    &[[nsn,1]^{2}[sss,1],\frac{3}{16}],\ [[nns,1]^{2}[sss,1],\frac{3}{16}] ,\\    &[[nns,1][nsn,1][snn,1],\frac{3}{8}].
\end{aligned}\]
    The corresponding group is $M_{n,t}=J_n$.
    \item \textbf{Model~2.} If $f_t$ is irreducible and $P_1(t)$ and $P_2(t)$ are squares, then the root type is $sss$. The initial data are
\[\begin{aligned}
    &[[sss,3],\frac{2}{3}],\ [[nns,1]^{2}[sss,1],\frac{1}{16}] ,\\    &[[nns,1][nsn,1][snn,1],\frac{1}{8}],\ [[nsn,1]^{2}[sss,1],\frac{1}{16}] ,\\    &[[snn,1]^{2}[sss,1],\frac{1}{16}],\ [[sss,1]^{3},\frac{1}{48}].
\end{aligned}\]
    The corresponding group is $M_{n,t}=L_n$.
    \item \textbf{Model~3.} If $f_t$ is irreducible, $P_1(t)$ is a square and $P_2(t)$ is not, then the root types are $sss$, $snn$, each with the same frequency. The initial data are
\[\begin{aligned}
    &[[nns,3],\frac{1}{3}],\ [[sss,3],\frac{1}{3}] ,\\    &[[nns,1]^{3},\frac{1}{96}],\ [[nns,1]^{2}[sss,1],\frac{1}{32}] ,\\    &[[nns,1][nsn,1]^{2},\frac{1}{32}],\ [[nns,1][nsn,1][snn,1],\frac{1}{16}] ,\\    &[[nns,1][snn,1]^{2},\frac{1}{32}],\ [[nns,1][sss,1]^{2},\frac{1}{32}] ,\\    &[[nsn,1]^{2}[sss,1],\frac{1}{32}],\ [[nsn,1][snn,1][sss,1],\frac{1}{16}] ,\\    &[[snn,1]^{2}[sss,1],\frac{1}{32}],\ [[sss,1]^{3},\frac{1}{96}].
\end{aligned}\]
    The corresponding group is $M_{n,t}=\langle L_n,r_n\rangle$.
    \item \textbf{Model~4.} If $f_t$ is irreducible, $P_2(t)$ is a square and $P_1(t)$ is not, then the root types are $sss$, $nsn$, each with the same frequency. The initial data are
\[\begin{aligned}
    &[[sss,3],\frac{1}{3}],\ [[nsn,2][nns,1],\frac{1}{8}] ,\\    &[[nns,2][nsn,1],\frac{1}{8}],\ [[sss,2][snn,1],\frac{1}{8}] ,\\    &[[snn,2][sss,1],\frac{1}{8}],\ [[nns,1]^{2}[sss,1],\frac{1}{32}] ,\\    &[[nns,1][nsn,1][snn,1],\frac{1}{16}],\ [[nsn,1]^{2}[sss,1],\frac{1}{32}] ,\\    &[[snn,1]^{2}[sss,1],\frac{1}{32}],\ [[sss,1]^{3},\frac{1}{96}].
\end{aligned}\]
    The corresponding group is $M_{n,t}=\langle L_n,p_n\rangle$.
    \item \textbf{Model~5.} If $f_t$ is irreducible, $P_3(t)$ is a square but $P_1(t)$ and $P_2(t)$ are not, then the root types are $sss$, $nns$, each with the same frequency. The initial data are
\[\begin{aligned}
    &[[sss,3],\frac{1}{3}],\ [[snn,2][nns,1],\frac{1}{8}] ,\\    &[[nns,2][snn,1],\frac{1}{8}],\ [[sss,2][nsn,1],\frac{1}{8}] ,\\    &[[nsn,2][sss,1],\frac{1}{8}],\ [[nns,1]^{2}[sss,1],\frac{1}{32}] ,\\    &[[nns,1][nsn,1][snn,1],\frac{1}{16}],\ [[nsn,1]^{2}[sss,1],\frac{1}{32}] ,\\    &[[snn,1]^{2}[sss,1],\frac{1}{32}],\ [[sss,1]^{3},\frac{1}{96}].
\end{aligned}\]
    The corresponding group is $M_{n,t}=\langle L_n,q_n\rangle$.
    \item \textbf{Model~6.} If $f_t$ is irreducible and none of $P_1(t)$, $P_2(t)$ and $P_3(t)$ is a square, then the root types are $sss$, $snn$, $nsn$, $nns$, each with the same frequency. The initial data are
\[\begin{aligned}
    &[[nns,3],\frac{1}{6}],\ [[sss,3],\frac{1}{6}] ,\\    &[[nsn,2][nns,1],\frac{1}{16}],\ [[snn,2][nns,1],\frac{1}{16}] ,\\    &[[nns,2][nsn,1],\frac{1}{16}],\ [[nns,2][snn,1],\frac{1}{16}] ,\\    &[[sss,2][nsn,1],\frac{1}{16}],\ [[nsn,2][sss,1],\frac{1}{16}] ,\\    &[[sss,2][snn,1],\frac{1}{16}],\ [[snn,2][sss,1],\frac{1}{16}] ,\\    &[[nns,1]^{3},\frac{1}{192}],\ [[nns,1]^{2}[sss,1],\frac{1}{64}] ,\\    &[[nns,1][nsn,1]^{2},\frac{1}{64}],\ [[nns,1][nsn,1][snn,1],\frac{1}{32}] ,\\    &[[nns,1][snn,1]^{2},\frac{1}{64}],\ [[nns,1][sss,1]^{2},\frac{1}{64}] ,\\    &[[nsn,1]^{2}[sss,1],\frac{1}{64}],\ [[nsn,1][snn,1][sss,1],\frac{1}{32}] ,\\    &[[snn,1]^{2}[sss,1],\frac{1}{64}],\ [[sss,1]^{3},\frac{1}{192}].
\end{aligned}\]
    The corresponding group is $M_{n,t}=M_n$.
\end{enumerate}
For Models~2--6 the initial data are the equally weighted averages of the first transitions from the stated root types; all factorisation states and their weights are displayed above. For example, the values $t=\frac{23}{48},\frac16,\frac13,\frac35,-60$ belong to Models~2--6 respectively. The remaining reducible cases can be treated with subgroups fixing the roots of $f_t$; we do not enumerate them.
\begin{theorem}\label{thm:three-models}
    For each of the Models~1--6 above and every $n\geq 2$, the cycle data of the group $M_{n,t}$ are the data obtained by applying the Markov process of Definition~\ref{def:three-model} $n-1$ times to the initial data of the model.
\end{theorem}
\begin{proof}
    In Models~2--6 the group is the union of the cosets $\tau^{-1}(c)$ for the root types $c$ of the model, which all have the same size, and the initial data are the average of the first steps from the corresponding $[c,1]$. Hence the result follows from Theorem~\ref{thm:three-data}. For Model~1, $J_n=L_n\cap\mathrm{St}_1$ by Corollary~\ref{cor:three-K}, so by the proof of Theorem~\ref{thm:three-data} with $\pi=\mathrm{id}$, for $\sigma$ uniform in $J_n$ the section types are uniformly distributed among the $16$ triples in $S^3$ with product $sss$, which gives the initial data of Model~1, and the rest of the argument is the same.
\end{proof}

\section{Refined Markov groups for co-critical polynomials of length two}\label{sec:refined-two}
We now extend the methods of Section~\ref{4} to the refined model of Definition~\ref{def:refined-model}. By Proposition~\ref{prop:cocritical-classification}, the polynomials in the list of~\cite{cubic-PCF} with non-colliding critical orbits, combined critical orbit of length two and co-critical points on the combined critical orbit are
$$-z^3+\frac{3}{2}z^2-1,\qquad -\frac{1}{4}z^3+\frac{3}{2}z+2,$$
and for both of them $\{\gamma_1^{\ast},\gamma_2^{\ast}\}=P_2\neq P_1$. More generally, throughout this section $f$ is a cubic with $\{\gamma_1^{\ast},\gamma_2^{\ast}\}=P_2\neq P_1$. Then $P_3=P_1$ (Proposition~\ref{prop:cocritical-classification}), so that $f(c_1c_2)=c_2c_1$ in Definition~\ref{def 4.8}, and by Proposition~\ref{prop:cocritical-types} some admissible types never occur. We consider the conjugates $f_t$ as in Section~\ref{4}, and we construct groups that follow the refined model.
 
The first-level data in the generic case are obtained by averaging the refined transition rules over all four root types:
\emph{The complete first-level data for the generic case are displayed in Model~6 below.}
We have the dynamics described in the table after Proposition~\ref{prop:cocritical-types}. We can restrict the dynamics to the following ones:
\begin{align*}
    &[ss,k]\to [ss,3k]\\
    &[nn,k]\to [ss,2k][nn,k]\\
    &[sn,k]\to [ns,3k]\\
    &[ns,k]\to [ns,2k][nn,k]
\end{align*}
If we forget the types, the patterns $[3]$, $[2,1]$ and $[1,1,1]$ again occur with frequencies $\frac{1}{3}$, $\frac{1}{2}$ and $\frac{1}{6}$, so again $\widetilde{M}_1=S_3=\langle(0,1,2),(0,1)\rangle$. As in Section~\ref{4}, we attach two types to $(0,1)$: we attach to $(0,1,2)$ the type $[ss,3]$, to $(0,1)$ the types $[ss,2][nn,1]$ and $[ns,2][nn,1]$, and to $\mathrm{id}$ the type $[ns,1]^2[ss,1]$. Iterating with respect to the restricted dynamics, we denote by $x_n$, $\tilde{y}_n$, $\tilde{l}_n$ and $\tilde{k}_n$ the elements obtained from $[ss,3]$, $[ss,2][nn,1]$, $[ns,2][nn,1]$ and $[ns,1]^2[ss,1]$, respectively. Explicitly, $x_n$ is as in Section~\ref{3}, and
$$\tilde{y}_n=\bigl(\tilde{y}_{n-1},\,\tilde{y}_{n-1}^{-1}x_{n-1},\,\tilde{y}_{n-1}\bigr)(0,1),\qquad \tilde{l}_n=\bigl(\tilde{y}_{n-1},\,\tilde{y}_{n-1}^{-1}\tilde{l}_{n-1},\,\tilde{y}_{n-1}\bigr)(0,1),\qquad \tilde{k}_n=\bigl(\tilde{l}_{n-1},\,\tilde{l}_{n-1},\,x_{n-1}\bigr),$$
with $\tilde{y}_0=\tilde{l}_0=\tilde{k}_0=\mathrm{id}$, so that $\tilde{y}_1=\tilde{l}_1=(0,1)$ and $\tilde{k}_1=\mathrm{id}$; by Lemma~\ref{lem:return}, $\tilde{y}_n$ follows $[nn,k]\to[ss,2k][nn,k]$, $\tilde{l}_n$ follows $[ns,k]\to[ns,2k][nn,k]$ and $\tilde{k}_n$ corresponds to $[ns,1]^2[ss,1]$.

\begin{remark}\label{rem:refined-two-generators}
    The elements $y_n$ and $z_n$ of Section~\ref{4} are not used, since they realise the transitions $[nn,k]\to[nn,2k][ss,k]$ and $[ss,k]\to[ss,k][nn,k]^2$, which never occur. The element $\tilde{l}_n$ has the same return sections as $l_n$ (with tildes), but the section of $l_n$ at the vertex $0$ is $\mathrm{id}$, whose second letter is wrong. For $n=2$ we have $\tilde{k}_2=k_2=z_2$.
\end{remark}
We propose the following groups for the refined model:
\begin{align*}
    \widetilde{M}_n&=\langle \widetilde{L}_n,\tilde{y}_n,\tilde{l}_n\rangle ,\\
    \widetilde{L}_n&=\langle x_n,\tilde{k}_n,\widetilde{L}_{n-1}\rangle ,
\end{align*}
where $\widetilde{L}_0$ is trivial and $\widetilde{L}_{n-1}$ is embedded in $\mathrm{Aut}(\mathcal{T}_n)$ as in \eqref{rem:embed}.
 
We will show that $\widetilde{M}_n$ has the cycle data corresponding to the full initial data of the refined model, and that $\widetilde{L}_n$ has the cycle data obtained by applying the process $n$ times to $[ss,1]$ (Theorems~\ref{thm:refined-two-data} and \ref{thm:refined-two-models}). The strategy is the same as in Section~\ref{4}. We expect $[\widetilde{M}_n:\widetilde{L}_n]=4$ for $n\geq 2$, since the data coming from $[ss,1]$ is $\frac{1}{4}$ of the initial data. For $n=1$ this fails: $\tilde{y}_1=\tilde{l}_1=(0,1)$, so $\widetilde{M}_1=S_3$ and $[\widetilde{M}_1:\widetilde{L}_1]=2$. Moreover, we expect $\widetilde{L}_n$ to contain $\widetilde{L}_{n-1}\times\widetilde{L}_{n-1}\times\widetilde{L}_{n-1}$, reflecting the transition $[ss,k]\to[ss,k]^3$. The refined model also suggests $[\widetilde{L}_n:\widetilde{H}_n]=12$, where $\widetilde{H}_n$ is the normal closure of $\widetilde{L}_{n-1}$ in $\widetilde{L}_n$, because starting from $[ss,1]$ the data $[ss,1]^3$ has probability $\frac{1}{3}\cdot\frac{1}{4}=\frac{1}{12}$. Unlike in Section~\ref{4}, we will see that this value is attained for every $n\geq 2$, as in Section~\ref{3}.
 
We first confirm that $[\widetilde{M}_n:\widetilde{L}_n]=4$ for $n\geq 2$.
\begin{theorem}\label{thm:refined-two-normal}
For every $n\geq 1$ we have $\widetilde{L}_n\trianglelefteq\widetilde{M}_n$, and the quotient $\widetilde{Q}_n:=\widetilde{M}_n/\widetilde{L}_n$ is an elementary abelian $2$-group generated by the images $\overline{\tilde{y}_n}$ and $\overline{\tilde{l}_n}$. More precisely, $\widetilde{Q}_1\cong C_2$ with $\overline{\tilde{y}_1}=\overline{\tilde{l}_1}$, and for $n\geq 2$
\[
  \widetilde{M}_n/\widetilde{L}_n\cong V_4=\{1,\ \overline{\tilde{y}_n},\ \overline{\tilde{l}_n},\ \overline{\tilde{y}_n\tilde{l}_n}\}.
\]
\end{theorem}
As in Section~\ref{4}, we attach types to the elements of $\widetilde{M}_n$; unlike there, we already need them in the proof of Theorem~\ref{thm:refined-two-normal}. When Theorem~\ref{thm:refined-two-normal} holds at level $n\geq 2$, we let $\tau\colon\widetilde{M}_n\to\{s,n\}^2$ be the homomorphism with kernel $\widetilde{L}_n$ sending $\tilde{y}_n,\tilde{l}_n,\tilde{y}_n\tilde{l}_n$ to $nn,ns,sn$; it induces an isomorphism $\widetilde{Q}_n\cong\{s,n\}^2$, where types are multiplied letter by letter. For $n=1$ we let $\tau\colon\widetilde{M}_1\to\{s,n\}$ be the sign. As in Subsection~\ref{sec:cocritical-types}, we write $\tau(a)_1$ and $\tau(a)_2$ for the first and the second letter of $\tau(a)$, and we identify $s$ and $n$ with $1$ and $-1$; thus $a\mapsto\tau(a)_2$ is a homomorphism $\widetilde{M}_n\to\{\pm1\}$ for $n\geq 2$.
\begin{lemma}\label{lem:refined-two-rho}
Let $n\geq 2$ and assume Theorem~\ref{thm:refined-two-normal} holds at level $n-1$. Every element of $\widetilde{M}_n$ has all its first-level sections in $\widetilde{M}_{n-1}$.
\begin{enumerate}[i)]
    \item For $n=2$ we have $\widetilde{L}_2=L_2$ and $\widetilde{M}_2=M_2=\mathrm{Aut}(\mathcal{T}_2)$. More precisely, $\tilde{y}_2\in L_2y_2$ and $\tilde{l}_2\in L_2l_2$.
    \item Let $n\geq 3$. For $\sigma=(a_0,a_1,a_2)\pi$ with $a_0,a_1,a_2\in\widetilde{M}_{n-1}$, we put $\mathrm{cl}(\sigma):=\overline{a_0}\,\overline{a_1}\,\overline{a_2}\in\widetilde{Q}_{n-1}$; as in Lemma~\ref{lem:structureL} i), $\mathrm{cl}$ is multiplicative. Every element $(a_0,a_1,a_2)\pi$ of $\widetilde{M}_n$ satisfies $\tau(a_0)_2=\tau(a_1)_2=\tau(a_2)_2=\mathrm{sgn}(\pi)$. Every element of $\widetilde{L}_n$ has even top permutation and $\mathrm{cl}=1$. Conversely, every $\sigma=(a_0,a_1,a_2)\pi$ with $a_0,a_1,a_2\in\widetilde{M}_{n-1}$, even top permutation, $\tau(a_0)_2=\tau(a_1)_2=\tau(a_2)_2=s$ and $\mathrm{cl}(\sigma)=1$ lies in $\widetilde{L}_n$. In particular, $\widetilde{L}_n$ contains $\widetilde{L}_{n-1}\times\widetilde{L}_{n-1}\times\widetilde{L}_{n-1}$.
\end{enumerate}
\end{lemma}

\begin{proof}
\emph{The sections.} The sections at the first level of the generators are $\mathrm{id}$, $x_{n-1}$, $\tilde{y}_{n-1}$, $\tilde{l}_{n-1}$, $\tilde{y}_{n-1}^{-1}x_{n-1}$, $\tilde{y}_{n-1}^{-1}\tilde{l}_{n-1}$ and the elements of $\widetilde{L}_{n-1}$, which lie in $\widetilde{M}_{n-1}$, and the sections of a product are products of sections.

\emph{i) The case $n=2$.} Since $\widetilde{L}_1=L_1=A_3$ and $\tilde{k}_2=k_2=z_2=((0,1),(0,1),(0,1,2))$, we have $\widetilde{L}_2=\langle x_2,z_2,k_2,L_1\rangle=L_2$. By Corollary~\ref{4.42}, $M_2=\mathrm{Aut}(\mathcal{T}_2)$. The sections of $\tilde{y}_2$ are three transpositions, and those of $\tilde{l}_2$ are two transpositions and the identity, so in the notation of Lemma~\ref{lem:structureL} we have $\mathrm{cl}(\tilde{y}_2)=\overline{y_1}=\mathrm{cl}(y_2)$ and $\mathrm{cl}(\tilde{l}_2)=1=\mathrm{cl}(l_2)$. These four elements have odd top permutation, so $\tilde{y}_2y_2^{-1}$ and $\tilde{l}_2l_2^{-1}$ lie in $L_2$ by Lemma~\ref{lem:structureL} i) and iii). Hence $\tilde{y}_2\in L_2y_2$, $\tilde{l}_2\in L_2l_2$ and $\widetilde{M}_2=\langle L_2,y_2,l_2\rangle=M_2$.

\emph{ii) The second letters.} Let $n\geq 3$. We have $\tau(x_{n-1})=ss$, $\tau(\tilde{y}_{n-1}^{-1}x_{n-1})=nn$ and $\tau(\tilde{y}_{n-1}^{-1}\tilde{l}_{n-1})=nn\cdot ns=sn$. Hence the types of the sections of $x_n$, $\tilde{k}_n$, $\tilde{y}_n$ and $\tilde{l}_n$ are $(ss,ss,ss)$, $(ns,ns,ss)$, $(nn,nn,nn)$ and $(nn,sn,nn)$, respectively, and those of $(g,\mathrm{id},\mathrm{id})$ with $g\in\widetilde{L}_{n-1}$ are $(ss,ss,ss)$. In each case the second letters are equal to the sign of the top permutation. By the multiplication rule, the section of $(a,\pi)(b,\rho)$ at the vertex $i$ is $a_ib_{(i)\pi}$, and
\[
  \tau(a_ib_{(i)\pi})_2=\tau(a_i)_2\,\tau(b_{(i)\pi})_2=\operatorname{sgn}(\pi)\operatorname{sgn}(\rho)=\operatorname{sgn}(\pi\rho)
\]
if both factors have this property, since $\tau$ is a homomorphism and types are multiplied letter by letter. So all the elements of $\widetilde{M}_n$ have it.

\emph{The elements of $\widetilde{L}_n$.} We check the generators. We have $\mathrm{cl}(x_n)=\overline{x_{n-1}}=1$, and $\mathrm{cl}(\tilde{k}_n)=\overline{\tilde{l}_{n-1}}^{\,2}\,\overline{x_{n-1}}=1$, because $\widetilde{Q}_{n-1}$ has exponent $2$. The generators $(g,\mathrm{id},\mathrm{id})$ with $g\in\widetilde{L}_{n-1}$ are clear, and all these generators have even top permutation.

\emph{The converse.} First, $\widetilde{L}_n$ contains $\widetilde{L}_{n-1}\times\widetilde{L}_{n-1}\times\widetilde{L}_{n-1}$: the sections of $x_n$ and $x_n^2$ at the vertex $0$ are trivial, so for $g\in\widetilde{L}_{n-1}$
\[
  (g,\mathrm{id},\mathrm{id})^{x_n}=(\mathrm{id},g,\mathrm{id}),
  \qquad
  (g,\mathrm{id},\mathrm{id})^{x_n^2}=(\mathrm{id},\mathrm{id},g).
\]
Identifying $\widetilde{Q}_{n-1}$ with $\{s,n\}^2$ through $\tau$, let $V:=\{v\in\{ss,ns\}^3\,:\,v_0v_1v_2=ss\}\cong V_4$. The classes of the sections of $\tilde{k}_n$ are $(ns,ns,ss)$, and conjugation by $x_n$ permutes the coordinates cyclically (conjugation by the section $x_{n-1}$ is trivial in the abelian group $\widetilde{Q}_{n-1}$). The three cyclic shifts of $(ns,ns,ss)$ are the three non-trivial elements of $V$. So, for every $v\in V$, there is an element of $\widetilde{L}_n$ with trivial top permutation whose sections have classes $v_0,v_1,v_2$. Now let $\sigma$ be as in the statement. Multiplying by a power of $x_n$, which changes neither $\mathrm{cl}$ nor the second letters of the types of the sections, we may assume that its top permutation is trivial. Then the classes of its sections form an element of $V$, and multiplying by the corresponding element of $\widetilde{L}_n$ we obtain an element with trivial top permutation whose sections lie in $\widetilde{L}_{n-1}$, which lies in $\widetilde{L}_n$. Hence $\sigma\in\widetilde{L}_n$.
\end{proof}

We proceed to prove Theorem~\ref{thm:refined-two-normal}.
\begin{proof}[Proof of Theorem~\ref{thm:refined-two-normal}]
We argue by induction on $n$. For $n=1$ we have $\widetilde{L}_1=\langle(0,1,2)\rangle\cong A_3$ and $\widetilde{M}_1=S_3$, so $\widetilde{L}_1\trianglelefteq\widetilde{M}_1$ and $\widetilde{M}_1/\widetilde{L}_1\cong C_2$, generated by $\overline{\tilde{y}_1}=\overline{\tilde{l}_1}=\overline{(0,1)}$. For $n=2$, by Lemma~\ref{lem:refined-two-rho} i) we have $\widetilde{L}_2=L_2$, $\widetilde{M}_2=M_2$, $\tilde{y}_2\in L_2y_2$ and $\tilde{l}_2\in L_2l_2$, so the statement is Theorem~\ref{th 4.39} for $n=2$.

Let $n\geq 3$ and assume the statement at level $n-1$. We write $x=x_{n-1}$, $\tilde{y}=\tilde{y}_{n-1}$ and $\tilde{l}=\tilde{l}_{n-1}$. We prove normality on generators, as in Theorem~\ref{th 4.39}. Let $g$ be a generator of $\widetilde{L}_n$ and let $w\in\{\tilde{y}_n,\tilde{l}_n\}$. Then $g^w$ has even top permutation, $\mathrm{cl}(g^w)=\mathrm{cl}(g)=1$, and the second letters of the types of its sections are $s$, because $g^w\in\widetilde{M}_n$. So $g^w\in\widetilde{L}_n$ by Lemma~\ref{lem:refined-two-rho} ii). For instance,
$$\tilde{k}_n^{\,\tilde{y}_n}=\bigl(x^{-1}\tilde{y}\tilde{l}\tilde{y}^{-1}x,\ \tilde{y}^{-1}\tilde{l}\tilde{y},\ \tilde{y}^{-1}x\tilde{y}\bigr).$$
Hence $\widetilde{L}_n$ is normal in $\widetilde{M}_n$. In the same way, $\tilde{y}_n^2$, $\tilde{l}_n^2$ and the commutator of $\tilde{y}_n$ and $\tilde{l}_n$ have even top permutation and $\mathrm{cl}=1$, and the second letters of the types of their sections are $s$, so they lie in $\widetilde{L}_n$. Thus $\widetilde{M}_n/\widetilde{L}_n$ is an elementary abelian $2$-group generated by $\overline{\tilde{y}_n}$ and $\overline{\tilde{l}_n}$. From the definitions we get, since $\widetilde{Q}_{n-1}$ is elementary abelian and $x\in\widetilde{L}_{n-1}$,
\[
  \mathrm{cl}(\tilde{y}_n)=\overline{\tilde{y}},
  \qquad
  \mathrm{cl}(\tilde{l}_n)=\overline{\tilde{l}}\,\overline{\tilde{y}},
  \qquad
  \mathrm{cl}(\tilde{y}_n\tilde{l}_n)=\overline{\tilde{l}},
\]
which are the formulas for $y_n$ and $l_n$ in the proof of Theorem~\ref{th 4.39}. The elements $\tilde{y}_n$ and $\tilde{l}_n$ have odd top permutation, and $\overline{\tilde{l}}\neq 1$ in $\widetilde{Q}_{n-1}\cong V_4$, so none of $\tilde{y}_n$, $\tilde{l}_n$, $\tilde{y}_n\tilde{l}_n$ lies in $\widetilde{L}_n$. Hence $\widetilde{M}_n/\widetilde{L}_n\cong V_4$, and the images of $\tilde{y}_n$, $\tilde{l}_n$, $\tilde{y}_n\tilde{l}_n$ are its three non-trivial elements. This completes the induction.
\end{proof}
 
We denote by $\widetilde{H}_n:=\langle\widetilde{L}_{n-1}\rangle^{\widetilde{L}_n}$ the normal closure of $\widetilde{L}_{n-1}$ in $\widetilde{L}_n$.
\begin{theorem}\label{thm:refined-two-H}
    Let $n\geq 2$. Then $\widetilde{H}_n=\widetilde{L}_{n-1}\times\widetilde{L}_{n-1}^{x_n}\times\widetilde{L}_{n-1}^{x_n^2}$, that is, $\widetilde{H}_n$ is the set of elements $(a_0,a_1,a_2)$ with $a_0,a_1,a_2\in\widetilde{L}_{n-1}$.
\end{theorem}
\begin{proof}
    For $n=2$ this is Theorem~\ref{thm:Hn-two}, since $\widetilde{L}_2=L_2$ and $\widetilde{L}_1=L_1$. For $n\geq 3$ we argue as in the proof of Theorem~\ref{th 4.22}: by Lemma~\ref{lem:refined-two-rho} ii), $N:=\widetilde{L}_{n-1}\times\widetilde{L}_{n-1}\times\widetilde{L}_{n-1}$ is contained in $\widetilde{L}_n$. It is normal in $\widetilde{L}_n$ on generators, since conjugation by $\tilde{k}_n$ acts on each coordinate by conjugation by an element of $\widetilde{M}_{n-1}$, and conjugation by $x_n$ permutes the coordinates and conjugates one of them by $x_{n-1}\in\widetilde{L}_{n-1}$. So $N$ contains $\widetilde{H}_n$; conversely, it is generated by $\widetilde{L}_{n-1}$ and its conjugates by $x_n$ and $x_n^2$.
\end{proof}
We construct an intermediate subgroup that will help us compute the index of $\widetilde{H}_n$ in $\widetilde{L}_n$. We define $\widetilde{J}_n=\langle \tilde{k}_n,\widetilde{L}_{n-1}\rangle ^{\widetilde{L}_n}$.
\begin{corollary}\label{cor:refined-two-K}
    For $n\geq 2$ we have $\widetilde{J}_n=\widetilde{L}_n\cap\mathrm{St}_1$, $[\widetilde{L}_n:\widetilde{J}_n]=3$, $\widetilde{J}_n/\widetilde{H}_n\cong V_4$ and $\widetilde{L}_n/\widetilde{H}_n\cong A_4$; in particular $[\widetilde{L}_n:\widetilde{H}_n]=12$. The cosets of $\widetilde{H}_n$ in $\widetilde{J}_n$ are $\widetilde{H}_n$, $\tilde{k}_n\widetilde{H}_n$, $\tilde{k}_n^{x_n}\widetilde{H}_n$ and $\tilde{k}_n^{x_n^2}\widetilde{H}_n$.
\end{corollary}
\begin{proof}
    Let $n\geq 3$, and let $\theta\colon\widetilde{L}_n\cap\mathrm{St}_1\to\widetilde{Q}_{n-1}^3$ be the map sending an element to the classes of its three sections, with $V=\{v\in\{ss,ns\}^3\,:\,v_0v_1v_2=ss\}\cong V_4$ as in the proof of Lemma~\ref{lem:refined-two-rho}. By the multiplication rule $\theta$ is a homomorphism. Its kernel is $\widetilde{L}_{n-1}^3=\widetilde{H}_n$ (Theorem~\ref{thm:refined-two-H}), and its image is $V$ by Lemma~\ref{lem:refined-two-rho} ii). The top permutations of the elements of $\widetilde{L}_n$ form the group $A_3$, so $[\widetilde{L}_n:\widetilde{L}_n\cap\mathrm{St}_1]=3$. The subgroup $\widetilde{L}_n\cap\mathrm{St}_1$ is normal in $\widetilde{L}_n$ and contains $\widetilde{H}_n$ and $\tilde{k}_n$, hence it contains $\widetilde{J}_n$. Conversely, the images of $\tilde{k}_n,\tilde{k}_n^{x_n},\tilde{k}_n^{x_n^2}$ are the three cyclic shifts of $(ns,ns,ss)$, which are the non-trivial elements of $V$, so $\widetilde{J}_n\supseteq\widetilde{L}_n\cap\mathrm{St}_1$ and the cosets of $\widetilde{H}_n$ in $\widetilde{J}_n$ are as stated. Hence $[\widetilde{L}_n:\widetilde{J}_n]=3$ and $\widetilde{J}_n/\widetilde{H}_n\cong V\cong V_4$. Finally, $x_n$ permutes the three non-trivial cosets cyclically, so $\widetilde{L}_n/\widetilde{H}_n\cong A_4$.

    For $n=2$ we have $\widetilde{L}_2=L_2$, $\widetilde{H}_2=H_2$ and $\widetilde{J}_2=J_2$, since $\tilde{k}_2=k_2=z_2$. By the proof of Corollary~\ref{4.42}, the map $\theta$ has kernel $H_2$ and image $V=\{v\in Q_1^3\,:\,v_0v_1v_2=1\}\cong V_4$, and the same argument applies, with the cyclic shifts of $(\overline{y_1},\overline{y_1},1)$, the image of $z_2$, in place of those of $(ns,ns,ss)$.
\end{proof}
\begin{lemma}\label{lem:uniform-refined-two}
    Let $n\geq 2$, let $\sigma=(g_0,g_1,g_2)\pi\in\widetilde{M}_n$, and write $\kappa:=[\widetilde{M}_{n-1}:\widetilde{L}_{n-1}]$, so that $\kappa=2$ if $n=2$ and $\kappa=4$ if $n\geq 3$ (Theorem~\ref{thm:refined-two-normal}).
    \begin{enumerate}[i)]
        \item If $\pi=\mathrm{id}$, then $CD(\widetilde{H}_n\sigma,\widetilde{M}_n)=\frac{\kappa^3}{48}\,CD(g_0\widetilde{L}_{n-1},\widetilde{M}_{n-1})\times CD(g_1\widetilde{L}_{n-1},\widetilde{M}_{n-1})\times CD(g_2\widetilde{L}_{n-1},\widetilde{M}_{n-1})$.
        \item If $\pi$ is a $3$-cycle, then $CD(\widetilde{H}_n\sigma,\widetilde{M}_n)=\frac{\kappa}{48}\,t\,CD(g_0g_1g_2\widetilde{L}_{n-1},\widetilde{M}_{n-1})$.
        \item If $\pi$ is a transposition exchanging $i$ and $j$ and fixing $k$, then $$CD(\widetilde{H}_n\sigma,\widetilde{M}_n)=\tfrac{\kappa^2}{48}\,d\,CD(g_ig_j\widetilde{L}_{n-1},\widetilde{M}_{n-1})\times CD(g_k\widetilde{L}_{n-1},\widetilde{M}_{n-1}).$$
    \end{enumerate}
    In particular, the constants in i), ii) and iii) are $\frac{4}{3}$, $\frac{1}{12}$ and $\frac{1}{3}$ for $n\geq 3$, and $\frac{1}{6}$, $\frac{1}{24}$ and $\frac{1}{12}$ for $n=2$.
\end{lemma}
\begin{proof}
    By Lemma~\ref{lem:refined-two-rho}, $g_0,g_1,g_2\in\widetilde{M}_{n-1}$; by Theorem~\ref{thm:refined-two-H}, $\widetilde{H}_n=\widetilde{L}_{n-1}\times\widetilde{L}_{n-1}\times\widetilde{L}_{n-1}$; and by Theorem~\ref{thm:refined-two-normal}, $\widetilde{L}_{n-1}$ is normal in $\widetilde{M}_{n-1}$ and $\widetilde{M}_{n-1}/\widetilde{L}_{n-1}$ is abelian. As in the proof of Lemma~\ref{lem:uniform-two}, these are the only facts used in the proof of Lemma~\ref{lem:uniform}, which therefore applies word for word: for every cycle type, the proportion of the elements of $\widetilde{H}_n\sigma$ with this cycle type equals the proportion of the tuples of return sections $(b_C)_C\in\prod_C g_{i_1}\cdots g_{i_r}\widetilde{L}_{n-1}$ which give this cycle type by Lemma~\ref{lem:return}. In ii) the return section lies in $g_{i_1}g_{i_2}g_{i_3}\widetilde{L}_{n-1}=g_0g_1g_2\widetilde{L}_{n-1}$, since $\widetilde{M}_{n-1}/\widetilde{L}_{n-1}$ is abelian. Only the constants change: by Theorem~\ref{thm:refined-two-normal} and Corollary~\ref{cor:refined-two-K} we have $[\widetilde{M}_n:\widetilde{H}_n]=4\cdot 12=48$, and $|g\widetilde{L}_{n-1}|/|\widetilde{M}_{n-1}|=\frac{1}{\kappa}$ for $g\in\widetilde{M}_{n-1}$. This gives the constants $\frac{\kappa^3}{48}$, $\frac{\kappa}{48}$ and $\frac{\kappa^2}{48}$. Note that $|\widetilde{H}_n\sigma|/|\widetilde{M}_n|=\frac{1}{48}$ in all cases.
\end{proof}
We introduce the following notation for the cycle data of the refined model.
\[
\resizebox{\textwidth}{!}{$
\begin{aligned}
\widetilde{A}_1^{(n)} &:= \textit{the cycle data obtained by applying the refined Markov process $n-1$ times to } ([ss,1]^3,\tfrac{1}{12}); \\
\widetilde{A}_2^{(n)} &:= \textit{the cycle data obtained by applying the refined Markov process $n-1$ times to } ([ss,3],\tfrac{2}{3}); \\
\widetilde{A}_3^{(n)} &:= \textit{the cycle data obtained by applying the refined Markov process $n-1$ times to } ([ns,1]^2[ss,1],\tfrac{1}{4}); \\
\widetilde{A}_4^{(n)} &:= \textit{the cycle data obtained by applying the refined Markov process $n-1$ times to } ([nn,1],\tfrac{1}{4}); \\
\widetilde{A}_5^{(n)} &:= \textit{the cycle data obtained by applying the refined Markov process $n-1$ times to } ([ss,1],\tfrac{1}{4}); \\
\widetilde{A}_6^{(n)} &:= \textit{the cycle data obtained by applying the refined Markov process $n-1$ times to } ([sn,1],\tfrac{1}{4}); \\
\widetilde{A}_7^{(n)} &:= \textit{the cycle data obtained by applying the refined Markov process $n-1$ times to } ([ns,1],\tfrac{1}{4}).
\end{aligned}
$}
\]
We also write $\widetilde{B}_c^{(n)}$, for a type $c\in\{s,n\}^2$, for the cycle data obtained by applying the refined Markov process $n-1$ times to $([c,1],1)$. Thus $\widetilde{B}_{ss}^{(n)}=4\widetilde{A}_5^{(n)}$, $\widetilde{B}_{nn}^{(n)}=4\widetilde{A}_4^{(n)}$, $\widetilde{B}_{sn}^{(n)}=4\widetilde{A}_6^{(n)}$, $\widetilde{B}_{ns}^{(n)}=4\widetilde{A}_7^{(n)}$, and $\widetilde{B}_c^{(1)}=\{([1],1)\}$ for every $c$.
 
As in Section~\ref{4}, the types of the elements of $\widetilde{M}_n$ describe the cosets of $\widetilde{L}_n$. Let $n\geq 3$, and let $\sigma=(a_0,a_1,a_2)\pi\in\widetilde{M}_n$ have type $c$. Then
\begin{equation}\label{eq:types-refined-two}
    \mathrm{sgn}(\pi)=c_1,\quad \tau(a_0)_2=\tau(a_1)_2=\tau(a_2)_2=c_1\quad\text{and}\quad \tau(a_0)\tau(a_1)\tau(a_2)=f(c),
\end{equation}
where $f(c_1c_2)=c_2c_1$. Indeed, the condition on the second letters is Lemma~\ref{lem:refined-two-rho} ii). The other two equalities are multiplicative in $\sigma$, so we only have to check the generators. For the generators of $\widetilde{L}_n$, of type $ss$, this is Lemma~\ref{lem:refined-two-rho} ii). The elements $\tilde{y}_n$ and $\tilde{l}_n$ have odd top permutation, and the types of their sections, given in the proof of that lemma, have products $nn=f(nn)$ and $sn=f(ns)$.

Compared with \eqref{eq:types-two}, the condition on the second letters is new; it is exactly Proposition~\ref{prop:cocritical-types}. For $n=2$, by Lemma~\ref{lem:refined-two-rho} i), $\widetilde{M}_2=M_2$ and $\tau$ is the map $\tau$ of Section~\ref{4}, so \eqref{eq:types-two} for $n=2$ applies.
    \begin{theorem}\label{thm:refined-two-data}
        Let $n\geq 2$ and $c\in\{s,n\}^2$. If $\sigma$ is chosen uniformly in $\tau^{-1}(c)$, that is, in $\widetilde{L}_n$, $\widetilde{L}_n\tilde{y}_n$, $\widetilde{L}_n\tilde{l}_n$ or $\widetilde{L}_n\tilde{y}_n\tilde{l}_n$ for $c=ss,nn,ns,sn$ respectively, then the cycle type of $\sigma$ is distributed as the data obtained by applying the refined Markov process of Definition~\ref{def:refined-model} $n$ times to $[c,1]$. In other words,
        \begin{align*}
            &CD(\widetilde{L}_n,\widetilde{M}_n)=\widetilde{A}_5^{(n+1)}, && CD(\widetilde{L}_n\tilde{y}_n,\widetilde{M}_n)=\widetilde{A}_4^{(n+1)},\\
            &CD(\widetilde{L}_n\tilde{l}_n,\widetilde{M}_n)=\widetilde{A}_7^{(n+1)}, && CD(\widetilde{L}_n\tilde{y}_n\tilde{l}_n,\widetilde{M}_n)=\widetilde{A}_6^{(n+1)}.
        \end{align*}
        Moreover:
        \begin{enumerate}[i)]
            \item $\widetilde{A}_1^{(n)}=CD(\widetilde{H}_n,\widetilde{L}_n);$
            \item $\widetilde{A}_2^{(n)}=CD(\widetilde{J}_nx_n\sqcup\widetilde{J}_nx_n^2,\widetilde{L}_n);$
            \item $\widetilde{A}_3^{(n)}=CD(\tilde{k}_n\widetilde{H}_n\sqcup\tilde{k}_n^{x_n}\widetilde{H}_n\sqcup\tilde{k}_n^{x_n^2}\widetilde{H}_n,\widetilde{L}_n)$.
        \end{enumerate}
    \end{theorem}
\begin{proof}
We use a separate induction for the refined groups, since their co-critical section condition is stronger than the length-two condition in \eqref{eq:types-two}. The subgroup $\widetilde H_n=\widetilde L_{n-1}^3$ carries the distributions known by induction; the quotient classes represented by $\tilde y_n$ and $\tilde l_n$ restrict the admissible section types and their multiplicities. The refined Markov transitions of Definition~\ref{def:refined-model} give
\begin{equation}\label{eq:markov-side-refined}
 \widetilde A_1^{(n)}=\tfrac1{12}(\widetilde B_{ss}^{(n)})^{\times3},\qquad
 \widetilde A_2^{(n)}=\tfrac23\,t\widetilde B_{ss}^{(n)},\qquad
 \widetilde A_3^{(n)}=\tfrac14(\widetilde B_{ns}^{(n)})^{\times2}\times\widetilde B_{ss}^{(n)},
\end{equation}
and, with $f(c_1c_2)=c_2c_1$,
\begin{equation}\label{eq:markov-step-refined}
 \widetilde B_c^{(n+1)}=\begin{cases}
 \tfrac23\,t\widetilde B_{f(c)}^{(n)}\ \sqcup\ \tfrac1{12}\displaystyle\bigsqcup_{d_0d_1d_2=f(c)}\widetilde B_{d_0}^{(n)}\times\widetilde B_{d_1}^{(n)}\times\widetilde B_{d_2}^{(n)},&c_1=s,\\[10pt]
 \tfrac12\displaystyle\bigsqcup_{e_1e_2=f(c)}d\widetilde B_{e_1}^{(n)}\times\widetilde B_{e_2}^{(n)},&c_1=n.
 \end{cases}
\end{equation}
The first union runs over the four ordered triples whose second letters are all $s$, and the second over the two ordered pairs with second letters $s,n$, respectively, as prescribed by Proposition~\ref{prop:cocritical-types}. For $c=ss$, the four triples are $(ss,ss,ss)$ and the three permutations of $(ns,ns,ss)$; hence
\begin{equation}\label{eq:markov-ss-refined}
 \widetilde A_5^{(n+1)}=\tfrac14\widetilde B_{ss}^{(n+1)}=\tfrac14\bigl(\widetilde A_1^{(n)}\sqcup\widetilde A_2^{(n)}\sqcup\widetilde A_3^{(n)}\bigr).
\end{equation}

\emph{Base case $n=2$.} By Lemma~\ref{lem:refined-two-rho} i), $\widetilde M_2=M_2$ and $\widetilde L_2=L_2$ with the same type character as Section~\ref{4}. One Markov step depends only on the first letter, and the allowed four triples (respectively two pairs) in the refined model have exactly the four sign patterns (respectively two sign patterns) used in \eqref{eq:markov-two-n2}. Thus $\widetilde B_c^{(3)}=B_c^{(3)}$, and the assertion $CD(\tau^{-1}(c),\widetilde M_2)=\frac14\widetilde B_c^{(3)}$ follows from Theorem~\ref{4.44}. The sets i)--iii) at $n=2$ are $\widetilde H_2$, the eight $3$-cycle-top cosets and the three nontrivial identity-top cosets of $\widetilde L_2$, so the same count proves the additional assertions at this level.

\emph{Inductive step.} Let $n\geq3$ and assume the theorem at level $n-1$. Put $\widetilde L=\widetilde L_{n-1}$ and $\widetilde U=\{\mathrm{id},\tilde y_{n-1},\tilde l_{n-1},\tilde y_{n-1}\tilde l_{n-1}\}$. By Theorems~\ref{thm:refined-two-normal} and \ref{thm:refined-two-H}, $\widetilde H_n=\widetilde L^3$, $\widetilde M_{n-1}/\widetilde L\cong V_4$ and $[\widetilde M_n:\widetilde H_n]=48$. The induction hypothesis is
\begin{equation}\label{eq:coset-L-refined}
 CD(g\widetilde L,\widetilde M_{n-1})=\tfrac14\widetilde B_{\tau(g)}^{(n)}\qquad(g\in\widetilde M_{n-1}).
\end{equation}
Every coset of $\widetilde H_n$ has the form $\widetilde H_n(u_0,u_1,u_2)\pi$, with $u_i\in\widetilde U$ and $\pi\in S_3$, subject to \eqref{eq:types-refined-two}. For a fixed parent type $c$, this means $\mathrm{sgn}(\pi)=c_1$, all three section types have second letter $c_1$, and their product is $f(c)$. Two of the four elements of $\widetilde U$ have each specified second letter. Hence each $\tau^{-1}(c)$ consists of $3\cdot2^2=12$ cosets, giving all $48$ cosets of $\widetilde H_n$. By Lemma~\ref{lem:uniform-refined-two} and \eqref{eq:coset-L-refined}, each coset contributes $1/48$ times its normalised return-section distribution.

If $c_1=s$, four of these cosets have trivial top permutation and eight have a $3$-cycle. Therefore
\[
 CD(\tau^{-1}(c),\widetilde M_n)=\tfrac1{48}\bigsqcup_{d_0d_1d_2=f(c)}\widetilde B_{d_0}^{(n)}\times\widetilde B_{d_1}^{(n)}\times\widetilde B_{d_2}^{(n)}\ \sqcup\ \tfrac8{48}\,t\widetilde B_{f(c)}^{(n)}=\tfrac14\widetilde B_c^{(n+1)}.
\]
If $c_1=n$, all $12$ cosets have a transposition at the top. For every prescribed allowed ordered pair $(e_1,e_2)$ of return-section types, there are two choices for the first exchanged section for a fixed transposition, hence six cosets in total. Consequently
\[
 CD(\tau^{-1}(c),\widetilde M_n)=\tfrac6{48}\bigsqcup_{e_1e_2=f(c)}d\widetilde B_{e_1}^{(n)}\times\widetilde B_{e_2}^{(n)}=\tfrac14\widetilde B_c^{(n+1)}.
\]
These match \eqref{eq:markov-step-refined}.

For the refinement i)--iii), the $12$ cosets in $\widetilde L_n=\tau^{-1}(ss)$ consist of $\widetilde H_n$, eight $3$-cycle-top cosets making up $\widetilde J_nx_n\sqcup\widetilde J_nx_n^2$, and three identity-top cosets whose section types are the permutations of $(ns,ns,ss)$. Those last cosets are represented by $\tilde k_n=(\tilde l_{n-1},\tilde l_{n-1},x_{n-1})$ and its conjugates by $x_n$. Relative to $\widetilde L_n$, their contributions have weights $1/12$, $2/3$ and $1/4$, respectively, and their return-section distributions are those in \eqref{eq:markov-side-refined}. Thus they give $\widetilde A_1^{(n)}$, $\widetilde A_2^{(n)}$ and $\widetilde A_3^{(n)}$.
\end{proof}
 
This allows us to propose the following refined models for the conjugates $f_t$ of such a cubic $f$, depending on irreducibility and on which of $P_1(t)$, $P_2(t)$ and $P_1(t)P_2(t)$ are squares in $K$. As in Section~\ref{4}, in each case the root types are equidistributed among the types $c$ with $c_1=s$ if $P_1(t)$ is a square, $c_2=s$ if $P_2(t)$ is a square, and $c_1=c_2$ if $P_1(t)P_2(t)$ is a square.
\begin{enumerate}
    \item \textbf{Model~1.} If $f_t$ is reducible and $P_1(t)$ and $P_2(t)$ are squares, then the polynomial splits completely and the root type is $ss$. The initial data are
\[\begin{aligned}
    &[[ss,1]^{3},\frac{1}{4}],\ [[ns,1]^{2}[ss,1],\frac{3}{4}].
\end{aligned}\]
    The corresponding group is $M_{n,t}=\widetilde J_n$.
    \item \textbf{Model~2.} If $f_t$ is irreducible and $P_1(t)$ and $P_2(t)$ are squares, then the root type is $ss$. The initial data are
\[\begin{aligned}
    &[[ss,3],\frac{2}{3}],\ [[ns,1]^{2}[ss,1],\frac{1}{4}],\ [[ss,1]^{3},\frac{1}{12}].
\end{aligned}\]
    The corresponding group is $M_{n,t}=\widetilde L_n$.
    \item \textbf{Model~3.} If $f_t$ is irreducible, $P_1(t)$ is a square and $P_2(t)$ is not, then the root types are $ss$, $sn$, each with the same frequency. The initial data are
\[\begin{aligned}
    &[[ns,3],\frac{1}{3}],\ [[ss,3],\frac{1}{3}],\ [[ns,1]^{3},\frac{1}{24}] ,\\    &[[ns,1]^{2}[ss,1],\frac{1}{8}],\ [[ns,1][ss,1]^{2},\frac{1}{8}],\ [[ss,1]^{3},\frac{1}{24}].
\end{aligned}\]
    The corresponding group is $M_{n,t}=\langle\widetilde L_n,\tilde y_n\tilde l_n\rangle$.
    \item \textbf{Model~4.} If $f_t$ is irreducible, $P_2(t)$ is a square and $P_1(t)$ is not, then the root types are $ss$, $ns$, each with the same frequency. The initial data are
\[\begin{aligned}
    &[[ss,3],\frac{1}{3}],\ [[ns,2][nn,1],\frac{1}{4}],\ [[ss,2][sn,1],\frac{1}{4}] ,\\    &[[ns,1]^{2}[ss,1],\frac{1}{8}],\ [[ss,1]^{3},\frac{1}{24}].
\end{aligned}\]
    The corresponding group is $M_{n,t}=\langle\widetilde L_n,\tilde l_n\rangle$.
    \item \textbf{Model~5.} If $f_t$ is irreducible, $P_1(t)P_2(t)$ is a square but neither factor is, then the root types are $ss$, $nn$, each with the same frequency. The initial data are
\[\begin{aligned}
    &[[ss,3],\frac{1}{3}],\ [[ss,2][nn,1],\frac{1}{4}],\ [[ns,2][sn,1],\frac{1}{4}] ,\\    &[[ns,1]^{2}[ss,1],\frac{1}{8}],\ [[ss,1]^{3},\frac{1}{24}].
\end{aligned}\]
    The corresponding group is $M_{n,t}=\langle\widetilde L_n,\tilde y_n\rangle$.
    \item \textbf{Model~6.} If $f_t$ is irreducible and none of $P_1(t)$, $P_2(t)$ and $P_1(t)P_2(t)$ is a square, then the root types are $ss$, $sn$, $ns$, $nn$, each with the same frequency. The initial data are
\[\begin{aligned}
    &[[ns,3],\frac{1}{6}],\ [[ss,3],\frac{1}{6}],\ [[ns,2][nn,1],\frac{1}{8}] ,\\    &[[ss,2][nn,1],\frac{1}{8}],\ [[ns,2][sn,1],\frac{1}{8}],\ [[ss,2][sn,1],\frac{1}{8}] ,\\    &[[ns,1]^{3},\frac{1}{48}],\ [[ns,1]^{2}[ss,1],\frac{1}{16}],\ [[ns,1][ss,1]^{2},\frac{1}{16}] ,\\    &[[ss,1]^{3},\frac{1}{48}].
\end{aligned}\]
    The corresponding group is $M_{n,t}=\widetilde M_n$.
\end{enumerate}
The displayed initial data are the first-step distributions of the refined process, averaged over the listed root types. Other reducible cases can be described using subgroups fixing the roots of $f_t$; we do not enumerate them.
\begin{theorem}\label{thm:refined-two-models}
    For each of the Models~1--6 above and every $n\geq 2$, the cycle data of the group $M_{n,t}$ are the data obtained by applying the refined Markov process $n-1$ times to the initial data of the model.
\end{theorem}
\begin{proof}
    In Models~2--6 the group is the union of the cosets $\tau^{-1}(c)$ for the root types $c$ of the model, which all have the same size, and the initial data are the average of the first steps from the corresponding $[c,1]$. Hence the result follows from Theorem~\ref{thm:refined-two-data}. For Model~1, $\widetilde{J}_n=\widetilde{L}_n\cap\mathrm{St}_1$ by Corollary~\ref{cor:refined-two-K}, so $\widetilde{J}_n$ is the union of the sets in i) and iii) of Theorem~\ref{thm:refined-two-data}, and $|\widetilde{L}_n|=3|\widetilde{J}_n|$. Hence, by \eqref{eq:markov-side-refined},
    $$CD(\widetilde{J}_n,\widetilde{J}_n)=3\bigl(\widetilde{A}_1^{(n)}\sqcup\widetilde{A}_3^{(n)}\bigr)=\tfrac{1}{4}\,\widetilde{B}_{ss}^{(n)}\times\widetilde{B}_{ss}^{(n)}\times\widetilde{B}_{ss}^{(n)}\ \sqcup\ \tfrac{3}{4}\,\widetilde{B}_{ns}^{(n)}\times\widetilde{B}_{ns}^{(n)}\times\widetilde{B}_{ss}^{(n)},$$
    which is the data obtained by applying the refined Markov process $n-1$ times to the initial data of Model~1.
\end{proof}

\section{Galois groups and the Markov groups}\label{5}
In this section we relate the Markov groups of Sections~\ref{3}, \ref{4}, \ref{sec:three} and \ref{sec:refined-two} to the Galois groups of the iterates. We describe the Markov groups by conditions on the signs of the permutations they induce on the levels of the tree, which identifies the group of Section~\ref{3} with the group $E_n$ of~\cite{Benedetto2017}. These sign conditions come from discriminants, and this allows us to prove that the Galois groups are contained in the Markov groups (Theorems~\ref{thm:upper-bound} and \ref{thm:refined-upper-bound}). We then treat the Belyi map, for which equality holds, and we state our conjecture.

To distinguish the constructions, we write $M_n^{(1)}$, $M_n^{(2)}$ and $M_n^{(3)}$ for the groups of Sections~\ref{3}, \ref{4} and \ref{sec:three}, and $\widetilde{M}_n$ for the group of Section~\ref{sec:refined-two}; when no confusion is possible we drop the superscript. The superscripts refer to the length of the combined critical orbit, and the group of the model of $f_t$ is denoted $M_{n,t}$ in all cases.

\subsection{The Markov groups as discriminant groups}\label{5.1}
For $\sigma\in \mathrm{Aut}(\mathcal{T}_n)$ and $1\leq k\leq n$ we write $\mathrm{sgn}_k(\sigma)\in \{\pm 1\}$ for the sign of the permutation induced by $\sigma$ on the $3^k$ vertices of level $k$, and we set $\mathrm{sgn}_0(\sigma)=1$. Each $\mathrm{sgn}_k$ is a homomorphism. For $\sigma=(\sigma_0,\sigma_1,\sigma_2)\pi$ we have
\begin{equation}\label{eq:sign-recursion}
    \mathrm{sgn}_k(\sigma)=\mathrm{sgn}(\pi)\prod_{i=0}^{2}\mathrm{sgn}_{k-1}(\sigma_i)\qquad (1\leq k\leq n)
\end{equation}
(see also~\cite[Lemma 2.1]{Benedetto2017}). Indeed, $\sigma$ acts on the vertices $iw$ of level $k$, with $i\in X$ and $w\in X^{k-1}$, as the composition of $iw\mapsto i\,(w)\sigma_i$, whose sign is $\prod_i\mathrm{sgn}_{k-1}(\sigma_i)$, and $iw\mapsto (i)\pi\,w$, whose sign is $\mathrm{sgn}(\pi)^{3^{k-1}}=\mathrm{sgn}(\pi)$.
For a subgroup $G\subseteq\mathrm{Aut}(\mathcal{T}_{n-1})$ we write $G\wr S_3=\{\sigma\in \mathrm{Aut}(\mathcal{T}_n)\,|\,\sigma_0,\sigma_1,\sigma_2\in G\}$. We also put $\chi:=\mathrm{sgn}_1\mathrm{sgn}_2$.
\begin{definition}\label{def:discriminant-groups}
    Following~\cite[Section 2]{Benedetto2017}, let $E_1=S_3$ and $E_n=(E_{n-1}\wr S_3)\cap \ker(\mathrm{sgn}_2)$ for $n\geq 2$. Similarly, let $D_n=E'_n=\widetilde{D}_n=\mathrm{Aut}(\mathcal{T}_n)$ for $n\leq 2$, and for $n\geq 3$ let
    \begin{align*}
        D_n&=(D_{n-1}\wr S_3)\cap \ker(\mathrm{sgn}_1\mathrm{sgn}_2\mathrm{sgn}_3),\\
        E'_n&=(E'_{n-1}\wr S_3)\cap\ker(\mathrm{sgn}_3),\\
        \widetilde{D}_n&=\bigl\{\sigma=(\sigma_0,\sigma_1,\sigma_2)\pi\in\widetilde{D}_{n-1}\wr S_3\ \big|\ \chi(\sigma_i)=\mathrm{sgn}(\pi)\text{ for }i=0,1,2\bigr\}.
    \end{align*}
\end{definition}
Unwinding the recursion, $\sigma\in E_n$ if and only if $\mathrm{sgn}_2(\sigma_v)=1$ for every vertex $v$ of depth at most $n-2$; $\sigma\in D_n$, respectively $\sigma\in E'_n$, if and only if $\mathrm{sgn}_1\mathrm{sgn}_2\mathrm{sgn}_3(\sigma_v)=1$, respectively $\mathrm{sgn}_3(\sigma_v)=1$, for every vertex $v$ of depth at most $n-3$; and $\sigma\in\widetilde{D}_n$ if and only if $\chi(\sigma_w)=\mathrm{sgn}(\pi_v)$ for every vertex $v$ of depth at most $n-3$ and every child $w$ of $v$, where $\sigma_v$ is the section of $\sigma$ at $v$ and $\pi_v$ is the top permutation of $\sigma_v$. Since $\chi$ is a homomorphism, the multiplication rule \eqref{eq:multiplication} shows that $\widetilde{D}_n$ is a subgroup of $\mathrm{Aut}(\mathcal{T}_n)$.
We first compute the signs of the generators, by applying \eqref{eq:sign-recursion} to their recursive definitions.
\begin{lemma}\label{lem:generator-signs}
    Let $m\geq 1$ and $1\leq j\leq m$.
    \begin{enumerate}[i)]
        \item For the elements of Sections~\ref{3} and \ref{4} we have $\mathrm{sgn}_j(x_m)=\mathrm{sgn}_j(z_m)=\mathrm{sgn}_j(k_m)=1$, $\mathrm{sgn}_j(y_m)=(-1)^j$ and $\mathrm{sgn}_j(l_m)=(-1)^{j(j+1)/2}$.
        \item For the elements of Section~\ref{sec:three} we have $\mathrm{sgn}_j(\hat{p}_m)=\mathrm{sgn}_j(\hat{q}_m)=1$ and
        $$\mathrm{sgn}_j(p_m)=\begin{cases}1 & \text{if } j\equiv 0,\\ -1 & \text{if } j\equiv 1,2,\end{cases}\qquad
        \mathrm{sgn}_j(q_m)=\begin{cases}1 & \text{if } j\equiv 0,2,\\ -1 & \text{if } j\equiv 1,\end{cases}\qquad
        \mathrm{sgn}_j(r_m)=\begin{cases}1 & \text{if } j\equiv 0,1,\\ -1 & \text{if } j\equiv 2,\end{cases}$$
        where the congruences are modulo $3$.
        \item For the elements of Section~\ref{sec:refined-two} we have $\mathrm{sgn}_j(\tilde{y}_m)=(-1)^j$, $\mathrm{sgn}_j(\tilde{l}_m)=(-1)^{j(j+1)/2}$ and $\mathrm{sgn}_j(\tilde{k}_m)=1$.
        \item $\mathrm{sgn}_1$ and $\mathrm{sgn}_2$ are trivial on the groups $L_m$ of Sections~\ref{3}, \ref{4} and \ref{sec:refined-two}, and $\mathrm{sgn}_1$, $\mathrm{sgn}_2$ and $\mathrm{sgn}_3$ are trivial on the group $L_m$ of Section~\ref{sec:three} (where defined).
    \end{enumerate}
\end{lemma}
\begin{proof}
    We use induction on $m$ and \eqref{eq:sign-recursion}. The element $x_m=(\mathrm{id},\mathrm{id},x_{m-1})(0,1,2)$ gives $\mathrm{sgn}_j(x_m)=\mathrm{sgn}_{j-1}(x_{m-1})=1$. For i), we get $\mathrm{sgn}_j(y_m)=-\mathrm{sgn}_{j-1}(y_{m-1})\mathrm{sgn}_{j-1}(x_{m-1})=-(-1)^{j-1}$,
    $$\mathrm{sgn}_j(l_m)=-\mathrm{sgn}_{j-1}(l_{m-1})\mathrm{sgn}_{j-1}(y_{m-1})=-(-1)^{(j-1)j/2+j-1}=(-1)^{j(j+1)/2},$$
    and $\mathrm{sgn}_j(z_m)=\mathrm{sgn}_{j-1}(y_{m-1})^2\mathrm{sgn}_{j-1}(x_{m-1})=1$, and the same for $k_m$. For ii), the recursions are
    \begin{gather*}
        \mathrm{sgn}_j(p_m)=-\mathrm{sgn}_{j-1}(p_{m-1})\,\mathrm{sgn}_{j-1}(q_{m-1}),\qquad \mathrm{sgn}_j(q_m)=-\mathrm{sgn}_{j-1}(q_{m-1})\,\mathrm{sgn}_{j-1}(r_{m-1}),\\
        \mathrm{sgn}_j(r_m)=\mathrm{sgn}_{j-1}(q_{m-1}),
    \end{gather*}
    and $\mathrm{sgn}_j(\hat{p}_m)=\mathrm{sgn}_{j-1}(p_{m-1})^2\mathrm{sgn}_{j-1}(x_{m-1})=1$, and the same for $\hat{q}_m$. The stated values satisfy these recursions, with $\mathrm{sgn}_0=1$, as one checks for each residue of $j$ modulo $3$. For iii), the recursions for $\tilde{y}_m$, $\tilde{l}_m$ and $\tilde{k}_m$ are those for $y_m$, $l_m$ and $k_m$:
    $$\mathrm{sgn}_j(\tilde{y}_m)=-\mathrm{sgn}_{j-1}(\tilde{y}_{m-1})^2\,\mathrm{sgn}_{j-1}(\tilde{y}_{m-1})^{-1}\mathrm{sgn}_{j-1}(x_{m-1})=-\mathrm{sgn}_{j-1}(\tilde{y}_{m-1}),$$
    and similarly $\mathrm{sgn}_j(\tilde{l}_m)=-\mathrm{sgn}_{j-1}(\tilde{l}_{m-1})\mathrm{sgn}_{j-1}(\tilde{y}_{m-1})$. For iv) it suffices to check the generators of $L_m$. The elements $x_m,z_m,k_m,\hat{p}_m,\hat{q}_m,\tilde{k}_m$ have been treated, and for $g\in L_{m-1}$ we have $\mathrm{sgn}_1((g,\mathrm{id},\mathrm{id}))=1$ and $\mathrm{sgn}_j((g,\mathrm{id},\mathrm{id}))=\mathrm{sgn}_{j-1}(g)$ for $j\geq 2$, which is $1$ by induction.
\end{proof}
\begin{theorem}\label{thm:Markov-discriminant}
    For every $n\geq 1$ we have $M_n^{(1)}=E_n$, $M_n^{(2)}=D_n$, $M_n^{(3)}=E'_n$ and $\widetilde{M}_n=\widetilde{D}_n$. Moreover, for $n\geq 2$ and identifying $1,-1$ with $s,n$, the map $\tau$ of Sections~\ref{4}, \ref{sec:three} and \ref{sec:refined-two} is given by
    $$\tau=(\mathrm{sgn}_1,\ \mathrm{sgn}_1\mathrm{sgn}_2)\ \text{ on } D_n \text{ and on } \widetilde{D}_n,\qquad \tau=(\mathrm{sgn}_1,\ \mathrm{sgn}_1\mathrm{sgn}_2,\ \mathrm{sgn}_2)\ \text{ on } E'_n.$$
    In particular, $\chi$ is the second letter of $\tau$. Finally, $\widetilde{D}_n\subseteq D_n$.
\end{theorem}
\begin{proof}
    \emph{The map $\tau$.} Let $n\geq 2$. By Lemma~\ref{lem:generator-signs}, $(\mathrm{sgn}_1,\mathrm{sgn}_1\mathrm{sgn}_2)$ is a homomorphism on $M_n^{(2)}$ which is trivial on $L_n$ and sends $y_n$ to $(-1,-1)$, that is, $nn$, and $l_n$ to $(-1,1)$, that is, $ns$; so it agrees with $\tau$. The same computation with $\tilde{y}_n$ and $\tilde{l}_n$ gives the formula on $\widetilde{M}_n$. In the same way $(\mathrm{sgn}_1,\mathrm{sgn}_1\mathrm{sgn}_2,\mathrm{sgn}_2)$ is trivial on the group $L_n$ of Section~\ref{sec:three} and sends $p_n,q_n,r_n$ to $nsn,nns,snn$. Moreover, for length one, $\tau=\mathrm{sgn}_1$ by \eqref{eq:types-one}. Thus, once the equalities of the theorem are proved, we have the formulas for $\tau$ and the fact that $\chi$ is the second letter of $\tau$.

    \emph{The equalities.} We use induction on $n$. Let $G_n$ be one of the groups $E_n$, $D_n$, $E'_n$, $\widetilde{D}_n$, and let $M_n$ be the corresponding Markov group. We have $M_1=S_3=G_1$, and for the last three groups $M_2=\mathrm{Aut}(\mathcal{T}_2)=G_2$, by Corollaries~\ref{4.42}, \ref{cor:three-K} and Lemma~\ref{lem:refined-two-rho} i). So let $n\geq 2$ for $E_n$ and $n\geq 3$ for the other groups, and assume that $M_{n-1}=G_{n-1}$.

    We first prove that $M_n\subseteq G_n$. The first-level sections of the generators of $M_n$ lie in $M_{n-1}=G_{n-1}$, by Lemmas~\ref{lem:psi-one}, \ref{lem:structureL}, \ref{lem:three-psi} and \ref{lem:refined-two-rho}. It remains to check the sign conditions on the generators, since $\mathrm{sgn}_2$, $\mathrm{sgn}_1\mathrm{sgn}_2\mathrm{sgn}_3$ and $\mathrm{sgn}_3$ are homomorphisms.
    \begin{itemize}
        \item $E_n$: by Lemma~\ref{lem:generator-signs}, $\mathrm{sgn}_2$ is trivial on $x_n,y_n,z_n$, and for $g\in L_{n-1}$ we have $\mathrm{sgn}_2((g,\mathrm{id},\mathrm{id}))=\mathrm{sgn}_1(g)=1$.
        \item $D_n$: $\mathrm{sgn}_1\mathrm{sgn}_2\mathrm{sgn}_3$ is trivial on $x_n,z_n,k_n$, on $y_n$, since $(-1)(+1)(-1)=1$, and on $l_n$, since $(-1)(-1)(+1)=1$. For $g\in L_{n-1}$ we have $\mathrm{sgn}_1\mathrm{sgn}_2\mathrm{sgn}_3((g,\mathrm{id},\mathrm{id}))=\mathrm{sgn}_1(g)\mathrm{sgn}_2(g)=1$.
        \item $E'_n$: $\mathrm{sgn}_3$ is trivial on $x_n,p_n,q_n,r_n,\hat{p}_n,\hat{q}_n$, and for $g\in L_{n-1}$ we have $\mathrm{sgn}_3((g,\mathrm{id},\mathrm{id}))=\mathrm{sgn}_2(g)=1$.
        \item $\widetilde{D}_n$: by Definition~\ref{def:discriminant-groups} we have to check that $\chi(a_i)=\mathrm{sgn}(\pi)$ for every section $a_i$ of a generator. By Lemma~\ref{lem:generator-signs}, $\chi$ is trivial on $x_{n-1}$, on $\tilde{l}_{n-1}$ and on $\widetilde{L}_{n-1}$, and $\chi(\tilde{y}_{n-1})=-1$. So $\chi$ is trivial on all the sections of $x_n$, $\tilde{k}_n$ and $(g,\mathrm{id},\mathrm{id})$, whose top permutations are even, and $\chi$ is $-1$ on all the sections of $\tilde{y}_n$ and $\tilde{l}_n$, whose top permutations are odd.
    \end{itemize}
    Since $G_n$ is a subgroup, this gives $M_n\subseteq G_n$.

    We prove the other inclusion using the sizes. For the first three groups, the element $(\mathrm{id},\mathrm{id},\mathrm{id})(0,1)$ of $G_{n-1}\wr S_3$ has $\mathrm{sgn}_j=-1$ for all $j$ by \eqref{eq:sign-recursion}, so $G_n$ has index $2$ in $G_{n-1}\wr S_3$, and $|G_n|=3|G_{n-1}|^3$. For $\widetilde{D}_n$, $\chi$ is surjective on $\widetilde{D}_{n-1}=\widetilde{M}_{n-1}$ because of $\tilde{y}_{n-1}$. So for each $\pi\in S_3$ there are $(|\widetilde{D}_{n-1}|/2)^3$ choices for the sections, and $|\widetilde{D}_n|=\frac{3}{4}|\widetilde{D}_{n-1}|^3$. On the other hand, in all cases $H_n$ is the set of elements $(a_0,a_1,a_2)$ with $a_0,a_1,a_2\in L_{n-1}$, so that $|M_n|=[M_n:L_n]\,[L_n:H_n]\,|L_{n-1}|^3$. By Theorems~\ref{TH 4.13}, \ref{th 4.22}, \ref{th 4.23}, this gives $|M_n|=2\cdot 12\,(|M_{n-1}|/2)^3=3|M_{n-1}|^3$ for length one. By Theorems~\ref{th 4.39}, \ref{thm:Hn-two}, \ref{thm:three-normal} and Corollaries~\ref{4.42}, \ref{cor:three-K}, it gives $|M_n|=4\cdot 48\,(|M_{n-1}|/4)^3=3|M_{n-1}|^3$ for lengths two and three. By Theorems~\ref{thm:refined-two-normal}, \ref{thm:refined-two-H} and Corollary~\ref{cor:refined-two-K}, it gives $|\widetilde{M}_n|=4\cdot 12\,(|\widetilde{M}_{n-1}|/4)^3=\frac{3}{4}|\widetilde{M}_{n-1}|^3$. Hence $|M_n|=|G_n|$, and $M_n=G_n$.

    For the last statement, if $\sigma\in\widetilde{D}_n$ with $n\geq 3$, then \eqref{eq:sign-recursion} gives
    $$\mathrm{sgn}_1\mathrm{sgn}_2\mathrm{sgn}_3(\sigma)=\mathrm{sgn}(\pi)^3\prod_{i=0}^2\chi(\sigma_i)=\mathrm{sgn}(\pi)^6=1,$$
    so $\widetilde{D}_n\subseteq D_n$ by induction.
\end{proof}
\begin{remark}\label{rem:not-new}
    Theorem~\ref{thm:Markov-discriminant} shows that the group of Section~\ref{3} is the arboreal Galois group $E_n$ of the Belyi map $-2z^3+3z^2$ computed in~\cite{Benedetto2017} (see also~\cite{Bouw2018DynamicalBM}). Hence Section~\ref{3} is a consistency check of the Markov model in a case where the Galois groups are known (see Theorem~\ref{thm:belyi}); the groups of Sections~\ref{4}, \ref{sec:three} and \ref{sec:refined-two} are new.
\end{remark}

The Hausdorff dimension of subgroups of profinite groups was introduced by Abercrombie~\cite{Abercrombie_1994}. We use it in the following form.
\begin{definition}\label{def 3.20}
    Let $G=\varprojlim_n G_n$ be a profinite group, with $G_n$ finite, and let $H=\varprojlim_n H_n$ be a closed subgroup, where $H_n$ is the image of $H$ in $G_n$. The \emph{Hausdorff dimension} of $H$ in $G$ (with respect to the filtration of $G$ by the kernels of the maps $G\to G_n$) is
    $$\mathrm{hdim}_{G}(H)=\liminf_{n\to \infty} \frac{\log|H_n|}{\log|G_n|}.$$
\end{definition}
All the groups of this paper form inverse systems with surjective restriction maps, and for all of them the limit exists, so the $\liminf$ is a limit.
\begin{proposition}\label{prop:hdim}
    We have $|\mathrm{Aut}(\mathcal{T}_n)|=6^{\frac{3^n-1}{2}}$. For the groups of Definition~\ref{def:discriminant-groups}, the orders for $n\geq n_0$, and the Hausdorff dimensions of $G=\varprojlim_n G_n$ in $\mathrm{Aut}(\mathcal{T})$, are the following.
    \begin{center}
    \renewcommand{\arraystretch}{1.3}
    \begin{tabular}{llll}
        \toprule
        $G_n$ & $n_0$ & $|G_n|$ & $\mathrm{hdim}_{\mathrm{Aut}(\mathcal{T})}(G)$\\
        \midrule
        $E_n$ & $1$ & $2^{3^{n-1}}\,3^{\frac{3^n-1}{2}}$ & $1-\frac{1}{3\log_2 6}\approx 0.871$\\
        $D_n$,\ $E'_n$ & $2$ & $2^{4\cdot 3^{n-2}}\,3^{\frac{3^n-1}{2}}$ & $1-\frac{1}{9\log_2 6}\approx 0.957$\\
        $\widetilde{D}_n$ & $2$ & $2^{3^{n-1}+1}\,3^{\frac{3^n-1}{2}}$ & $1-\frac{1}{3\log_2 6}\approx 0.871$\\
        \bottomrule
    \end{tabular}
    \end{center}
    In particular $[D_n:\widetilde{D}_n]=2^{3^{n-2}-1}$ for $n\geq 2$. The groups of the models of Sections~\ref{3}, \ref{4}, \ref{sec:three} and \ref{sec:refined-two} have index at most $12$ in the corresponding group $M_n$, so their limits have the same Hausdorff dimension as $\varprojlim_n M_n$.
\end{proposition}
\begin{proof}
    For $n>n_0$ we have $|G_n|=c\,|G_{n-1}|^3$, where $c=3$ for $E_n$, $D_n$ and $E'_n$, and $c=\frac{3}{4}$ for $\widetilde{D}_n$, as shown in the proof of Theorem~\ref{thm:Markov-discriminant}. The recursion can be written as $\log|G_n|+\frac{1}{2}\log c=3\bigl(\log|G_{n-1}|+\frac{1}{2}\log c\bigr)$, so
    $$\log|G_n|=3^{n-n_0}\bigl(\log|G_{n_0}|+\tfrac{1}{2}\log c\bigr)-\tfrac{1}{2}\log c.$$
    With $|G_1|=6$ and $|G_2|=6^4$, this gives the orders in the table. The same argument with $|\mathrm{Aut}(\mathcal{T}_n)|=6\,|\mathrm{Aut}(\mathcal{T}_{n-1})|^3$ and $|\mathrm{Aut}(\mathcal{T}_1)|=6$ gives the order of $\mathrm{Aut}(\mathcal{T}_n)$. Dividing by $\log|\mathrm{Aut}(\mathcal{T}_n)|=\frac{3^n-1}{2}\log 6$, we get
    $$\lim_{n\to\infty}\frac{\log|G_n|}{\log|\mathrm{Aut}(\mathcal{T}_n)|}=\frac{2}{3^{n_0}\log 6}\Bigl(\log|G_{n_0}|+\frac{1}{2}\log c\Bigr),$$
    and this is the value in the table. Finally, a subgroup of index at most $12$ in $M_n$ for every $n$ has the same Hausdorff dimension, since the logarithms of the orders differ by a bounded amount.
\end{proof}
Thus the refined groups have the same Hausdorff dimension as $E_n$, which is smaller than that of $D_n$. The groups $D_n$ and $E'_n$ have the same order, but they are different: $E'_3$ is the kernel of $\mathrm{sgn}_3$ on $\mathrm{Aut}(\mathcal{T}_3)$, while $D_3$ is the kernel of $\mathrm{sgn}_1\mathrm{sgn}_2\mathrm{sgn}_3$. The Markov models are also different: for length two the action of $f$ on types has order two, and for length three it permutes the three non-trivial types of $S$ cyclically.

\subsection{The Galois groups are contained in the Markov groups}\label{5.2}
Let $F$ be a field of characteristic different from $2$ and $3$, let $f=ax^3+bx^2+cx+d\in F[x]$ with $a\neq 0$, and let $\gamma_1,\gamma_2\in\overline{F}$ be its critical points. For $t\in F$ and $k\geq 1$ we put $P_k(t):=\bigl(f^k(\gamma_1)-t\bigr)\bigl(f^k(\gamma_2)-t\bigr)\in F$, which only depends on the set $P_k$ of Definition~\ref{comb}.

Let $t\in F$ with $f_t^n$ separable. As in the Introduction, the roots of $f_t^0(x)=x,\ f_t,\dots,\ f_t^n$ form a $3$-ary rooted tree $\mathcal{T}_n(f_t)$, in which a root $\alpha$ of $f_t^k$ is joined to the root $f_t(\alpha)$ of $f_t^{k-1}$, and $\mathrm{Gal}(f_t^n/F)$ acts on it. If $\alpha$ is a root of $f_t$, then \eqref{eq:conjugate} gives $f_{t+\alpha}^k(x)=f_t^k(x+\alpha)-\alpha$ for every $k\geq 1$, so the map $x\mapsto x+\alpha$ sends the roots of $f_{t+\alpha}^k$ onto the vertices of depth $k$ below $\alpha$. In this way, the subtree below $\alpha$ is the tree $\mathcal{T}_{n-1}(f_{t+\alpha})$ of the conjugate $f_{t+\alpha}$ over the field $F(\alpha)$. An isomorphism of rooted trees $\lambda\colon\mathcal{T}_n(f_t)\to\mathcal{T}_n$, which we call a \emph{labelling}, induces an embedding $\rho_\lambda\colon \mathrm{Gal}(f_t^n/F)\hookrightarrow \mathrm{Aut}(\mathcal{T}_n)$, and changing $\lambda$ conjugates its image. For $H\subseteq\mathrm{Aut}(\mathcal{T}_n)$ we say that $\mathrm{Gal}(f_t^n/F)$ is \emph{contained in $H$ up to conjugation} if $\rho_\lambda(\mathrm{Gal}(f_t^n/F))\subseteq H$ for some labelling $\lambda$. The choice of labelling matters: for instance, $E_n$ is not normal in $\mathrm{Aut}(\mathcal{T}_n)$ for $n\geq 3$~\cite[Proposition 2.3]{Benedetto2017}.
\begin{lemma}\label{lem:disc-iterates}
    Let $t\in F$ with $f_t^k$ separable. Then $\mathrm{disc}(f_t^k)\equiv -3\,\mathrm{disc}(f_t^{k-1})\,P_k(t)$ modulo $(F^\times)^2$, and hence $\mathrm{disc}(f_t^k)\equiv(-3)^k\prod_{j=1}^kP_j(t)$.
\end{lemma}
\begin{proof}
    Let $g=f_t^{k-1}$, of odd degree $m=3^{k-1}$ and leading coefficient $\ell$, and $h=g\circ f_t=f_t^k$, of degree $N=3m$. The polynomial $f_t$ has leading coefficient $a$ and critical points $\gamma_i-t$, and $h(\gamma_i-t)=f^k(\gamma_i)-t$ (Remark~\ref{rem:conjugates}). We use $h'=g'(f_t)\,f_t'$ and the formulas
    $$\mathrm{disc}(P)=(-1)^{\frac{N(N-1)}{2}}\mathrm{lc}(P)^{-1}\mathrm{Res}(P,P')\quad(\deg P=N),\qquad \mathrm{Res}(P,Q)=\mathrm{lc}(P)^{\deg Q}\prod_{P(\alpha)=0}Q(\alpha).$$ On the one hand, as $\deg h\cdot\deg f_t'$ is even,
    $$\mathrm{Res}(h,f_t')=\mathrm{Res}(f_t',h)=(3a)^{3m}h(\gamma_1-t)\,h(\gamma_2-t)=(3a)^{3m}P_k(t).$$
    On the other hand, the roots of $h$ are the roots of $f_t-\beta$ for the roots $\beta$ of $g$, so $\mathrm{Res}(h,g'(f_t))=(\ell a^m)^{3(m-1)}\bigl(\prod_{g(\beta)=0}g'(\beta)\bigr)^3$ and $\prod_{g(\beta)=0}g'(\beta)=\ell^{1-m}\mathrm{Res}(g,g')=(-1)^{m(m-1)/2}\ell^{2-m}\mathrm{disc}(g)$. Collecting terms, $\mathrm{disc}(h)=(-1)^m\ell^2a^{3m^2-m}3^{3m}\mathrm{disc}(g)^3P_k(t)$, and modulo squares this is $-3\,\mathrm{disc}(g)\,P_k(t)$ because $m$ is odd and $3m^2-m$ is even. The second statement follows by induction from $\mathrm{disc}(f_t^0)=\mathrm{disc}(x)=1$.
\end{proof}
All the inclusions below are proved by induction on $n$. The induction step has two parts: the induction hypothesis controls the sections at the first level, and a discriminant controls the sign condition at the root. The first part does not depend on the group.
\begin{lemma}\label{lem:sections}
    Let $n\geq 2$, let $G_{n-1}$ be a subgroup of $\mathrm{Aut}(\mathcal{T}_{n-1})$ and let $t\in F$ with $f_t^n$ separable. Assume that for every root $\alpha$ of $f_t$ the group $\mathrm{Gal}(f_{t+\alpha}^{n-1}/F(\alpha))$ is contained in $G_{n-1}$ up to conjugation. Then every bijection between the roots of $f_t$ and the vertices of the first level is induced by a labelling $\lambda$ of $\mathcal{T}_n(f_t)$ with $\rho_\lambda(\mathrm{Gal}(f_t^n/F))\subseteq G_{n-1}\wr S_3$.
\end{lemma}
\begin{proof}
    Let $G=\mathrm{Gal}(\Omega/F)$, where $\Omega$ is the splitting field of $f_t^n$, and let $\alpha_0,\alpha_1,\alpha_2$ be the roots of $f_t$. As for tree automorphisms, we let Galois groups act on the right and read products from left to right. The group $G$ permutes the three subtrees below $\alpha_0,\alpha_1,\alpha_2$, and we only have to label them compatibly. For each $G$-orbit on $\{\alpha_0,\alpha_1,\alpha_2\}$ we choose a representative $\alpha$. The subtree below $\alpha$ is the tree $\mathcal{T}_{n-1}(f_{t+\alpha})$, and the stabiliser $G_\alpha=\mathrm{Gal}(\Omega/F(\alpha))$ acts on it through its surjection onto $\mathrm{Gal}(f_{t+\alpha}^{n-1}/F(\alpha))$. By hypothesis, there is a labelling $\lambda_\alpha$ of this subtree under which $G_\alpha$ acts through elements of $G_{n-1}$. For every $\alpha'$ in the orbit of $\alpha$ we fix $g_{\alpha'}\in G$ mapping $\alpha$ to $\alpha'$, with $g_\alpha=\mathrm{id}$. It maps the subtree below $\alpha$ onto the subtree below $\alpha'$, and we label the latter by transporting $\lambda_\alpha$ along $g_{\alpha'}$, that is, by $\lambda_{\alpha'}:=g_{\alpha'}^{-1}\lambda_\alpha$. Together with the given bijection on the first level, this defines a labelling $\lambda$ of $\mathcal{T}_n(f_t)$. Now let $\sigma\in G$ map $\alpha'$ to $\alpha''$, two elements of the orbit of $\alpha$. Then $g_{\alpha'}\sigma g_{\alpha''}^{-1}$ fixes $\alpha$, and by construction the section of $\rho_\lambda(\sigma)$ at the vertex corresponding to $\alpha'$ is $\lambda_\alpha^{-1}(g_{\alpha'}\sigma g_{\alpha''}^{-1})\lambda_\alpha$, the automorphism by which $g_{\alpha'}\sigma g_{\alpha''}^{-1}\in G_\alpha$ acts on the subtree below $\alpha$ in the labelling $\lambda_\alpha$. Hence all first-level sections of $\rho_\lambda(\sigma)$ lie in $G_{n-1}$. 
\end{proof}
For the second part, let $I$ be a finite set of positive integers with largest element $m\geq 2$, put $\mathrm{sgn}_I:=\prod_{k\in I}\mathrm{sgn}_k$, and let $G_n=\mathrm{Aut}(\mathcal{T}_n)$ for $n<m$ and $G_n=(G_{n-1}\wr S_3)\cap\ker(\mathrm{sgn}_I)$ for $n\geq m$. For $I=\{2\}$, $I=\{1,2,3\}$ and $I=\{3\}$ these are the groups $E_n$, $D_n$ and $E'_n$ of Definition~\ref{def:discriminant-groups}.
\begin{proposition}\label{prop:containment}
    Assume that $\prod_{k\in I}\mathrm{disc}(f_{t'}^k)$ is a square in $F'$ for every finite extension $F'$ of $F$ and every $t'\in F'$ with $f_{t'}^m$ separable. Then for every $n\geq 1$ and every $t\in F$ with $f_t^n$ separable, $\mathrm{Gal}(f_t^n/F)$ is contained in $G_n$ up to conjugation, and the labelling can be chosen to induce any prescribed bijection between the roots of $f_t$ and the vertices of the first level.
\end{proposition}
\begin{proof}
    We use induction on $n$, simultaneously for all finite extensions of $F$ and all $t$; the hypothesis passes to the finite extensions of $F$. If $n<m$, then $G_n=\mathrm{Aut}(\mathcal{T}_n)$ and there is nothing to prove. Let $n\geq m$ and $G=\mathrm{Gal}(f_t^n/F)$. For every root $\alpha$ of $f_t$ the polynomial $f_t^{n-1}-\alpha$ divides $f_t^n$, so $f_{t+\alpha}^{n-1}$, which is $f_t^{n-1}-\alpha$ after the translation $x\mapsto x+\alpha$, is separable, and the induction hypothesis, applied over the field $F(\alpha)$ and with $t+\alpha$ instead of $t$, shows that $\mathrm{Gal}(f_{t+\alpha}^{n-1}/F(\alpha))$ is contained in $G_{n-1}$ up to conjugation. By Lemma~\ref{lem:sections}, there is a labelling $\lambda$ inducing the prescribed bijection with $\rho_\lambda(G)\subseteq G_{n-1}\wr S_3$.

    It remains to treat the sign condition. For $k\leq n$ let $\delta_k$ be the product of the differences of the roots of $f_t^k$, taken in a fixed order. As the vertices of level $k$ are these roots, $\sigma(\delta_k)=\mathrm{sgn}_k(\rho_\lambda(\sigma))\delta_k$ for every $\sigma\in G$ and every labelling, and $\delta_k^2$ equals $\mathrm{disc}(f_t^k)$ up to a non-zero square factor in $F$. The polynomial $f_t^m$ is separable, since a multiple root of it would give multiple roots of $f_t^n$. Hence, by hypothesis, $\bigl(\prod_{k\in I}\delta_k\bigr)^2$ is a square in $F$, so $\prod_{k\in I}\delta_k\in F$ and $\mathrm{sgn}_I$ is trivial on $\rho_\lambda(G)$. Therefore $\rho_\lambda(G)\subseteq(G_{n-1}\wr S_3)\cap\ker(\mathrm{sgn}_I)=G_n$.
\end{proof}
\begin{theorem}\label{thm:upper-bound}
    Let $f\in F[x]$ be a PCF cubic polynomial and let $t\in F$ with $f_t^n$ separable.
    \begin{enumerate}[i)]
        \item If $f$ has combined critical orbit of length one, then $\mathrm{Gal}(f_t^n/F)$ is contained in $E_n=M_n^{(1)}$ up to conjugation.
        \item If $f$ has combined critical orbit of length two and $P_3=P_1$, then $\mathrm{Gal}(f_t^n/F)$ is contained in $D_n=M_n^{(2)}$ up to conjugation.
        \item If $f=-2z^3+3z^2+\frac{1}{2}$ and $\sqrt{-3}\in F$, then $\mathrm{Gal}(f_t^n/F)$ is contained in $E'_n=M_n^{(3)}$ up to conjugation.
    \end{enumerate}
    Moreover, the labelling can be chosen to induce any prescribed bijection between the roots of $f_t$ and the vertices of the first level.
\end{theorem}
\begin{proof}
    We apply Proposition~\ref{prop:containment} with $I=\{2\}$, $I=\{1,2,3\}$ and $I=\{3\}$ respectively; the groups of the statement are the Markov groups by Theorem~\ref{thm:Markov-discriminant}. Let $F'$ be a finite extension of $F$ and let $t'\in F'$ with $f_{t'}^m$ separable, where $m=\max I$. We use Lemma~\ref{lem:disc-iterates} over $F'$. In case i) we have $P_2=P_1$, so $\mathrm{disc}(f_{t'}^2)\equiv 9P_1(t')^2$ is a square. In case ii) we have $P_3=P_1$, so $\prod_{k=1}^3\mathrm{disc}(f_{t'}^k)\equiv(-3)^6P_1(t')^3P_2(t')^2P_3(t')$ is a square. In case iii), $P_1(t')P_2(t')P_3(t')$ is a square by Lemma~\ref{lem:three-relation}, so $\mathrm{disc}(f_{t'}^3)\equiv(-3)^3P_1(t')P_2(t')P_3(t')\equiv -3$, which is a square since $\sqrt{-3}\in F'$.
\end{proof}
If $\sqrt{-3}\notin F$, Theorem~\ref{thm:upper-bound} iii) applies to $\mathrm{Gal}(f_t^n/F(\sqrt{-3}))$, which has index at most $2$ in $\mathrm{Gal}(f_t^n/F)$. For $f=-2z^3+3z^2$, part i) of the theorem is~\cite[Lemma 3.1]{Benedetto2017}.

Over a field containing $\sqrt{-3}$ we also recover the groups of the different models.
\begin{corollary}\label{cor:upper-bound-models}
    In the situation of Theorem~\ref{thm:upper-bound} assume that $\sqrt{-3}\in F$, and let $\lambda$ be a labelling as in that theorem. If $P_1(t)$ is a square in $F$, then $\mathrm{sgn}_1$ is trivial on $\rho_\lambda(\mathrm{Gal}(f_t^n/F))$; if $P_2(t)$ is a square, then $\mathrm{sgn}_1\mathrm{sgn}_2$ is trivial; and if $P_1(t)P_2(t)$ is a square, which for length three means that $P_3(t)$ is a square (Lemma~\ref{lem:three-relation}), then $\mathrm{sgn}_2$ is trivial. Consequently, $\rho_\lambda(\mathrm{Gal}(f_t^n/F))$ is contained in the group of the corresponding model:
    \begin{enumerate}[i)]
        \item for length one, in $L_n$ if $P_1(t)$ is a square and in $M_n$ otherwise (Section~\ref{3});
        \item for length two, in $L_n$, $\langle L_n,y_nl_n\rangle$, $\langle L_n,l_n\rangle$, $\langle L_n,y_n\rangle$ or $M_n$, as in Models~2--6 of Section~\ref{4};
        \item for length three, in $L_n$, $\langle L_n,r_n\rangle$, $\langle L_n,p_n\rangle$, $\langle L_n,q_n\rangle$ or $M_n$, as in Models~2--6 of Section~\ref{sec:three}.
    \end{enumerate}
    Moreover, if $f_t$ has a root in $F$ we may assume that $\rho_\lambda(\mathrm{Gal}(f_t^n/F))$ fixes the vertex $2$, and if $f_t$ splits over $F$ then $\rho_\lambda(\mathrm{Gal}(f_t^n/F))$ acts trivially on the first level. This gives the groups $J_n$ and $\langle J_n,y_n\rangle$ of Models~1 and 3 of Section~\ref{3} and the groups $J_n$ of Model~1 of Sections~\ref{4} and \ref{sec:three}.
\end{corollary}
\begin{proof}
    Since $-3$ is a square, Lemma~\ref{lem:disc-iterates} gives the congruences $\mathrm{disc}(f_t)\equiv P_1(t)$, $\mathrm{disc}(f_t^2)\equiv P_1(t)P_2(t)$ and $\mathrm{disc}(f_t)\mathrm{disc}(f_t^2)\equiv P_2(t)$, and the first statement follows as in the proof of Proposition~\ref{prop:containment}. In length one, $L_n=M_n\cap\ker(\mathrm{sgn}_1)$ by \eqref{eq:types-one}. In lengths two and three, the letters of $\tau$ are given by Theorem~\ref{thm:Markov-discriminant}. So the first statement says that the type $c$ of an element of $\rho_\lambda(\mathrm{Gal}(f_t^n/F))$ satisfies $c_1=s$ if $P_1(t)$ is a square, $c_2=s$ if $P_2(t)$ is a square, and $c_1=c_2$ if $P_1(t)P_2(t)$ is a square, which for length three means $c_3=s$. These are the conditions defining the root types of the model, and in Models~2--6 the group of the model is the set of the elements of $M_n$ whose type is a root type, which gives ii) and iii). The last statement follows from the freedom in the choice of the bijection on the first level, together with Theorem~\ref{th 4.23}, the proof of Theorem~\ref{Th 5.19} and Corollaries~\ref{4.42} and \ref{cor:three-K}.
\end{proof}

We now refine Theorem~\ref{thm:upper-bound} for the co-critical polynomials of length two. Proposition~\ref{prop:containment} does not apply here, because the groups $\widetilde{D}_n$ are not cut out by a single sign character: Definition~\ref{def:discriminant-groups} imposes a condition at each of the three vertices of the first level, so for $n\geq 3$ the group $\widetilde{D}_n$ has index $8$ in $\widetilde{D}_{n-1}\wr S_3$, whereas a discriminant over $F$ only sees the sign on a whole level. The extra input is the identity \eqref{eq:cocritical}, which gives a sign condition at every fixed vertex, over the field $F(\alpha)$.
\begin{lemma}\label{lem:fixed-vertex}
    Let $f\in F[x]$ be a cubic with $\{\gamma_1^{\ast},\gamma_2^{\ast}\}=P_2$, and let $t\in F$ and $n\geq 3$ with $f_t^n$ separable. Let $\lambda$ be a labelling, let $\sigma\in\mathrm{Gal}(f_t^n/F)$ fix a root $\alpha$ of $f_t$, and let $i$ be the vertex of the first level corresponding to $\alpha$. Write $\rho_\lambda(\sigma)=(\sigma_0,\sigma_1,\sigma_2)\pi$. Then $\chi(\sigma_i)=\mathrm{sgn}(\pi)$.
\end{lemma}
\begin{proof}
    Let $\delta$ be the product of the differences of the roots of $f_t$, taken in a fixed order. Then $\sigma(\delta)=\mathrm{sgn}(\pi)\delta$, and $\delta^2\equiv-3P_1(t)$ modulo $(F^\times)^2$ by Lemma~\ref{lem:disc-iterates}. Put $t'=t+\alpha$, and let $\varepsilon$ be the product of the differences of the roots of $f_{t'}$ times the product of the differences of the roots of $f_{t'}^2$. These roots, translated by $\alpha$, are the vertices of depth one and two below $\alpha$, and $\sigma$ fixes $\alpha$, so $\sigma(\varepsilon)=\chi(\sigma_i)\varepsilon$. By Lemma~\ref{lem:disc-iterates} over the field $F(\alpha)$,
    $$\varepsilon^2\equiv-3P_1(t')\cdot 9P_1(t')P_2(t')\equiv-3P_2(t')\mod \bigl(F(\alpha)^\times\bigr)^2.$$
    Since $f(t')=f_t(\alpha)+t=t$ and $f'(t')=f_t'(\alpha)\neq 0$ (as $f_t$ is separable), \eqref{eq:cocritical} gives $P_1(t)=\bigl(\frac{f'(t')}{3}\bigr)^2P_2(t')$. Hence $(\delta/\varepsilon)^2$ is a square in $F(\alpha)$, so $\delta/\varepsilon\in F(\alpha)$ is fixed by $\sigma$, and $\mathrm{sgn}(\pi)=\chi(\sigma_i)$.
\end{proof}
Lemma~\ref{lem:fixed-vertex} only concerns the vertices fixed by $\sigma$, whereas the definition of $\widetilde{D}_n$ imposes a condition at all three vertices of the first level. The next lemma shows that, for a group, the condition at the fixed vertices is enough, after a change of labelling.
\begin{lemma}\label{lem:twist}
    Let $H$ be a subgroup of $\{\pm1\}^3\rtimes S_3$, whose elements we write as $(v_0,v_1,v_2)\pi$ and multiply by the multiplication rule \eqref{eq:multiplication}. Assume that $v_i=\mathrm{sgn}(\pi)$ whenever $(v_0,v_1,v_2)\pi\in H$ and $(i)\pi=i$. Then there is $u=(u_0,u_1,u_2)\in\{\pm1\}^3$ such that
    $$u^{-1}Hu\subseteq\Delta:=\bigl\{(\mathrm{sgn}(\pi),\mathrm{sgn}(\pi),\mathrm{sgn}(\pi))\pi\ \big|\ \pi\in S_3\bigr\}.$$
\end{lemma}
The idea of the proof is the following. After twisting by the sign (the map $\theta$ below), the hypothesis says that an element of $H$ which fixes a vertex has trivial section there. Then, if an element of $H$ maps a base point $o$ of an orbit to a vertex $i$, its section at $o$ only depends on $i$ and not on the element, and these sections can be removed by conjugating with a suitable $u$, much as one trivialises a cocycle.
\begin{proof}
    The map $\theta\bigl((v_i)_i\pi\bigr)=(\mathrm{sgn}(\pi)v_i)_i\pi$ is an automorphism of $\{\pm1\}^3\rtimes S_3$, because $\mathrm{sgn}(\pi\rho)\,v_iw_{(i)\pi}=\bigl(\mathrm{sgn}(\pi)v_i\bigr)\bigl(\mathrm{sgn}(\rho)w_{(i)\pi}\bigr)$, and it commutes with conjugation by elements of $\{\pm1\}^3$, since $u^{-1}\bigl((v_i)_i\pi\bigr)u=(u_iv_iu_{(i)\pi})_i\pi$. We have $\theta(\Delta)=\{(1,1,1)\pi\}$, and $H':=\theta(H)$ satisfies $v_i=1$ whenever $(v)\pi\in H'$ and $(i)\pi=i$. For each orbit $O$ of $H'$ on $\{0,1,2\}$ fix $o\in O$, and for $i\in O$ put $u_i:=v_o$, where $(v)\pi\in H'$ is any element with $(o)\pi=i$. This is well defined: if $h=(v)\pi$ and $h'=(v')\pi'$ both map $o$ to $i$, then $hh'^{-1}$ fixes $o$ and its section there is $v_ov'_o$, which is $1$. Now let $h=(v)\pi\in H'$ and $i\in O$, and choose $g=(w)\rho\in H'$ with $(o)\rho=i$. Then $gh$ maps $o$ to $(i)\pi$, and its section at $o$ is $w_ov_i$, so $u_{(i)\pi}=u_iv_i$. Hence $u^{-1}hu=(u_iv_iu_{(i)\pi})_i\pi=(1,1,1)\pi$, and applying $\theta^{-1}$ gives the claim.
\end{proof}
\begin{theorem}\label{thm:refined-upper-bound}
    Let $f\in F[x]$ be a cubic with $\{\gamma_1^{\ast},\gamma_2^{\ast}\}=P_2$, and let $t\in F$ with $f_t^n$ separable. Then $\mathrm{Gal}(f_t^n/F)$ is contained in $\widetilde{M}_n=\widetilde{D}_n$ up to conjugation. Moreover, the labelling can be chosen to induce any prescribed bijection between the roots of $f_t$ and the vertices of the first level.
\end{theorem}
\begin{proof}
    We use induction on $n$, simultaneously for all fields $F$ of characteristic different from $2,3$ and all $t$; for $n\leq 2$ there is nothing to prove. Let $G=\mathrm{Gal}(f_t^n/F)$. Lemma~\ref{lem:sections}, with the induction hypothesis for $f_{t+\alpha}^{n-1}$ over $F(\alpha)$, gives a labelling $\lambda$ inducing the prescribed bijection on the first level such that all first-level sections of $\rho_\lambda(\sigma)$ lie in $\widetilde{D}_{n-1}$ for every $\sigma\in G$. Consider the homomorphism
    $$\Phi\colon\widetilde{D}_{n-1}\wr S_3\to\{\pm1\}^3\rtimes S_3,\qquad (a_0,a_1,a_2)\pi\mapsto\bigl(\chi(a_0),\chi(a_1),\chi(a_2)\bigr)\pi,$$
    so that $\widetilde{D}_n=\Phi^{-1}(\Delta)$. By Lemma~\ref{lem:fixed-vertex}, $H:=\Phi(\rho_\lambda(G))$ satisfies the hypothesis of Lemma~\ref{lem:twist}, so there is $u\in\{\pm1\}^3$ with $u^{-1}Hu\subseteq\Delta$. By Theorem~\ref{thm:Markov-discriminant}, $\chi$ is the second letter of $\tau$, which is surjective on $\widetilde{D}_{n-1}=\widetilde{M}_{n-1}$ since $n-1\geq 2$ (Theorem~\ref{thm:refined-two-normal}), so we can choose $\tilde{u}_i\in\widetilde{D}_{n-1}$ with $\chi(\tilde{u}_i)=u_i$. We replace $\lambda$ by the labelling $\lambda'$ for which $\rho_{\lambda'}=\tilde{u}^{-1}\rho_\lambda\tilde{u}$, where $\tilde{u}=(\tilde{u}_0,\tilde{u}_1,\tilde{u}_2)$. This does not change the bijection on the first level, the sections of $\rho_{\lambda'}(\sigma)$ are $\tilde{u}_i^{-1}a_i\tilde{u}_{(i)\pi}\in\widetilde{D}_{n-1}$, and $\Phi(\rho_{\lambda'}(G))=u^{-1}Hu\subseteq\Delta$. Hence $\rho_{\lambda'}(G)\subseteq\widetilde{D}_n$.
\end{proof}
\begin{corollary}\label{cor:refined-strict}
    In the situation of Theorem~\ref{thm:refined-upper-bound}, $|\mathrm{Gal}(f_t^n/F)|\leq|\widetilde{M}_n|<|M_n^{(2)}|$ for $n\geq 3$; in particular $\mathrm{Gal}(f_t^n/F)$ is never conjugate to $M_n^{(2)}$. Let $\lambda'$ be a labelling as in Theorem~\ref{thm:refined-upper-bound}. If $\sqrt{-3}\in F$, then $\rho_{\lambda'}(\mathrm{Gal}(f_t^n/F))$ is contained in the group $\widetilde{L}_n$, $\langle\widetilde{L}_n,\tilde{y}_n\tilde{l}_n\rangle$, $\langle\widetilde{L}_n,\tilde{l}_n\rangle$, $\langle\widetilde{L}_n,\tilde{y}_n\rangle$ or $\widetilde{M}_n$ of the corresponding refined model (Models~2--6 of Section~\ref{sec:refined-two}). If moreover $f_t$ splits over $F$ and $P_2(t)$ is a square, then $\rho_{\lambda'}(\mathrm{Gal}(f_t^n/F))$ is contained in the group $\widetilde{J}_n$ of Model~1.
\end{corollary}
\begin{proof}
    The first statement follows from Theorem~\ref{thm:refined-upper-bound} and Proposition~\ref{prop:hdim}. The signs $\mathrm{sgn}_1,\mathrm{sgn}_2$ of $\rho_{\lambda'}(\sigma)$ do not depend on the labelling, so the argument of Corollary~\ref{cor:upper-bound-models}, together with Theorem~\ref{thm:Markov-discriminant}, gives the second statement. If $f_t$ splits, then $\rho_{\lambda'}(\mathrm{Gal}(f_t^n/F))$ acts trivially on the first level; if moreover $P_1(t)$ and $P_2(t)$ are squares, it lies in $\widetilde{L}_n\cap\mathrm{St}_1=\widetilde{J}_n$ by Corollary~\ref{cor:refined-two-K}. Note that $P_1(t)$ is a square when $f_t$ splits, since then $\mathrm{disc}(f_t)\equiv P_1(t)$ is a square.
\end{proof}

\subsection{The Belyi map}\label{5.3}
For the Belyi map the inclusion of Theorem~\ref{thm:upper-bound} i) is an equality.
\begin{theorem}\label{thm:belyi}
    Let $f=-2z^3+3z^2$, let $F$ be a number field and let $t\in F$. Assume that there are primes $\mathfrak{p}$ and $\mathfrak{q}$ of $F$ above $2$ and $3$ such that $v_{\mathfrak{q}}(t)=1$ and either $v_{\mathfrak{p}}(t)=\pm 1$ or $v_{\mathfrak{p}}(1-t)=1$. Then for every $n\geq 1$ the group $\mathrm{Gal}(f_t^n/F)$ is conjugate in $\mathrm{Aut}(\mathcal{T}_n)$ to $M_n^{(1)}$. For example, this holds for $F=\mathbb{Q}(\sqrt{-3})$ and $t=2\sqrt{-3}$; in this case the factorisation of $f_t^n$ modulo the primes of $F$ follows the Markov model with the initial data of Model~4 of Section~\ref{3}.
\end{theorem}
\begin{proof}
    By~\cite[Theorem 1.1]{Benedetto2017}, which concerns the solutions of $f^n(x)=t$, we have $|\mathrm{Gal}(f_t^n/F)|=|E_n|$. Hence the inclusion of Theorem~\ref{thm:upper-bound} i) is an equality, and $E_n=M_n^{(1)}$ by Theorem~\ref{thm:Markov-discriminant}. For $F=\mathbb{Q}(\sqrt{-3})$ the prime $2$ is inert and $3=-(\sqrt{-3})^2$, so $t=2\sqrt{-3}$ satisfies the hypotheses. The last statement follows from the Chebotarev density theorem and Theorems~\ref{th 4.24} and \ref{Th 5.19}.
\end{proof}
\begin{remark}\label{rem:iota}
    Theorem~\ref{thm:belyi} also holds for $g=2z^3-3z^2+1$. Indeed, let $\iota(z)=1-z$ and $B=-2z^3+3z^2$. Then $B\circ\iota=\iota\circ B$ and $g=\iota\circ B$, so $g^k=\iota^k\circ B^k$, and \eqref{eq:conjugate} gives $g_t^k(x)=B_t^k\bigl(\iota^k(x+t)-t\bigr)$. Hence $\alpha\mapsto\iota^k(\alpha+t)-t$ maps the roots of $g_t^k$ to the roots of $B_t^k$. These maps are compatible with the tree structure and defined over $F$, so they give a Galois-equivariant isomorphism between the trees of $g_t$ and of $B_t$. Similarly $B_{1-t}(-x)=-B_t(x)$, so $\alpha\mapsto-\alpha$ is a Galois-equivariant isomorphism between the trees of $B_t$ and of $B_{1-t}$, and in Theorem~\ref{thm:belyi} we may replace $t$ by $1-t$.
\end{remark}

\subsection{The conjecture}\label{5.4}
We now collect the polynomials of this paper. In every case the Galois groups are contained in the groups of the models of $f_t$, and the cycle data of these groups are the frequencies of the models, so one statement covers all the cases.
\begin{definition}\label{def:model-of-f}
Let $K$ be a number field containing $\sqrt{-3}$ and let $f\in K[x]$ be a PCF cubic polynomial with non-colliding critical orbits in one of the following four classes. The polynomials listed are representatives over $\mathbb{Q}$ up to affine conjugation. The indicated processes and groups are those constructed in the cited sections.
\begin{enumerate}[i)]
    \item \emph{Combined critical orbit of length one, not co-critical:}
    \[
    f(z)\in\left\{-2z^3+3z^2,\ -z^3+\tfrac32 z^2+1,\ 2z^3-3z^2+1,\ z^3-\tfrac32 z^2\right\}.
    \]
    The process is that of Definition~\ref{Definition 4.9}; its group is $M_n^{(1)}=E_n$ (Section~\ref{3}).
    \item \emph{Combined critical orbit of length two, not co-critical, with $P_3=P_1$:}
    \[
    f(z)\in\left\{2z^3-3z^2+\tfrac12,\ -\tfrac1{28}z^3-\tfrac34 z+\tfrac72\right\}.
    \]
    The process is that of Definition~\ref{Definition 4.9}; its group is $M_n^{(2)}=D_n$ (Section~\ref{4}).
    \item \emph{Combined critical orbit of length two, co-critical, with $P_3=P_1$:}
    \[
    f(z)\in\left\{-z^3+\tfrac32 z^2-1,\ -\tfrac14 z^3+\tfrac32 z+2\right\}.
    \]
    The process is that of Definition~\ref{def:refined-model}; its group is $\widetilde M_n=\widetilde D_n$ (Section~\ref{sec:refined-two}). Here the co-critical points (Definition~\ref{def:cocritical}) form $P_2$ (Proposition~\ref{prop:cocritical-classification}).
    \item \emph{The combined critical orbit of length three:}
    \[
    f(z)=-2z^3+3z^2+\tfrac12.
    \]
    The process is that of Definition~\ref{def:three-model}; its group is $M_n^{(3)}=E'_n$ (Section~\ref{sec:three}).
\end{enumerate}
For $t\in K$, the model of the conjugate $f_t$ in \eqref{eq:conjugate} is the appropriate enumerated model of Sections~\ref{3}, \ref{4}, \ref{sec:refined-two} or \ref{sec:three}. It consists of initial data and a group $M_{n,t}$ depending on $t$. For $n\geq2$ and $t\in K$ with $f_t^n$ separable, we say that $f_t^n$ \emph{follows the model of $f_t$} if, for every partition $\lambda$ of $3^n$, the primes of $K$ of good reduction where $f_t^n$ has cycle type $\lambda$ have density equal to the frequency of $\lambda$ in the data obtained by applying the Markov process $n-1$ times to the initial data of the model of $f_t$.
\end{definition}
\begin{theorem}\label{thm:follows-model}
    Let $K$ and $f$ be as in Definition~\ref{def:model-of-f}, let $t\in K$ belong to one of the models of $f_t$, and let $n\geq 2$ with $f_t^n$ separable. The following are equivalent:
    \begin{enumerate}[i)]
        \item $\mathrm{Gal}(f_t^n/K)$ is conjugate in $\mathrm{Aut}(\mathcal{T}_n)$ to $M_{n,t}$;
        \item $f_t^n$ follows the model of $f_t$;
        \item $f_t^n$ splits completely modulo a set of primes of $K$ of density $\frac{1}{|M_{n,t}|}$.
    \end{enumerate}
\end{theorem}
\begin{proof}
    Let $G=\mathrm{Gal}(f_t^n/K)$. By Corollaries~\ref{cor:upper-bound-models} and \ref{cor:refined-strict}, $G$ is contained in $M_{n,t}$ up to conjugation, and by Theorems~\ref{Th 5.19}, \ref{Th 6.9}, \ref{thm:three-models} and \ref{thm:refined-two-models}, the frequencies predicted by the model are the cycle-type frequencies of $M_{n,t}$. If i) holds, then ii) follows from the Chebotarev density theorem. Next, ii) implies iii), because the identity is the only element of $M_{n,t}$ that fixes every vertex of level $n$. Finally, by the Chebotarev density theorem, $f_t^n$ splits completely modulo a set of primes of density $\frac{1}{|G|}$, so iii) gives $|G|=|M_{n,t}|$, and together with the inclusion this gives i).
\end{proof}
Thus, once the inclusion is known, the fact that cycle data do not determine a group up to conjugation is no obstacle: the frequency of a single cycle type gives the order of the group.
\begin{Conjecture}\label{conj:main}
    Let $K$ be a number field containing $\sqrt{-3}$ and let $f\in K[x]$ belong to one of the four classes of Definition~\ref{def:model-of-f}. For every fixed $n\geq2$ there exists a thin subset $T_n\subseteq K$ such that, for every $t\in K\setminus T_n$, the polynomial $f_t^n$ follows the model of $f_t$. Equivalently, the Galois group $\mathrm{Gal}(f_t^n/K)$ is conjugate in $\mathrm{Aut}(\mathcal{T}_n)$ to the generic Markov group $M_n(f)$ of the appropriate class, namely $E_n$, $D_n$, $\widetilde D_n$ or $E'_n$.
\end{Conjecture}
Thin subsets are meant in the sense of~\cite[Chapter 9]{SerreMW}. No single thin subset independent of $n$ is asserted. Throughout the paper, \emph{almost all $t$} means all $t\in K$ outside a thin subset of $K$. The values of $t$ outside Model~4 (length one), respectively Model~6 (the other three cases), form a thin set (Subsection~\ref{app:exceptions}), and for $t$ in these models the group $M_{n,t}$ of Theorem~\ref{thm:follows-model} is the group of $f$ listed in Definition~\ref{def:model-of-f}. The co-critical polynomials of combined critical orbit length two need their own model: their Galois groups are contained in the smaller group $\widetilde{M}_n$ (Corollary~\ref{cor:refined-strict}), so they do not follow the model of Definition~\ref{Definition 4.9}.

The computational evidence is summarised in Appendix~\ref{app:computations}. We also tested values of $t$ in the other models with $f_t$ irreducible. For $K=\mathbb{Q}(\sqrt{-3})$, Lemma~\ref{lem:certificate} proves $\mathrm{Gal}(f_t^n/K)=M_{n,t}$ for $n\leq 3$ for $66$ of the $68$ pairs $(f,t)$ of the first two classes of Definition~\ref{def:model-of-f}, for $19$ of the $21$ of the co-critical class, and for $8$ of the $9$ of the length-three class. The frequencies agree with the model of $f_t$ up to $n=5$, except for the values of $t$ of Subsection~\ref{app:exceptions}.

\appendix
\section{Computational evidence}\label{app:computations}
In this appendix we summarise our computations for individual values of $t$. Their main outcome is that the inclusions of Theorems~\ref{thm:upper-bound} and \ref{thm:refined-upper-bound} are very often equalities. For $93$ of the $98$ pairs $(f,t)$ that we tested, the Galois group $\mathrm{Gal}(f_t^n/K)$ is equal to the group predicted for $f_t$ at every level $n$ at which the group can be computed, that is, for $n\leq 3$ (Table~\ref{tab:summary}). These equalities are proved and not only estimated: the proof is the computer-assisted argument of Lemma~\ref{lem:certificate}, which combines the inclusions proved in Section~\ref{5} with Frobenius elements. For larger $n$, where the groups are too large, the factorisation of $f_t^n$ agrees with the model of $f_t$ up to $n=5$. All the polynomials are treated in the same way, so we organise the results by the kind of computation: after the method (Subsection~\ref{app:method}), we give the results on the Galois groups for $n\leq 3$ (Subsection~\ref{app:groups}), on the factorisation for $n\leq 5$ (Subsection~\ref{app:tests}) and on the relations of Section~\ref{sec:cocritical} between the letters of a type (Subsection~\ref{app:relations}), and we discuss the exceptional values of $t$ (Subsection~\ref{app:exceptions}). The results for the individual pairs $(f,t)$ are given in Tables~\ref{tab:groups}--\ref{tab:three-patterns} in Appendix~\ref{app:tables}.

Throughout, $K=\mathbb{Q}(\sqrt{-3})$. The primes of $K$ of degree one lie above the rational primes $p\equiv 1\pmod 3$ and have residue field $\mathbb{F}_p$, and we only use these primes. Unless stated otherwise, for each pair $(f,t)$ we use the first $6000$ primes $p\equiv 1\pmod 3$ with $p>100$ at which $f_t^n$ has good reduction; they are all smaller than $1.3\cdot 10^5$. The computations were done in Sage, with GAP for the group computations and PARI for an independent check; the scripts and their output are available from the author. The scripts factor the translate $f_t^n(x-t)$ of $f_t^n$, which is given by \eqref{eq:conjugate}; the translation does not change the degrees of the factors, the trees or the types.

For each pair $(f,t)$ we first determine the model of $f_t$ over $K$. Since $f$ and $t$ have rational coefficients, $f_t$ is irreducible over $K$ if and only if it is irreducible over $\mathbb{Q}$, and a rational number $r$ is a square in $K$ if and only if $r$ or $-3r$ is a square in $\mathbb{Q}$. All the comparisons below use the group and the initial data of this model, not the largest group $M_n$. We only use values of $t$ with $f_t$ irreducible over $K$, so Models~1 and 3 of Section~\ref{3} and Model~1 of Sections~\ref{4}, \ref{sec:three} and \ref{sec:refined-two} are not tested.

We omit the polynomial $2z^3-3z^2+1$, since each of its conjugates has the same tree as the conjugate of $-2z^3+3z^2$ with the same $t$ (Remark~\ref{rem:iota}). For the same reason, we only use one of $t$ and $1-t$ for $-2z^3+3z^2$. For each of the other polynomials in Definition~\ref{def:model-of-f} and each model with $f_t$ irreducible, we used at least one value of $t$ of small height and, in most cases, values of larger height.

\begin{table}[htbp]
\centering
\small
\renewcommand{\arraystretch}{1.15}
\caption{Summary of the computations. For each family: the group that contains the group $M_{n,t}$ of the model of $f_t$; the number of pairs $(f,t)$ for which we computed with the groups, and the number of them for which $\mathrm{Gal}(f_t^n/K)=M_{n,t}$ is proved for $n\leq 3$ (Lemma~\ref{lem:certificate}); the number of pairs, not counting the exceptional values of $t$, to which we applied the frequency and transition tests of Subsection~\ref{app:method}, and the number of these tests that are passed.}
\label{tab:summary}
\begin{tabular}{@{}lcccccc@{}}
\toprule
 & & \multicolumn{2}{c}{Galois groups, $n\leq 3$} & \multicolumn{3}{c}{factorisation, $n\leq 5$}\\
\cmidrule(lr){3-4}\cmidrule(l){5-7}
polynomials & group & pairs & equality proved & pairs & frequency tests & transition tests\\
\midrule
length one, not co-critical & $E_n$ & $24$ & $24$ & $17$ & $84/85$ & $162/162$\\
length two, not co-critical & $D_n$ & $44$ & $42$ & $20$ & $100/100$ & $369/370$\\
length three & $E'_n$ & $9$ & $8$ & $8$ & $40/40$ & $149/149$\\
length two, co-critical & $\widetilde{D}_n$ & $21$ & $19$ & $19$ & $95/95$ & $333/333$\\
\midrule
total & & $98$ & $93$ & $64$ & $319/320$ & $1013/1014$\\
\bottomrule
\end{tabular}
\end{table}

\subsection{Method}\label{app:method}
\paragraph{Frobenius elements and factorisation trees.} Let $p\equiv 1\pmod 3$ be a prime which does not divide the denominators of the coefficients of $f$ and of $t$, nor the numerator of the leading coefficient of $f$, and such that $f_t^n$ is separable modulo $p$. Let $\mathfrak{p}$ be a prime of $K$ above $p$. Then $\mathfrak{p}$ is unramified in the splitting field $\Omega_n$ of $f_t^n$ over $K$, and we write $\mathrm{Frob}_{\mathfrak{p}}\in\mathrm{Gal}(\Omega_n/K)$ for a Frobenius element. The \emph{factorisation tree} of $f_t^n$ modulo $p$ is the rooted tree whose vertices at depth $k\leq n$ are the irreducible factors of $f_t^k$ over $\mathbb{F}_p$, where a factor $h$ of $f_t^{k-1}$ is joined to the irreducible factors of $h\circ f_t$ (Lemma~\ref{lem:composition}), and each vertex is labelled by its degree. For $\sigma\in\mathrm{Aut}(\mathcal{T}_n)$ we write $\Theta(\sigma)$ for the \emph{orbit tree} of $\sigma$: its vertices at depth $k$ are the orbits of $\sigma$ on the vertices of level $k$ of $\mathcal{T}_n$, labelled by their lengths, and an orbit is joined to the orbits below it. The orbit tree is invariant under conjugation in $\mathrm{Aut}(\mathcal{T}_n)$. For every labelling $\lambda$, the orbits of $\rho_\lambda(\mathrm{Frob}_{\mathfrak{p}})$ on level $k$ correspond to the irreducible factors of $f_t^k$ modulo $p$, with lengths equal to the degrees, and this correspondence respects the tree structure. Hence $\Theta(\rho_\lambda(\mathrm{Frob}_{\mathfrak{p}}))$ is the factorisation tree of $f_t^n$ modulo $p$, and in particular it does not depend on $\lambda$.

\paragraph{A certificate for equality.} The following lemma turns the inclusions of Section~\ref{5} into a test for equality.
\begin{lemma}\label{lem:certificate}
    Let $H\subseteq\mathrm{Aut}(\mathcal{T}_n)$ be a subgroup with $\rho_\lambda(\mathrm{Gal}(\Omega_n/K))\subseteq H$ for some labelling $\lambda$, and let $S$ be a set of factorisation trees of $f_t^n$ modulo primes as above. Assume that for every maximal subgroup $U$ of $H$ there is a tree in $S$ which is not the orbit tree of any element of $U$. Then $\rho_\lambda(\mathrm{Gal}(\Omega_n/K))=H$.
\end{lemma}
\begin{proof}
    Otherwise $\rho_\lambda(\mathrm{Gal}(\Omega_n/K))$ is contained in a maximal subgroup $U$ of $H$. Every tree in $S$ is the factorisation tree of $f_t^n$ modulo some prime $p$, hence it is $\Theta(\rho_\lambda(\mathrm{Frob}_{\mathfrak{p}}))$ with $\rho_\lambda(\mathrm{Frob}_{\mathfrak{p}})\in U$, a contradiction.
\end{proof}
Since conjugate subgroups of $H$ contain the same orbit trees and $\Theta$ is constant on conjugacy classes, it suffices to compute $\Theta$ on representatives of the conjugacy classes of one maximal subgroup in each conjugacy class of maximal subgroups; this is done in GAP. We apply Lemma~\ref{lem:certificate} to the group of the model of $f_t$, which contains the Galois group by Corollary~\ref{cor:upper-bound-models}, and by Corollary~\ref{cor:refined-strict} for the co-critical polynomials. When the hypothesis of the lemma holds, the equality is proved; no statistics is involved. This is feasible for $n\leq 3$, where the groups have order at most $|D_3|=2^{12}3^{13}\approx 6.5\cdot 10^9$, but not for $n=4$, where $|D_4|=2^{36}3^{40}\approx 8\cdot10^{29}$.

If the hypothesis of the lemma fails, some maximal subgroup contains all the observed trees. We then replace $H$ by the maximal subgroup whose distribution of orbit trees is closest to the observed frequencies, and repeat until every maximal subgroup misses an observed tree. By Lemma~\ref{lem:certificate}, the last group of this chain is the Galois group if it contains it, but this is not guaranteed; hence its index in the group of the model is only an \emph{estimate} of the index of the Galois group.

For $n=2$ we also computed $|\mathrm{Gal}(f_t^2/\mathbb{Q})|$ with PARI. Since $|\mathrm{Gal}(f_t^2/\mathbb{Q})|\leq 2|\mathrm{Gal}(f_t^2/K)|$, with equality if and only if $\sqrt{-3}$ lies in the splitting field over $\mathbb{Q}$, this gives an independent proof of $\mathrm{Gal}(f_t^2/K)=M_{2,t}$ whenever $|\mathrm{Gal}(f_t^2/\mathbb{Q})|=2|M_{2,t}|$, or $|\mathrm{Gal}(f_t^2/\mathbb{Q})|=|M_{2,t}|$ and some prime $p\equiv 2\pmod 3$ splits completely in the splitting field of $f_t^2$.

\paragraph{Statistical tests.} To go beyond $n=3$, we compare the factorisation of $f_t^n$ with the model of $f_t$ for $n\leq 5$. By Theorem~\ref{thm:follows-model}, this tests the equality. We use two tests, both with significance level $10^{-3}$.
\begin{enumerate}[(a)]
    \item \emph{Frequencies.} For $n\leq 5$, the distributions of the number of irreducible factors and of the number of roots of $f_t^n$ modulo $p$ are compared with those predicted by the model of $f_t$, with a $\chi^2$ test (cells with expected count smaller than $5$ are pooled).
    \item \emph{Transitions.} For every irreducible factor $h$ of $f_t^k$ modulo $p$, $0\leq k\leq 4$, we compute its type with respect to $f_t$ (Definition~\ref{def types}) and the degrees and types of the irreducible factors of $h\circ f_t$. For every level $k$ and every type of the parent, the observed transitions are compared with Definition~\ref{Definition 4.9}, respectively Definition~\ref{def:refined-model} or \ref{def:three-model}, with a $\chi^2$ test. The test also fails if a transition which is expected at least $15$ times never occurs.
\end{enumerate}
We also compute the total variation distance $d_{TV}$ between the observed distribution of factorisation trees and the distribution of orbit trees of a group, and we compare it with the distance expected from sampling alone, namely the mean and the standard deviation of $d_{TV}$ for simulated samples of the same size drawn from the group.

\subsection{The Galois groups for \texorpdfstring{$n\leq 3$}{n <= 3}}\label{app:groups}
Table~\ref{tab:summary} gives the result for each family, and Tables~\ref{tab:groups}, \ref{tab:cocrit} and \ref{tab:three} give the groups and the result for each pair. The orders of the groups $M_{n,t}$ computed in GAP agree with Proposition~\ref{prop:hdim} and the indices of the subgroups of the models. For $93$ of the $98$ pairs, Lemma~\ref{lem:certificate} proves $\mathrm{Gal}(f_t^n/K)=M_{n,t}$ for $n=1,2,3$. For instance, for $t=-60$ the Galois group of $f_t^3$ over $K$ is the group $E_3$ of order $2^{9}3^{13}$ for $f=-2z^3+3z^2$, the group $D_3$ of order $2^{12}3^{13}$ for $f=2z^3-3z^2+\frac{1}{2}$, the group $E'_3$ of the same order for $f=-2z^3+3z^2+\frac{1}{2}$, and the group $\widetilde{D}_3$ of order $2^{10}3^{13}$ for $f=-z^3+\frac{3}{2}z^2-1$. For $46$ of the values of $t$ of large height, the PARI computation also gives an independent proof for $n=2$.

The five exceptions are $t=-\frac{1}{4}$ and $t=\frac{25}{48}$ for $f=2z^3-3z^2+\frac{1}{2}$, $t=\frac{3}{4}$ for $f=-2z^3+3z^2+\frac{1}{2}$, $t=4$ for $f=-\frac{1}{4}z^3+\frac{3}{2}z+2$ and $t=-\frac{217}{96}$ for $f=-z^3+\frac{3}{2}z^2-1$. For the first three, equality is proved for $n=1,2$, while for $n=3$ the data show that the Galois group is smaller. For the last two, this already happens at $n=2$ (Subsection~\ref{app:exceptions}).

\subsection{The factorisation for \texorpdfstring{$n\leq 5$}{n <= 5}}\label{app:tests}
For $n\leq 5$ we applied the tests (a) and (b) to $64$ pairs that are not exceptional (Table~\ref{tab:summary}, and Tables~\ref{tab:highn}, \ref{tab:cocrit-tests} and \ref{tab:three} for the individual pairs). Based on more than $8\cdot 10^6$ irreducible factors, $319$ of the $320$ frequency tests and $1013$ of the $1014$ transition tests are passed. The $p$-values are well calibrated: $4.4\%$ of the $640$ $p$-values of the frequency tests and $4.3\%$ of the $1014$ $p$-values of the transition tests are smaller than $0.05$. The two failed tests have $p$-values $8\cdot10^{-4}$ and $7\cdot10^{-4}$, which is what one expects among about $1300$ tests at level $10^{-3}$. For the first of them, $f=z^3-\frac{3}{2}z^2$, $t=3635$ and $n=3$, we know that $\mathrm{Gal}(f_t^3/K)=M_{3,t}$, so by the Chebotarev density theorem the deviation is a sampling fluctuation. Table~\ref{tab:observed} compares observed values with the model for one value of $t$ of each family.

\begin{table}[!tb]
\centering
\small
\renewcommand{\arraystretch}{1.08}
\caption{Observed values against the model of $f_t$, for $t=-60$ and one polynomial of each family, $6000$ primes. This value of $t$ lies in the largest model of each polynomial. Columns: proportion of primes at which $f_t^n$ is irreducible, mean number of irreducible factors of $f_t^n$, and proportion of primes at which $f_t^n$ has a root. The standard error of an observed proportion is at most $0.007$.}
\label{tab:observed}
\begin{tabular}{@{}lccccccc@{}}
\toprule
 & & \multicolumn{2}{c}{irreducible} & \multicolumn{2}{c}{number of factors} & \multicolumn{2}{c}{root}\\
\cmidrule(lr){3-4}\cmidrule(lr){5-6}\cmidrule(l){7-8}
$f$ & $n$ & observed & model & observed & model & observed & model\\
\midrule
$-2z^3+3z^2$ & $1$ & $0.331$ & $0.333$ & $1.84$ & $1.83$ & $0.669$ & $0.667$\\
 & $2$ & $0.223$ & $0.222$ & $3.30$ & $3.31$ & $0.491$ & $0.488$\\
 & $3$ & $0.146$ & $0.148$ & $5.93$ & $5.94$ & $0.392$ & $0.387$\\
 & $4$ & $0.098$ & $0.099$ & $10.65$ & $10.67$ & $0.330$ & $0.322$\\
 & $5$ & $0.065$ & $0.066$ & $19.12$ & $19.15$ & $0.285$ & $0.275$\\
\midrule
$2z^3-3z^2+\frac{1}{2}$ & $1$ & $0.339$ & $0.333$ & $1.82$ & $1.83$ & $0.661$ & $0.667$\\
 & $2$ & $0.113$ & $0.111$ & $3.34$ & $3.36$ & $0.484$ & $0.494$\\
 & $3$ & $0.072$ & $0.074$ & $6.11$ & $6.14$ & $0.385$ & $0.392$\\
 & $4$ & $0.049$ & $0.049$ & $11.14$ & $11.22$ & $0.318$ & $0.325$\\
 & $5$ & $0.034$ & $0.033$ & $20.38$ & $20.50$ & $0.273$ & $0.278$\\
\midrule
$-2z^3+3z^2+\frac{1}{2}$ & $1$ & $0.332$ & $0.333$ & $1.84$ & $1.83$ & $0.668$ & $0.667$\\
 & $2$ & $0.114$ & $0.111$ & $3.36$ & $3.36$ & $0.498$ & $0.494$\\
 & $3$ & $0.076$ & $0.074$ & $6.13$ & $6.14$ & $0.403$ & $0.392$\\
 & $4$ & $0.052$ & $0.049$ & $11.18$ & $11.22$ & $0.333$ & $0.325$\\
 & $5$ & $0.036$ & $0.033$ & $20.38$ & $20.50$ & $0.285$ & $0.278$\\
\midrule
$-z^3+\frac{3}{2}z^2-1$ & $1$ & $0.339$ & $0.333$ & $1.82$ & $1.83$ & $0.661$ & $0.667$\\
 & $2$ & $0.112$ & $0.111$ & $3.37$ & $3.36$ & $0.496$ & $0.494$\\
 & $3$ & $0.075$ & $0.074$ & $6.10$ & $6.12$ & $0.395$ & $0.394$\\
 & $4$ & $0.049$ & $0.049$ & $11.08$ & $11.11$ & $0.330$ & $0.327$\\
 & $5$ & $0.030$ & $0.033$ & $20.11$ & $20.17$ & $0.284$ & $0.279$\\
\bottomrule
\end{tabular}
\end{table}

For $f=-2z^3+3z^2+\frac{1}{2}$ we also compared the frequencies of the factorisation patterns of $f_t^n$ for $n\leq 4$ with those predicted by the model of $f_t$ (Table~\ref{tab:three-patterns}). Except for $t=\frac{3}{4}$, the distance $d_{TV}$ between them exceeds the mean distance expected from sampling alone by at most $1.4$ standard deviations.

\subsection{Relations between the letters of a type}\label{app:relations}
The relations of Section~\ref{sec:cocritical} hold in all the data. For the co-critical polynomials of length two, there is no violation of Proposition~\ref{prop:cocritical-types} in $5\,076\,629$ checks, that is, pairs formed by an irreducible factor of $f_t^k$ modulo $p$ with $k\leq 4$ and one of its children (Table~\ref{tab:cocrit-tests}). For the polynomial of length three, in accordance with Lemma~\ref{lem:three-relation}, among the $1\,209\,763$ irreducible factors of $f_t^k$ modulo $p$ with $k\leq 4$ and their children, there is no factor whose type lies outside $S$.

For the co-critical polynomials the refinement of the model is necessary. In this subsection and in Tables~\ref{tab:cocrit}--\ref{tab:evidence}, we write $\widetilde{M}_{n,t}$ for the group of the refined model of $f_t$ and $M_{n,t}$ for the group of the corresponding model of Section~\ref{4}, which contains it. For $n=3$ the index of $\widetilde{M}_{3,t}$ in $M_{3,t}$ is $4$, as predicted by Proposition~\ref{prop:hdim} and Corollary~\ref{cor:refined-strict}. For the $19$ values of $t$ that are not exceptional, at $n=3$ the observed factorisation trees are at the distance from $\widetilde{M}_{3,t}$ expected from sampling, and far from $M_{3,t}$ (Table~\ref{tab:cocrit}). For the model of Definition~\ref{Definition 4.9} every transition test fails, since the transitions excluded by Proposition~\ref{prop:cocritical-types} are expected many times and never occur, and the coarser frequency tests detect the difference for some $n\leq 5$ in $16$ of the $19$ cases (Table~\ref{tab:cocrit-tests}).
\begin{remark}\label{rem:refined-evidence}
    The cycle data of $M_n^{(2)}$ and $\widetilde{M}_n$ first differ at level $3$. For $f=-z^3+\frac{3}{2}z^2-1$ and $t=\frac{2}{7}$, for which $f_t$ is irreducible and none of $P_1(t),P_2(t),P_1(t)P_2(t)$ is a square in $K$, we recorded the degrees of the irreducible factors of $f_t$, $f_t^2$ and $f_t^3$ modulo the $16\,891$ primes $p\equiv1\pmod 3$ below $4\cdot 10^5$ of good reduction, which give primes of degree one of $K$. The total variation distance between the frequencies of the resulting factorisation patterns and the frequencies predicted by $\widetilde{M}_3$ is $0.040$, which is the sampling error expected for $16\,891$ samples from $\widetilde{M}_3$ (namely $0.039\pm0.002$), whereas the distance to the frequencies predicted by $M_3^{(2)}$ is $0.308$. For $t=5$ we obtain $0.038$ and $0.305$, and for $f=-\frac{1}{4}z^3+\frac{3}{2}z+2$ and $t=\frac{1}{7}$ we obtain $0.040$ and $0.308$. In these three cases no prime shows a factorisation pattern which is impossible in $\widetilde{M}_3$. For the control polynomial $2z^3-3z^2+\frac{1}{2}$ and $t=5$ the frequencies agree with $M_3^{(2)}$ (distance $0.049$, expected sampling error $0.050$), and $4880$ of the $16\,892$ primes show factorisation patterns that are impossible in $\widetilde{M}_3$ (Table~\ref{tab:evidence}).
\end{remark}

\begin{table}[htbp]
\centering
\small
\renewcommand{\arraystretch}{1.15}
\caption{Exceptional values of $t$, $6000$ primes for each pair: values of $n\leq5$ for which the frequency test fails, root types with missing transitions at level $0$, and estimated index of the Galois group in the group of the model at $n=3$ (n.c.: not computed). The second block refers to the model of Definition~\ref{def:three-model}, and the two rows of the last block to the refined model.}
\label{tab:exceptions}
\begin{tabular}{@{}lllccc@{}}
\toprule
$f$ & model & $t$ & frequency test fails & missing at level $0$ & index at $n=3$\\
\midrule
$2z^3-3z^2+\frac{1}{2}$ & Model~3 & $-1/4$ & $n=3$ & $ss$, $sn$ & $4$\\
 & Model~3 & $1/4$ & none & $ss$, $sn$ & n.c.\\
 & Model~2 & $25/48$ & n.c. & n.c. & $4$\\
$-\frac{1}{28}z^3-\frac{3}{4}z+\frac{7}{2}$ & Model~3 & $5$ & $n=2,3,4$ & $ss$, $sn$ & n.c.\\
 & Model~4 & $3$ & $n=3$ & $ss$, $ns$ & n.c.\\
 & Model~5 & $7$ & $n=3,4,5$ & $ss$, $nn$ & n.c.\\
 & Model~5 & $-217/65$ & $n=3,4,5$ & $ss$, $nn$ & n.c.\\
\midrule
$-2z^3+3z^2+\frac{1}{2}$ & Model~3 & $3/4$ & none & $sss$, $snn$ & $4$\\
\midrule
$-\frac{1}{4}z^3+\frac{3}{2}z+2$ & Model~5 & $4$ & $n=2,\dots,5$ & $ss$, $nn$ & $36$\\
$-z^3+\frac{3}{2}z^2-1$ & Model~2 & $-217/96$ & $n=2,\dots,5$ & $ss$ & $4$\\
\bottomrule
\end{tabular}
\end{table}

\subsection{Exceptional values of \texorpdfstring{$t$}{t}}\label{app:exceptions}
Table~\ref{tab:exceptions} lists the values of $t$ in our computations for which the factorisation of $f_t^n$ does not follow the model of $f_t$. Five of them are the exceptions of Table~\ref{tab:summary}. The other five, namely $t=\frac{1}{4}$ for $2z^3-3z^2+\frac{1}{2}$ and $t=5,3,7,-\frac{217}{65}$ for $-\frac{1}{28}z^3-\frac{3}{4}z+\frac{7}{2}$, were found by the tests (a) and (b); we did not compute their groups, so they are not among the pairs of Table~\ref{tab:summary}. In all the cases in which we computed the transitions, some transitions from level $0$, that is, from the polynomial $x$ to the irreducible factors of $f_t$, never occur, although the model expects each of them between $186$ and $1510$ times. Except for $t=4$, all the transition tests at the levels $k=1,\dots,4$ are passed ($119$ of $119$, including the $16$ for the polynomial of length three). Thus the deviation comes from an additional relation between the types of the irreducible factors of $f_t$, that is, between the classes of the values $P_j(\alpha+t)$ modulo squares of $K(\alpha)$, where $\alpha$ runs over the roots of $f_t$. For instance, for $f=-\frac{1}{4}z^3+\frac{3}{2}z+2$ and $t=4$ we have
$$(x^2-2)^2-2P_1(x)=-4x\,f_4(x-4),$$
so that $P_1(\alpha+4)\equiv 2\equiv P_2(4)$ modulo squares of $K(\alpha)$ for every root $\alpha$ of $f_4$. Similarly, for $f=-2z^3+3z^2+\frac{1}{2}$ and $t=\frac{3}{4}$ we have
$$(4x^2-8x+1)^2-12P_2(x)=-4(2x-5)\,f_{3/4}\bigl(x-\tfrac{3}{4}\bigr),$$
so that $P_2(\alpha+\frac{3}{4})\equiv 3\equiv P_2(\frac{3}{4})$ modulo squares of $K(\alpha)$ for every root $\alpha$ of $f_{3/4}$. Hence the second letter of the type of every irreducible factor of $f_{3/4}$ is the second letter of the root type, and at the primes at which $f_{3/4}$ splits, only $4$ of the $16$ triples of types allowed by Definition~\ref{def:three-model} occur. (With $x=-\frac{1}{4}+\frac{u}{2}$ we have $f_{3/4}(x)=-\frac{1}{4}(u^3-3u-1)$, so the splitting field of $f_{3/4}$ is the cyclic cubic field $\mathbb{Q}(\cos\frac{2\pi}{9})$.) The other two values of Model~3 of this polynomial that we tested follow this model. We do not study the exceptional values further.

All the values of $t$ of Table~\ref{tab:exceptions} belong to Models~2--5, but exceptions also occur in Model~6. For instance, for $f=2z^3-3z^2+\frac{1}{2}$ we have
$$\bigl(x^2+\tfrac{x}{2}+1\bigr)^2-6P_1(x)=\bigl(\tfrac{x}{2}+\tfrac{5}{4}\bigr)\bigl(f(x)+\tfrac{3}{2}\bigr),$$
so for $t=-\frac{3}{2}$, which belongs to Model~6, we have $P_1(\alpha+t)\equiv 6$ modulo squares of $K(\alpha)$ for every root $\alpha$ of $f_t$. Hence $f_t^2$ has no irreducible factor of degree $4$ modulo any prime of good reduction, whereas the model predicts one at a quarter of the primes. The values of $t$ that do not belong to Model~4 (length one), respectively Model~6 (length two), form a thin subset of $K$: they are the values for which $f_t$ has a root in $K$, or $P_1(t)$, $P_2(t)$ or $P_1(t)P_2(t)$ is a square in $K$. For the polynomial of length three, the values of $t$ outside Model~6 are those for which $f_t$ has a root in $K$ or one of $P_1(t)$, $P_2(t)$, $P_3(t)$ is a square in $K$, and they also form a thin subset of $K$. Hence the values of Table~\ref{tab:exceptions} do not contradict Conjecture~\ref{conj:main}, which concerns almost all $t$, and neither do exceptions in Model~4 or Model~6 that lie in a thin subset of $K$. The computations of Subsections~\ref{app:groups} and \ref{app:tests} show that equality also holds for most of the values of $t$ of the other models that we tested.

\FloatBarrier
\clearpage
\section{Detailed computational tables}\label{app:tables}
\setcounter{table}{0}
\renewcommand{\thetable}{S\arabic{table}}
\renewcommand{\theHtable}{supplement.\arabic{table}}
\makeatletter
\setlength{\@fptop}{0pt}
\makeatother

This appendix contains the results for the individual pairs $(f,t)$ of the computations summarised in Appendix~\ref{app:computations}. The tables are numbered S1--S7 to retain the numbering of the separate journal supplement; all references point to this same paper. We use the notation of Appendix~\ref{app:computations}: $K=\mathbb{Q}(\sqrt{-3})$, the conjugate $f_t(x)=f(x+t)-t$ of $f$, the certificate of Lemma~\ref{lem:certificate}, and the frequency tests (a) and the transition tests (b) of Subsection~\ref{app:method}, at significance level $10^{-3}$. Unless stated otherwise, for each pair $(f,t)$ we used the first $6000$ primes $p\equiv 1\pmod 3$ with $p>100$ at which $f_t^n$ has good reduction.

\subsection{Polynomials that are not co-critical}\label{app:tables-not-co-critical}
For each polynomial $f$ of Sections~\ref{3} and \ref{4} that is not co-critical, and each model with $f_t$ irreducible, we used one value of $t$ of small height and, where we found them, several values of $t$ of height between $100$ and $5000$, chosen at random among the values of that model that we found. For Models~5 and 2 of $-\frac{1}{28}z^3-\frac{3}{4}z+\frac{7}{2}$ these values lie on a curve of genus one, and we found only a few; the value $-\frac{21559}{107448}$ of Model~2 has larger height. Table~\ref{tab:groups} contains the groups and the result of the certificate for $n\leq 3$ (Subsection~\ref{app:groups}).

For $n\leq 5$ we applied the tests (a) and (b) to the values of $t$ of small height of Table~\ref{tab:groups}, a few more values of small height, and one value of $t$ of large height for each model (Table~\ref{tab:highn}).

\begin{table}[!htbp]
\centering
\small
\renewcommand{\arraystretch}{1.15}
\caption{Polynomials of Sections~\ref{3} and \ref{4} that are not co-critical: the models, the groups $M_{n,t}$ with their orders for $n=1,2,3$, the values of $t$ tested, and the number of them for which $\mathrm{Gal}(f_t^n/K)=M_{n,t}$ is proved for $n=1,2,3$. $6000$ primes for each pair. $^{\ast}$: equality proved for $n=1,2$; for $n=3$ the Galois group is smaller (Table~\ref{tab:exceptions}).}
\label{tab:groups}
\begin{tabular}{@{}lll>{\raggedright\arraybackslash}p{6.2cm}c@{}}
\toprule
model & $M_{n,t}$ & $|M_{n,t}|$, $n=1,2,3$ & values of $t$ & proved\\
\midrule
\multicolumn{5}{@{}l}{$f=-2z^3+3z^2$ (Belyi map)}\\
Model~4 & $M_n=E_n$ & $6,\ 2^33^4,\ 2^93^{13}$ & $-60$, $-3869/4$, $1411/9$, $-3284/29$ & $4/4$\\
Model~2 & $L_n$ & $3,\ 2^23^4,\ 2^83^{13}$ & $-1/3$, $961/285$, $-144/2257$, $-324/517$ & $4/4$\\
\midrule
\multicolumn{5}{@{}l}{$f=-z^3+\frac{3}{2}z^2+1$}\\
Model~4 & $M_n=E_n$ & $6,\ 2^33^4,\ 2^93^{13}$ & $-60$, $-3172/43$, $-4363/36$, $-1957/48$ & $4/4$\\
Model~2 & $L_n$ & $3,\ 2^23^4,\ 2^83^{13}$ & $5/3$, $265/96$, $167/95$, $817/825$ & $4/4$\\
\midrule
\multicolumn{5}{@{}l}{$f=z^3-\frac{3}{2}z^2$}\\
Model~4 & $M_n=E_n$ & $6,\ 2^33^4,\ 2^93^{13}$ & $-60$, $3635$, $-3752/19$, $-107/12$ & $4/4$\\
Model~2 & $L_n$ & $3,\ 2^23^4,\ 2^83^{13}$ & $-2/3$, $-2888/3175$, $-1849/1650$, $-648/671$ & $4/4$\\
\midrule
\multicolumn{5}{@{}l}{$f=2z^3-3z^2+\frac{1}{2}$}\\
Model~6 & $M_n=D_n$ & $6,\ 2^43^4,\ 2^{12}3^{13}$ & $-60$, $-1700/27$, $-911/19$, $-2485/6$ & $4/4$\\
Model~5 & $\langle L_n,y_n\rangle$ & $6,\ 2^33^4,\ 2^{11}3^{13}$ & $2/3$, $162/35$, $200/39$, $128/31$ & $4/4$\\
Model~4 & $\langle L_n,l_n\rangle$ & $6,\ 2^33^4,\ 2^{11}3^{13}$ & $-2/3$, $-121/192$, $-1521/1792$, $-1682/2139$ & $4/4$\\
Model~3 & $\langle L_n,y_nl_n\rangle$ & $3,\ 2^33^4,\ 2^{11}3^{13}$ & $-1/4^{\ast}$, $521/558$, $1609/480$, $-2593/4030$, $-2957/2550$, $925/1066$, $1105/1152$ & $6/7$\\
Model~2 & $L_n$ & $3,\ 2^23^4,\ 2^{10}3^{13}$ & $25/48^{\ast}$, $625/672$, $-1369/1680$, $1681/1440$ & $3/4$\\
\midrule
\multicolumn{5}{@{}l}{$f=-\frac{1}{28}z^3-\frac{3}{4}z+\frac{7}{2}$}\\
Model~6 & $M_n=D_n$ & $6,\ 2^43^4,\ 2^{12}3^{13}$ & $-60$, $4299/4$, $-1459/54$, $-1203/56$ & $4/4$\\
Model~5 & $\langle L_n,y_n\rangle$ & $6,\ 2^33^4,\ 2^{11}3^{13}$ & $1$, $7/3$, $-119/141$, $271/79$ & $4/4$\\
Model~4 & $\langle L_n,l_n\rangle$ & $6,\ 2^33^4,\ 2^{11}3^{13}$ & $-3$, $-2659/1188$, $-177/400$, $-3219/527$ & $4/4$\\
Model~3 & $\langle L_n,y_nl_n\rangle$ & $3,\ 2^33^4,\ 2^{11}3^{13}$ & $2$, $2362/1131$, $793/360$, $1991/1128$, $-55/544$, $1409/360$, $314/235$ & $7/7$\\
Model~2 & $L_n$ & $3,\ 2^23^4,\ 2^{10}3^{13}$ & $103/24$, $-21559/107448$ & $2/2$\\
\midrule
total & & & $68$ values of $t$ & $66/68$\\
\bottomrule
\end{tabular}
\end{table}

\begin{table}[!htbp]
\centering
\small
\renewcommand{\arraystretch}{1.15}
\caption{Tests (a) and (b) for $n\leq 5$, polynomials that are not co-critical, $6000$ primes for each pair: number of frequency tests ($n=1,\dots,5$, the smaller of the two $p$-values) and of transition tests (levels $k=0,\dots,4$ and types of the parent) that are passed. The exceptional values of $t$ are in Table~\ref{tab:exceptions}.}
\label{tab:highn}
\begin{tabular}{@{}ll>{\raggedright\arraybackslash}p{5.4cm}cc@{}}
\toprule
$f$ & model & values of $t$ & frequencies & transitions\\
\midrule
$-2z^3+3z^2$ & Model~4 & $-60$, $-59$, $-3869/4$ & $15/15$ & $30/30$\\
 & Model~2 & $-1/3$, $961/285$ & $10/10$ & $18/18$\\
$-z^3+\frac{3}{2}z^2+1$ & Model~4 & $-60$, $-59$, $-3172/43$ & $15/15$ & $30/30$\\
 & Model~2 & $5/3$, $3/5$, $265/96$ & $15/15$ & $27/27$\\
$z^3-\frac{3}{2}z^2$ & Model~4 & $-60$, $-59$, $3635$ & $14/15$ & $30/30$\\
 & Model~2 & $-2/3$, $2/5$, $-1849/1650$ & $15/15$ & $27/27$\\
\midrule
$2z^3-3z^2+\frac{1}{2}$ & Model~6 & $-60$, $-59$, $-1700/27$ & $15/15$ & $60/60$\\
 & Model~5 & $2/3$, $-2/5$, $722/1003$ & $15/15$ & $54/54$\\
 & Model~4 & $-2/3$, $2/5$, $-121/192$ & $15/15$ & $54/54$\\
 & Model~3 & $521/558$ & $5/5$ & $17/18$\\
 & Model~2 & $625/672$ & $5/5$ & $17/17$\\
$-\frac{1}{28}z^3-\frac{3}{4}z+\frac{7}{2}$ & Model~6 & $-60$, $-59$, $4299/4$ & $15/15$ & $60/60$\\
 & Model~5 & $1$ & $5/5$ & $18/18$\\
 & Model~4 & $-3$, $-2659/1188$ & $10/10$ & $36/36$\\
 & Model~3 & $2$, $2362/1131$ & $10/10$ & $36/36$\\
 & Model~2 & $-21559/107448$ & $5/5$ & $17/17$\\
\midrule
total & & $37$ values of $t$ & $184/185$ & $531/532$\\
\bottomrule
\end{tabular}
\end{table}

\FloatBarrier
\subsection{Co-critical polynomials of length two}\label{app:tables-co-critical}
For $-z^3+\frac{3}{2}z^2-1$ and $-\frac{1}{4}z^3+\frac{3}{2}z+2$, and each of Models~2--6 of Section~\ref{sec:refined-two}, we used one value of $t$ of small height and one of height between $200$ and $1000$. The exception is Model~2 of $-\frac{1}{4}z^3+\frac{3}{2}z+2$, whose values of $t$ lie on a curve of genus one: there we used the two values of smallest height, $41/12$ and $4793/276$. Table~\ref{tab:cocrit} contains the groups and the certificate, and Table~\ref{tab:cocrit-tests} the tests (a) and (b) for $n\leq 5$, for the refined model of Definition~\ref{def:refined-model} and for the model of Definition~\ref{Definition 4.9} (Subsection~\ref{app:relations}).

Table~\ref{tab:evidence} contains the computation of Remark~\ref{rem:refined-evidence}, with all the primes $p\equiv1\pmod 3$ below $4\cdot 10^5$ and $n=3$, both for the factorisation trees and for the coarser degree patterns, that is, the degrees of the irreducible factors of $f_t$, $f_t^2$ and $f_t^3$.

\begin{table}[!htbp]
\centering
\small
\renewcommand{\arraystretch}{1.15}
\caption{Co-critical polynomials of length two and the refined models, $6000$ primes for each pair. $=$: $\mathrm{Gal}(f_t^n/K)=\widetilde{M}_{n,t}$ proved by Lemma~\ref{lem:certificate}. $d_{TV}$: distance at $n=3$ between the observed factorisation trees and the orbit trees of $\widetilde{M}_{3,t}$ (in brackets: mean and standard deviation of the distance expected from sampling), and of the group $M_{3,t}$ of the corresponding model of Section~\ref{4}. In each block, the values of $t$ of small height come first. Last two rows: the exceptions; $4^{\ast}$, $36^{\ast}$: estimated index of the Galois group in $\widetilde{M}_{n,t}$.}
\label{tab:cocrit}
\setlength{\tabcolsep}{4pt}
\begin{tabular}{@{}llllcccc@{}}
\toprule
$f$ & $t$ & model & $\widetilde{M}_{n,t}$ & $|\widetilde{M}_{3,t}|$ & $n=1,2,3$ & $d_{TV}(\cdot,\widetilde{M}_{3,t})$ & $d_{TV}(\cdot,M_{3,t})$\\
\midrule
$-z^3+\frac{3}{2}z^2-1$ & $-60$ & Model~6 & $\widetilde{M}_n$ & $2^{10}3^{13}$ & $=\ =\ =$ & $0.071\ (0.070\pm0.004)$ & $0.386$\\
 & $7/2$ & Model~5 & $\langle\widetilde{L}_n,\tilde{y}_n\rangle$ & $2^{9}3^{13}$ & $=\ =\ =$ & $0.049\ (0.050\pm0.004)$ & $0.385$\\
 & $0$ & Model~4 & $\langle\widetilde{L}_n,\tilde{l}_n\rangle$ & $2^{9}3^{13}$ & $=\ =\ =$ & $0.050\ (0.050\pm0.004)$ & $0.384$\\
 & $-1/3$ & Model~3 & $\langle\widetilde{L}_n,\tilde{y}_n\tilde{l}_n\rangle$ & $2^{9}3^{13}$ & $=\ =\ =$ & $0.052\ (0.054\pm0.004)$ & $0.265$\\
 & $-7/6$ & Model~2 & $\widetilde{L}_n$ & $2^{8}3^{13}$ & $=\ =\ =$ & $0.038\ (0.040\pm0.004)$ & $0.258$\\
\addlinespace
 & $-442/9$ & Model~6 & $\widetilde{M}_n$ & $2^{10}3^{13}$ & $=\ =\ =$ & $0.073\ (0.070\pm0.004)$ & $0.387$\\
 & $231/19$ & Model~5 & $\langle\widetilde{L}_n,\tilde{y}_n\rangle$ & $2^{9}3^{13}$ & $=\ =\ =$ & $0.050\ (0.050\pm0.004)$ & $0.379$\\
 & $-381/40$ & Model~4 & $\langle\widetilde{L}_n,\tilde{l}_n\rangle$ & $2^{9}3^{13}$ & $=\ =\ =$ & $0.054\ (0.050\pm0.004)$ & $0.383$\\
 & $449/63$ & Model~3 & $\langle\widetilde{L}_n,\tilde{y}_n\tilde{l}_n\rangle$ & $2^{9}3^{13}$ & $=\ =\ =$ & $0.050\ (0.054\pm0.004)$ & $0.261$\\
\midrule
$-\frac{1}{4}z^3+\frac{3}{2}z+2$ & $-60$ & Model~6 & $\widetilde{M}_n$ & $2^{10}3^{13}$ & $=\ =\ =$ & $0.066\ (0.070\pm0.004)$ & $0.383$\\
 & $5/2$ & Model~5 & $\langle\widetilde{L}_n,\tilde{y}_n\rangle$ & $2^{9}3^{13}$ & $=\ =\ =$ & $0.054\ (0.050\pm0.004)$ & $0.388$\\
 & $-3$ & Model~4 & $\langle\widetilde{L}_n,\tilde{l}_n\rangle$ & $2^{9}3^{13}$ & $=\ =\ =$ & $0.051\ (0.050\pm0.004)$ & $0.384$\\
 & $1/2$ & Model~3 & $\langle\widetilde{L}_n,\tilde{y}_n\tilde{l}_n\rangle$ & $2^{9}3^{13}$ & $=\ =\ =$ & $0.056\ (0.054\pm0.004)$ & $0.264$\\
 & $41/12$ & Model~2 & $\widetilde{L}_n$ & $2^{8}3^{13}$ & $=\ =\ =$ & $0.041\ (0.040\pm0.004)$ & $0.261$\\
\addlinespace
 & $917/5$ & Model~6 & $\widetilde{M}_n$ & $2^{10}3^{13}$ & $=\ =\ =$ & $0.071\ (0.070\pm0.004)$ & $0.389$\\
 & $212/75$ & Model~5 & $\langle\widetilde{L}_n,\tilde{y}_n\rangle$ & $2^{9}3^{13}$ & $=\ =\ =$ & $0.055\ (0.050\pm0.004)$ & $0.380$\\
 & $-227/15$ & Model~4 & $\langle\widetilde{L}_n,\tilde{l}_n\rangle$ & $2^{9}3^{13}$ & $=\ =\ =$ & $0.056\ (0.050\pm0.004)$ & $0.385$\\
 & $-241/24$ & Model~3 & $\langle\widetilde{L}_n,\tilde{y}_n\tilde{l}_n\rangle$ & $2^{9}3^{13}$ & $=\ =\ =$ & $0.060\ (0.054\pm0.004)$ & $0.266$\\
 & $4793/276$ & Model~2 & $\widetilde{L}_n$ & $2^{8}3^{13}$ & $=\ =\ =$ & $0.038\ (0.040\pm0.004)$ & $0.257$\\
\midrule
$-\frac{1}{4}z^3+\frac{3}{2}z+2$ & $4$ & Model~5 & $\langle\widetilde{L}_n,\tilde{y}_n\rangle$ & $2^{9}3^{13}$ & $=\ 36^{\ast}\ 36^{\ast}$ & $0.681\ (0.050\pm0.004)$ & $0.733$\\
$-z^3+\frac{3}{2}z^2-1$ & $-217/96$ & Model~2 & $\widetilde{L}_n$ & $2^{8}3^{13}$ & $=\ 4^{\ast}\ 4^{\ast}$ & $0.259\ (0.040\pm0.004)$ & $0.321$\\
\bottomrule
\end{tabular}
\end{table}

\begin{table}[!htbp]
\centering
\small
\renewcommand{\arraystretch}{1.15}
\caption{Co-critical polynomials of length two, $6000$ primes for each pair, $n\leq 5$. Checks: violations of Proposition~\ref{prop:cocritical-types} among all the pairs (factor, child) at levels $k\leq 4$. Transitions: number of passed transition tests (levels $k=0,\dots,4$ and types of the parent) for the refined and for the original model. Frequencies: values of $n\leq 5$ for which the frequency test fails.}
\label{tab:cocrit-tests}
\setlength{\tabcolsep}{4pt}
\begin{tabular}{@{}llrcccc@{}}
\toprule
 & & & \multicolumn{2}{c}{transitions passed} & \multicolumn{2}{c}{frequency test fails for}\\
\cmidrule(lr){4-5}\cmidrule(l){6-7}
$f$ & $t$ & violations / checks & refined & original & refined & original\\
\midrule
$-z^3+\frac{3}{2}z^2-1$ & $-60$ & $0\ /\ 254\,933$ & $20/20$ & $0/20$ & -- & $n=3,4,5$\\
 & $7/2$ & $0\ /\ 249\,051$ & $18/18$ & $0/18$ & -- & $n=3,4$\\
 & $0$ & $0\ /\ 250\,845$ & $18/18$ & $0/18$ & -- & $n=3,4,5$\\
 & $-1/3$ & $0\ /\ 227\,578$ & $16/16$ & $0/16$ & -- & --\\
 & $-7/6$ & $0\ /\ 218\,800$ & $15/15$ & $0/15$ & -- & --\\
\addlinespace
 & $-442/9$ & $0\ /\ 254\,339$ & $20/20$ & $0/20$ & -- & $n=3,5$\\
 & $231/19$ & $0\ /\ 250\,220$ & $18/18$ & $0/18$ & -- & $n=3,4,5$\\
 & $-381/40$ & $0\ /\ 251\,331$ & $18/18$ & $0/18$ & -- & $n=3,4,5$\\
 & $449/63$ & $0\ /\ 230\,662$ & $16/16$ & $0/16$ & -- & $n=3$\\
\midrule
$-\frac{1}{4}z^3+\frac{3}{2}z+2$ & $-60$ & $0\ /\ 255\,095$ & $20/20$ & $0/20$ & -- & $n=3,4,5$\\
 & $5/2$ & $0\ /\ 250\,437$ & $18/18$ & $0/18$ & -- & $n=3,4,5$\\
 & $-3$ & $0\ /\ 252\,858$ & $18/18$ & $0/18$ & -- & $n=3,4$\\
 & $1/2$ & $0\ /\ 229\,238$ & $16/16$ & $0/16$ & -- & $n=3,4$\\
 & $41/12$ & $0\ /\ 220\,798$ & $15/15$ & $0/15$ & -- & --\\
\addlinespace
 & $917/5$ & $0\ /\ 254\,373$ & $20/20$ & $0/20$ & -- & $n=3,4,5$\\
 & $212/75$ & $0\ /\ 251\,954$ & $18/18$ & $0/18$ & -- & $n=3,4,5$\\
 & $-227/15$ & $0\ /\ 249\,534$ & $18/18$ & $0/18$ & -- & $n=3,4$\\
 & $-241/24$ & $0\ /\ 230\,484$ & $16/16$ & $0/16$ & -- & $n=5$\\
 & $4793/276$ & $0\ /\ 216\,278$ & $15/15$ & $0/15$ & -- & $n=4$\\
\midrule
$-\frac{1}{4}z^3+\frac{3}{2}z+2$ & $4$ & $0\ /\ 272\,445$ & $13/16$ & $0/16$ & $n=2,\dots,5$ & $n=2,\dots,5$\\
$-z^3+\frac{3}{2}z^2-1$ & $-217/96$ & $0\ /\ 205\,376$ & $11/12$ & $0/12$ & $n=2,\dots,5$ & $n=2,\dots,5$\\
\bottomrule
\end{tabular}
\end{table}

\begin{table}[!htbp]
\centering
\small
\renewcommand{\arraystretch}{1.15}
\caption{Remark~\ref{rem:refined-evidence}: all primes $p\equiv1\pmod3$ below $4\cdot10^5$ of good reduction, $n=3$. Distances $d_{TV}$ between the observed frequencies and those of $\widetilde{M}_{3,t}$ and of the group $M_{3,t}$ of Section~\ref{4} (in brackets: distance expected from sampling, $\pm 0.002$ in all cases). Last column: number of primes whose degree pattern is impossible in $\widetilde{M}_{3,t}$.}
\label{tab:evidence}
\setlength{\tabcolsep}{4pt}
\begin{tabular}{@{}llrccccr@{}}
\toprule
 & & & \multicolumn{2}{c}{factorisation trees} & \multicolumn{2}{c}{degree patterns} & \\
\cmidrule(lr){4-5}\cmidrule(lr){6-7}
$f$ & $t$ & primes & $\widetilde{M}_{3,t}$ & $M_{3,t}$ & $\widetilde{M}_{3,t}$ & $M_{3,t}$ & impossible\\
\midrule
$-z^3+\frac{3}{2}z^2-1$ & $2/7$ & $16\,891$ & $0.042\ (0.041)$ & $0.383\ (0.060)$ & $0.040\ (0.039)$ & $0.308\ (0.050)$ & $0$\\
 & $5$ & $16\,891$ & $0.041\ (0.041)$ & $0.380\ (0.060)$ & $0.038\ (0.039)$ & $0.305\ (0.050)$ & $0$\\
$-\frac{1}{4}z^3+\frac{3}{2}z+2$ & $1/7$ & $16\,891$ & $0.042\ (0.041)$ & $0.384\ (0.060)$ & $0.040\ (0.039)$ & $0.308\ (0.050)$ & $0$\\
\midrule
$2z^3-3z^2+\frac{1}{2}$ (control) & $5$ & $16\,892$ & $0.381\ (0.041)$ & $0.059\ (0.060)$ & $0.307\ (0.039)$ & $0.049\ (0.050)$ & $4\,880$\\
\bottomrule
\end{tabular}
\end{table}

\FloatBarrier
\subsection{The polynomial of length three}\label{app:tables-length-three}
Table~\ref{tab:three} contains the groups, the certificate and the tests (a) and (b) for $f=-2z^3+3z^2+\frac{1}{2}$, and Table~\ref{tab:three-patterns} the comparison of the factorisation patterns for $n\leq 4$ with the model of $f_t$ (Subsections~\ref{app:groups} and \ref{app:tests}).

For the exceptional value $t=\frac{3}{4}$ (Subsection~\ref{app:exceptions}), the $16$ transition tests at the levels $k=1,\dots,4$ are passed, and the two tests at level $0$ fail. Six transitions from level $0$ never occur, although the model expects each of them between $186$ and $378$ times. The other four transitions with three children of the same degree occur about four times as often as predicted, and at $n=3$ the distance $d_{TV}$ to the model is $0.115$, whereas sampling alone explains $0.043\pm0.004$.

\begin{table}[!htbp]
\centering
\small
\renewcommand{\arraystretch}{1.15}
\caption{The polynomial $-2z^3+3z^2+\frac{1}{2}$, $6000$ primes for each $t$: the group $M_{3,t}$ of the model of $f_t$ and its order; $=$: $\mathrm{Gal}(f_t^n/K)=M_{n,t}$ proved by Lemma~\ref{lem:certificate} for $n=1,2,3$, and $4^{\ast}$: estimated index of the Galois group in $M_{3,t}$; tests (a) for $n=1,\dots,5$ (the smaller of the two $p$-values) and tests (b) for the levels $k=0,\dots,4$ and the types of the parent that are passed, with the model of Definition~\ref{def:three-model}. $^{\ast}$: exceptional value (Table~\ref{tab:exceptions}).}
\label{tab:three}
\begin{tabular}{@{}lllcccc@{}}
\toprule
$t$ & model & group & $|M_{3,t}|$ & $n=1,2,3$ & (a) & (b)\\
\midrule
$-60$ & Model~6 & $M_n$ & $2^{12}3^{13}$ & $=\ =\ =$ & $5/5$ & $20/20$\\
$2$ & Model~6 & $M_n$ & $2^{12}3^{13}$ & $=\ =\ =$ & $5/5$ & $20/20$\\
$5$ & Model~6 & $M_n$ & $2^{12}3^{13}$ & $=\ =\ =$ & $5/5$ & $20/20$\\
$1/6$ & Model~3 & $\langle L_n,r_n\rangle$ & $2^{11}3^{13}$ & $=\ =\ =$ & $5/5$ & $18/18$\\
$3/8$ & Model~3 & $\langle L_n,r_n\rangle$ & $2^{11}3^{13}$ & $=\ =\ =$ & $5/5$ & $18/18$\\
$3/4^{\ast}$ & Model~3 & $\langle L_n,r_n\rangle$ & $2^{11}3^{13}$ & $=\ =\ 4^{\ast}$ & $5/5$ & $16/18$\\
$1/3$ & Model~4 & $\langle L_n,p_n\rangle$ & $2^{11}3^{13}$ & $=\ =\ =$ & $5/5$ & $18/18$\\
$3/5$ & Model~5 & $\langle L_n,q_n\rangle$ & $2^{11}3^{13}$ & $=\ =\ =$ & $5/5$ & $18/18$\\
$23/48$ & Model~2 & $L_n$ & $2^{10}3^{13}$ & $=\ =\ =$ & $5/5$ & $17/17$\\
\bottomrule
\end{tabular}
\end{table}

\begin{table}[!htbp]
\centering
\small
\renewcommand{\arraystretch}{1.15}
\caption{The polynomial $-2z^3+3z^2+\frac{1}{2}$: factorisation patterns for $n\leq 4$, $6000$ primes for each $t$. Columns $n=3$ and $n=4$: total variation distance between the observed frequencies of the factorisation patterns of $f_t^n$ and those of the model of $f_t$ (in brackets: mean distance expected from sampling). $^{\ast}$: exceptional value (Table~\ref{tab:exceptions}).}
\label{tab:three-patterns}
\begin{tabular}{@{}lclll@{}}
\toprule
$t$ & model & root types & $n=3$ & $n=4$\\
\midrule
$-60$ & 6 & all four & $0.055\ (0.059)$ & $0.387\ (0.389)$\\
$2$ & 6 & all four & $0.057\ (0.059)$ & $0.387\ (0.390)$\\
$5$ & 6 & all four & $0.054\ (0.059)$ & $0.390\ (0.390)$\\
$1/6$ & 3 & $sss$, $snn$ & $0.041\ (0.044)$ & $0.277\ (0.284)$\\
$3/8$ & 3 & $sss$, $snn$ & $0.038\ (0.043)$ & $0.290\ (0.284)$\\
$3/4^{\ast}$ & 3 & $sss$, $snn$ & $0.115\ (0.043)$ & $0.290\ (0.285)$\\
$1/3$ & 4 & $sss$, $nsn$ & $0.055\ (0.051)$ & $0.356\ (0.357)$\\
$3/5$ & 5 & $sss$, $nns$ & $0.051\ (0.053)$ & $0.349\ (0.351)$\\
$23/48$ & 2 & $sss$ & $0.034\ (0.036)$ & $0.258\ (0.263)$\\
\bottomrule
\end{tabular}
\end{table}

\FloatBarrier

\clearpage

\bibliographystyle{plain} 
\bibliography{bib} 
\end{document}